\documentclass[leqno,11pt]{amsart}
\usepackage{amsmath}
\usepackage{amsfonts}
\usepackage{amssymb}
\usepackage{amsthm}
\usepackage{enumerate}
\usepackage{mathrsfs}
\usepackage{dsfont}
\usepackage{txfonts}
\usepackage{graphicx}
\usepackage[all]{xy}
\usepackage{pgf,tikz,pgfplots}
\usepackage{subfigure}
\usepackage{tikz-3dplot}
\usetikzlibrary{patterns,arrows}
\tikzset{
    partial ellipse/.style args={#1:#2:#3}{
        insert path={+ (#1:#3) arc (#1:#2:#3)}
    }
}

\usepackage[pdftex]{hyperref}
\usepackage{color}
\usepackage[shortlabels]{enumitem}

\newtheorem{theorem}{Theorem}[section]
\newtheorem{principle}[theorem]{Principle}
\newtheorem{lemma}[theorem]{Lemma}
\newtheorem{corollary}[theorem]{Corollary}

\newtheorem{proposition}[theorem]{Proposition}

\newtheorem{claim}[theorem]{Claim}

\newtheorem{convention}{Convention}[section]

\theoremstyle{definition}
\newtheorem{definition}[theorem]{Definition}
\newtheorem{notation}[theorem]{Notation}

\theoremstyle{remark}
\newtheorem{remark}[theorem]{Remark}

\newcommand{\dist}{\mathrm{dist}}

\newcommand{\eps}{\varepsilon}
\newcommand{\Hh}{\mathcal{H}}
\newcommand{\reg}{\operatorname{reg}}
\newcommand{\sing}{\operatorname{sing}}
\newcommand{\spt}{\operatorname{spt}}
  \newcommand{\Div}{\operatorname{Div}}
   \newcommand{\Cc}{\mathcal C}

\newcommand{\RR}{\mathbb{R}}   % added (bw, sep 16)
\newcommand{\Dd}{\mathcal{D}}   %  "
\newcommand{\Mm}{\mathcal{M}}  % "

\begin{document}

\title{Ancient asymptotically cylindrical flows and applications}

\author{Kyeongsu Choi, Robert Haslhofer, Or Hershkovits, Brian White}

\begin{abstract}
In this paper, we prove the mean-convex neighborhood conjecture for neck singularities of the mean curvature flow in $\mathbb{R}^{n+1}$ for all $n\geq 3$: we show that if a mean curvature flow $\{M_t\}$ in $\mathbb{R}^{n+1}$ has an $S^{n-1}\times \mathbb{R}$ singularity at $(x_0,t_0)$, then there exists an $\eps=\eps(x_0,t_0)>0$ such that $M_t\cap B(x_0,\eps)$ is mean-convex for all  $t\in(t_0-\eps^2,t_0+\eps^2)$.  As in the case $n=2$, which was resolved by the first three authors in \cite{CHH}, the existence of such a mean-convex neighborhood follows from classifying a certain class of ancient Brakke flows that arise as potential blowup limits near a neck singularity. Specifically, we prove that any ancient unit-regular integral Brakke flow with a cylindrical 
blowdown must be either a round shrinking cylinder, a translating bowl soliton, or an ancient oval.
In particular, combined with a prior result of the last two authors \cite{HershkovitsWhite}, we obtain uniqueness of mean curvature flow through neck singularities.

The main difficulty in addressing the higher dimensional case is in promoting the spectral analysis on the cylinder to global geometric properties of the solution. Most crucially, due to the potential wide variety of self-shrinking flows with entropy lower than the cylinder when $n\geq 3$, smoothness does not follow from the spectral analysis by soft arguments. This precludes the use of the classical moving plane method to derive symmetry. To overcome this, we introduce a novel variant of the moving plane method, which we call ``moving plane method without assuming smoothness'' - where smoothness and symmetry are established in tandem. 
\end{abstract}

\maketitle

\tableofcontents

\section{Introduction}

The mean curvature flow is perhaps the most natural evolution equation for hypersurfaces in $\mathbb{R}^{n+1}$.  Given a smooth hypersurface $M_0$, its evolution $\{M_t\}_{t\geq 0}$ is dictated by the equation
\begin{equation}\label{MCF_eq}
\partial_t x = {\bf{H}}(x),
\end{equation}
where ${\bf{H}}(x)$ denotes the mean curvature vector at $x\in M_t$. Because the equation is parabolic,
 one expects the solution to improve with  time. Once this ``improvement'' is understood well enough, one expects that the mean curvature flow will become a central tool in the study of the geometry and the topology of embedded hypersurfaces. Some successes of this methodology include  \cite{HuiskenSinestrari_surgery,IW,BHH,Schulze_isoperim,BW_topology,HaslhoferKetover,BW_isotopy}. \\

From a different perspective, the fact that solutions of the mean curvature flow equation improve should mean that  if a solution has existed for infinitely long time, it should be quite rigid.  

\begin{definition}[ancient]
A  mean curvature flow $\{M_t\}$ is called \textit{ancient}, if it defined for all $t\in (-\infty,T)$, where $T\in (-\infty,\infty]$.
\end{definition}

Ancient solutions of the mean curvature flow, as well as other parabolic equations, have been extensively studied over the last 30 years. In particular, all singularity models (blowup limits) are ancient solutions, and thus the analysis of ancient solutions is crucial to understand the formation of singularities.\\

The simplest and most important kind of ancient solutions are the \emph{self-similarly shrinking} ones: these are solutions to the mean curvature flow that evolve by homotheties: 
\begin{equation}\label{ss_sol_eq}
M_t=\sqrt{-t}M_{-1}.
\end{equation}
One easily sees that $M_{-1}$ is a time $-1$ slice of a self-similarly shrinking solution if and only if it satisfies
\begin{equation}\label{ss_eq}
{\bf H}(x)+\frac{x^{\perp}}{2}=0,
\end{equation}
in which case $M_{-1}$ is called a \textit{shrinker}.

 In \cite{Huisken_shrinker}, Huisken classified all smooth, mean-convex shrinkers in $\RR^{n+1}$: 
 each such shrinker is
 (up to a rotation) either a hyperplane through the origin, a round cylinder of the form
 \begin{equation}\label{shrinking_cylinders}
 S^k\left(\sqrt{2k}\right)\times \mathbb{R}^{n-k},
 \end{equation}
 where $k\in \{1,\ldots,n-1\}$, or the sphere $S^{n}\left(\sqrt{2n}\right)$.
 More recently, without any curvature assumptions, Brendle showed that 
 every smooth, two-dimensional shrinker of genus $0$ in $\RR^3$ is either the round sphere,
  the round cylinder or a flat plane \cite{Brendle_genus_zero}. Those are by no means the only smooth,
  two-dimensional shrinkers in $\RR^3$, as examples constructed in \cite{Angenent_torus,KKM,Ketover_shrink,Edelen_White} indicate.
\\

Another important type of ancient (indeed, eternal) flows are \emph{translating solutions}, i.e., solutions of the form
\begin{equation}
M_t=M_0+tv,
\end{equation}
for some fixed vector $v\in \mathbb{R}^{n+1}$. 
For $n\ge 2$, there is a unique translator that is the graph of an entire, 
rotationally invariant function on $\RR^n$~\cite{AltschulerWu,CSS}.  It is called the {\em bowl soliton},
 a name suggestive of its leading order paraboloidal shape.
In \cite{Wang_convex}, Wang proved 
that (up to rigid motions) the bowl soliton
is the unique convex translator in $\RR^3$ that is an entire graph.
Wang also constructed convex, entire graphical translators in higher dimensions that are not rotationally symmetric. 
In arbitrary dimension, the second author  \cite{Haslhofer_bowl} proved that the bowl soliton
is the only uniformly two-convex, 
noncollapsed\footnote{We recall that a mean-convex flow is \emph{$\alpha$-noncollapsed} \cite{ShengWang,Andrews_noncollapsing,HaslhoferKleiner_meanconvex} if at each point $p\in M_t$ admits interior and exterior balls tangent at $p$ of radius $\alpha/H(p)$. Noncollapsed solutions are the most important ones for singularity analysis.} 
translating solution in $\RR^{n+1}$.
 A complete classification of graphical translators in $\mathbb{R}^3$, both collapsed and noncollapsed, has been obtained in \cite{HIMW}, building on important prior work of Spruck-Xiao \cite{SpruckXiao}.
See~\cite{HMW},~\cite{nguyen-trident}, and~\cite{HMW_tridents} for other examples of translators,
and~\cite{translators_survey} for a survey article about translators.

\begin{figure}
\subfigure{
\begin{tikzpicture}[x=1cm,y=1cm]
\clip(-8,-2.2) rectangle (8,2);
\draw  [dashed] (0,0) [partial ellipse=0:180:0.8*1.9910199803423714cm and 0.8*0.8419484839039086cm];
\draw  (0,0) [partial ellipse=180:360:0.8*1.9910199803423714cm and 0.8*0.8419484839039086cm];
\draw  (0.8-0.8*4.020936148441656,0.8*1.5856378953801935)-- (0.8-0.8*7.999039871847921,0.8*1.6098208967382863);
\draw  (0.8-0.8*4.020936148441656,0.8*1.5856378953801935)-- (0.8-0.8*5.0003477034444135,-0.8*1.6186097845671044);
\draw  (0.8-0.8*5.0003477034444135,-0.8*1.6186097845671044)-- (0.8-0.8*8.978451426850679,-0.8*1.6186097845671044);
\draw  (0.8-0.8*7.999039871847921,0.8*1.6098208967382863)-- (0.8-0.8*8.978451426850679,-0.8*1.6186097845671044);
\draw  (0,0) circle (0.8*2cm);
\draw  [dashed] (-1+0.8*5,0) [partial ellipse=-90:90:0.8*0.6348995602519808cm and 0.8*1.1845241456416877cm];
\draw  (-1+0.8*5,0) [partial ellipse=90:270:0.8*0.6348995602519808cm and 0.8*1.1845241456416877cm];
\draw  (-1+0.8*9,0) ellipse (0.8*0.6471357575224044cm and 0.8*1.1911274863187848cm);
\draw (-1+0.8*5.011728784659188,0.8*1.1843220081126127)-- (-1+0.8*9.0035002791945,0.8*1.1911100623932829);
\draw  (-1+0.8*4.986047288841359,-0.8*1.1842380748465917)-- (-1+0.8*8.986643760518538,-0.8*1.1908737678350958);
\begin{scriptsize}
\draw  (-4.8,-2) node {\large static plane};
\draw  (0,-2) node {\large round shrinking sphere};
\draw  (4.8,-2) node {\large round shrinking cylinder};
\end{scriptsize}
\end{tikzpicture}
}
\subfigure{
\begin{tikzpicture}[x=1cm,y=1cm]
\clip(-8,-0.4) rectangle (5,4.5);
\draw [samples=100,rotate around={0:(-4,0.2)},xshift=-4cm,yshift=0.2cm,domain=-4:4)] plot (\x,{(\x)^2/2/0.4072952864749447});
\draw [dashed] (-4,3.5) [partial ellipse=0:180:1.6365985589255831cm and 0.3580894575064108cm];
\draw (-4,3.5) [partial ellipse=180:360:1.6365985589255831cm and 0.3580894575064108cm];
\draw [dashed] (-4,2.5) [partial ellipse=0:180:1.364975075938888cm and 0.33800886942172814cm];
\draw (-4,2.5) [partial ellipse=180:360:1.364975075938888cm and 0.33800886942172814cm];
\draw (1.5,2) ellipse (2.500324874551538cm and 0.6526621872547251cm);
\draw [dashed] (0,2) [partial ellipse=-90:90:0.16492359520392494cm and 0.5264976659539734cm];
\draw (0,2) [partial ellipse=90:270:0.16492359520392494cm and 0.5264976659539734cm];
\draw [dashed] (1.4814310289935502,2.010750356118509) [partial ellipse=-90:90:0.27793025604476596cm and 0.6455817210648582cm];
\draw (1.4814310289935502,2.010750356118509) [partial ellipse=90:270:0.27793025604476596cm and 0.6455817210648582cm];
\draw [dashed] (3,2) [partial ellipse=-90:90:0.18672903265872898cm and 0.5337300175535058cm];
\draw (3,2) [partial ellipse=90:270:0.18672903265872898cm and 0.5337300175535058cm];
\begin{scriptsize}
\draw  (-4,-0.2) node {\large translating bowl};
\draw  (1.2,-0.2) node {\large ancient oval};
\end{scriptsize}
\end{tikzpicture}
}
\caption{The classification by Angenent-Daskalopoulos-Sesum and Brendle-Choi}\label{figure_ADS_BC}
\end{figure}
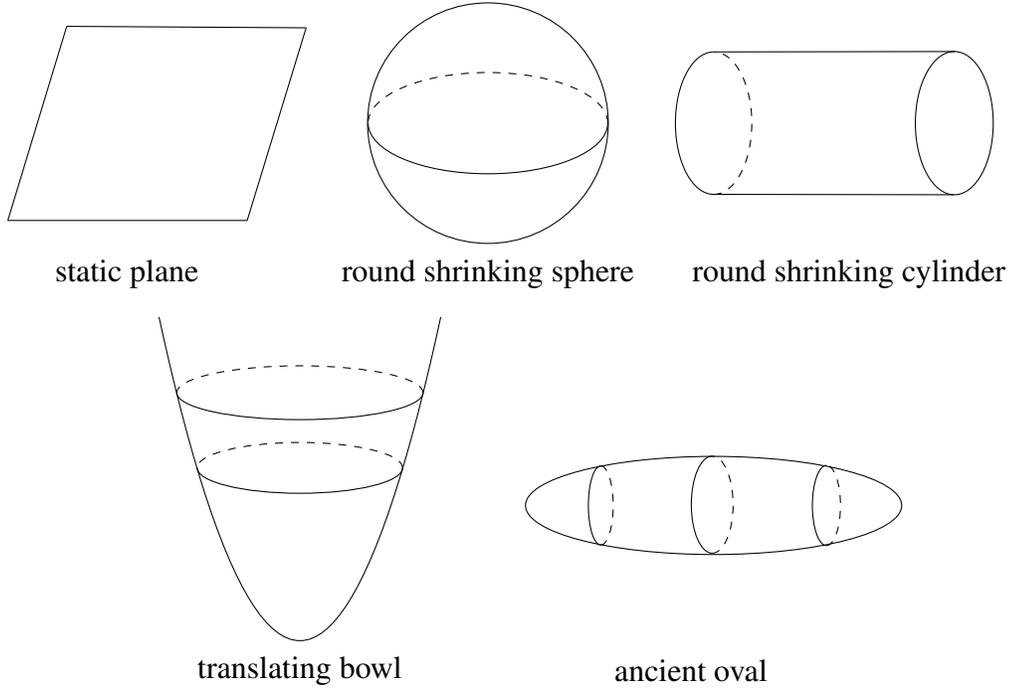

Ancient solutions that are not self-similar are more challenging to construct and are even harder to classify. This is mostly due to the fact that they do not satisfy any elliptic equation or variational principle. In dimension one, a convex example resembling a \emph{paperclip} was found by Angenent, and analogues collapsed convex ancient solutions in $\mathbb{R}^{n+1}$, so called \emph{ancient pancakes}, were constructed in \cite{Wang_convex,BLT}. In \cite{White_nature}, the last author (see also \cite{Angenent_oval} and \cite{HaslhoferHershkovits_ancient}) gave an example of a noncollapsed ancient, compact, uniformly two-convex mean curvature flow in $\mathbb{R}^{n+1}$ which for $t\to -\infty$ looks like a cylinder, capped off by two bowl solitons, and for $t\to 0$ becomes round. This solution is called an \emph{ancient oval}.  In \cite{DHS_ancient}, Daskalopoulos, Hamilton and Sesum showed that closed embedded ancient evolutions of curves are either a family of round shrinking circles or Angenent's paper-clip solution. A key a feature, which is unique to the one dimensional case, is that the paper-clip solution is given by an explicit formula. This allowed them to design monotone quantities that identify the paper-clip. The oval, on the other hand, (and also, the bowl soliton)  is not given by any explicit formula. \\

The classification of ancient solutions of the mean curvature flow, in particular in the most challenging situation without any self-similarity assumptions, has enjoyed recent significant developments by Angenent-Daskalopoulos-Sesum and by Brendle-Choi. The combined results of their papers provide a  complete classification in the noncollapsed uniformly two-convex setting:\footnote{We recall that a mean-convex mean curvature flow is called \emph{uniformly two-convex}, if $\lambda_1+\lambda_2\geq \beta H$ for some $\beta>0$, where $\lambda_1,\lambda_2$ denotes the smallest two principal curvatures. This condition is preserved under mean curvature flow.}

\begin{theorem}[{classification of ancient noncollapsed uniformly two-convex smooth mean curvature flows by Angenent-Daskalopoulos-Sesum \cite{ADS,ADS2} and Brendle-Choi \cite{BC,BC2}}]\label{thm_class_conv} Any ancient, noncollapsed, uniformly two-convex, smooth mean curvature flow in $\mathbb{R}^{n+1}$ is either
\begin{enumerate}
\item a static hyperplane, 
\item a family of round shrinking spheres,
\item a family of round shrinking cylinders,
\item a translating bowl soliton, or 
\item an ancient oval.
\end{enumerate}
\end{theorem}

These results answer fundamental questions regarding the precise nature of singularities and high curvature regions of mean curvature flow starting from any closed two-convex hypersurface, c.f. \cite[Conj. 1]{White_nature}. Moreover, the methods also turned out to be fundamental for the classification of high curvature regions in three-dimensional Ricci flow \cite{Brendle_ricci1,ABDS_ricci,BDS_ricci}, as conjectured by Perelman \cite{Perelman1,Perelman2}.
On the other hand, even before these results, mean curvature flow with two-convex initial conditions had reached almost full maturity, from the perspective of flowing through singularities, and geometric and topological applications (as long as one applies it to one object at a time). In  fact, a coarser description of the behavior of such flows at high curvature regions turned out to be sufficient for the construction of mean curvature flow with surgery by Huisken-Sinestrari \cite{HuiskenSinestrari_surgery},  Brendle-Huisken \cite{BH_surgery} and Haslhofer-Kleiner \cite{HaslhoferKleiner_surgery}. It was therefore clear that more general classification results are needed to gain new insights about the qualitative local behavior of mean curvature flow without two-convexity assumption. \\

For $n=2$, such a generalization has been obtained recently by the first three authors \cite[Thm. 1.2]{CHH}, who showed that any \emph{ancient low entropy flow} in $\mathbb{R}^3$ (as introduced in \cite[Def. 1.1]{CHH} and reviewed below in Section \ref{ent_sec}), must be one of the types (i)--(v) from above. In particular, this classification of ancient low entropy flows in $\mathbb{R}^3$ was the key to confirm two fundamental conjectures for the mean curvature flow in $\mathbb{R}^3$: The mean-convex neighborhood conjecture \cite[Thm. 1.7]{CHH}, and the uniqueness conjecture for mean curvature flow through cylindrical singularities \cite[Thm. 1.9]{CHH}.\\

The goal of the present paper is to prove a classification result for ancient flows in $\mathbb{R}^{n+1}$ for $n\geq 3$, that is general enough to facilitate conclusions about mean-convex neighborhoods and uniqueness.\\

As we will explain in Section \ref{subsec_asymptcyl}, a suitable class of flows to consider for this purpose is the one of \emph{ancient asymptotically cylindrical flows} (see Definition \ref{def_ancient_flow}). Loosely speaking, this is the class of all ancient Brakke flows that one potentially gets when one blows up near any neck singularity (see Definition \ref{def_neck_sing}) by the scale of the neck . It turns out that for $n=2$ the class of ancient low entropy flows, i.e. the class of flows considered in \cite{CHH}, is essentially equivalent (after eliminating the trivial examples of static planes and round shrinking spheres) to the class of ancient asymptotically cylindrical flows (see Section \ref{ent_sec}). In stark contrast, for $n\geq 4$ these classes of flows are most likely dramatically different (see Section \ref{sec_intro_roleofentropy}).\\

Our main classification result (Theorem \ref{thm_classification_asympt_cyl}) proves that any ancient asymptotically cylindrical flow in $\mathbb{R}^{n+1}$, where $n\geq 3$ is arbitrary, is either (i) a family of round shrinking cylinders, (ii) a translating bowl soliton, or (iii) an ancient oval. Our main applications are a proof of the mean-convex neighborhood conjecture for neck singularities in arbitrary dimension (Theorem \ref{thm_mean_convex_nbd_intro}), and a proof of the uniqueness conjecture for mean curvature flow through neck singularities in arbitrary dimension (Theorem \ref{thm_nonfattening_intro}). Combined with a recent result by Colding-Minicozzi \cite{CM_complexity}, we also obtain a classification result in higher codimension (Corollary \ref{cor_highercodim}). These three applications will be discussed in Section \ref{intro_application}.\\

As we shall see, there are many obstacles that only present themselves for $n\geq 3$. To overcome them, this paper contains several new ideas that are very different in nature from the ones in \cite{CHH}. One of these ideas, which we call ``moving plane method without assuming smoothness'' (see Section \ref{intro_moving_planes}) is seemingly novel to geometric analysis, and we hope it will find many future applications.\\

\bigskip

\subsection{Ancient asymptotically cylindrical flows}\label{subsec_asymptcyl}
Our main result described in this section is a classification of ancient asymptotically cylindrical flows in $\mathbb{R}^{n+1}$ for $n\geq 3$, see Theorem \ref{thm_classification_asympt_cyl}. 
To get a classification strong enough for our applications (see Section \ref{intro_application}), we have to extend the class of smooth mean curvature flows to a class with better compactness properties: 

As in \cite[Def. 6.2, 6.3]{Ilmanen_book} an $n$-dimensional \emph{integral Brakke flow} in $\mathbb{R}^{n+1}$ is given by a family of Radon measures $\mathcal M = \{\mu_t\}_{t\in I}$ in $\mathbb{R}^{n+1}$ that is integer $n$-rectifiable for almost all times and satisfies
\begin{equation}\label{eq_brakke_flow}
\frac{d}{dt} \int \varphi \, d\mu_t \leq \int \left( -\varphi {\bf H}^2 + \nabla\varphi \cdot {\bf H} \right)\, d\mu_t
\end{equation}
for all test functions $\varphi\in C^1_c(\mathbb{R}^{n+1},\mathbb{R}_+)$, see Section \ref{sec_prelim} (preliminaries) for details. Of course, whenever $\{M_t\}_{t\in I}$ is a classical solution of \eqref{MCF_eq}, then the associated family of area measures $\mu_t=\mathcal{H}^n \llcorner M_t $ solves \eqref{eq_brakke_flow}. A somewhat silly  quirk, coming from the very definition of Brakke flows via the inequality \eqref{eq_brakke_flow}, is that the flow can suddenly vanish without any cause. To prevent this to some extent, we often assume that the flows are \emph{unit-regular} as defined in \cite{White_regularity,SchulzeWhite}, i.e.,  that every backwardly regular point is regular. All Brakke flows constructed via Ilmanen's elliptic regularization \cite{Ilmanen_book} are integral and unit-regular, and these properties are preserved under passing to weak limits, see Section \ref{sec_prelim} (preliminaries).\\

\begin{figure}
\begin{tikzpicture}[x=1cm,y=1cm]
\clip(-8,-2) rectangle (8,13);
\draw  (0,0) ellipse (7.848566748139434cm and 1.5512092057488571cm);
\draw  (0,4) ellipse (6cm and 1.3606721028332194cm);
\draw (0,8-0.7) ellipse (4cm and 1.0248201843525584cm);
\draw  (0,10) ellipse (1.4896173792050247cm and 0.8825870701690894cm);
\draw  (0,12.127925031044105) circle (0.5336805109630947cm);
\draw [dashed] (0,12.13) [partial ellipse=-90:90:0.23028952525624294cm and 0.5233863443409839cm];
\draw (0,12.13) [partial ellipse=90:270:0.23028952525624294cm and 0.5233863443409839cm];
\draw [dashed] (0,10) [partial ellipse=-90:90:0.5177983645429698cm and 0.8706980798895485cm];
\draw (0,10) [partial ellipse=90:270:0.5177983645429698cm and 0.8706980798895485cm];
\draw [dashed] (2,8-0.7) [partial ellipse=-90:90:0.3756702608059214cm and 0.8838145421149775cm];
\draw (2,8-0.7) [partial ellipse=90:270:0.3756702608059214cm and 0.8838145421149775cm];
\draw [dashed] (-2,8-0.7) [partial ellipse=-90:90:0.3756702608059214cm and 0.8838145421149775cm];
\draw (-2,8-0.7) [partial ellipse=90:270:0.3756702608059214cm and 0.8838145421149775cm];
\draw [dashed] (0,4) [partial ellipse=-90:90:0.610323936979811cm and 1.3462894592362127cm];
\draw (0,4) [partial ellipse=90:270:0.610323936979811cm and 1.3462894592362127cm];
\draw [dashed] (3,4) [partial ellipse=-90:90:0.41122606218909463cm and 1.1743538113462901cm];
\draw (3,4) [partial ellipse=90:270:0.41122606218909463cm and 1.1743538113462901cm];
\draw [dashed] (-3,4) [partial ellipse=-90:90:0.41122606218909463cm and 1.1743538113462901cm];
\draw (-3,4) [partial ellipse=90:270:0.41122606218909463cm and 1.1743538113462901cm];
\draw [dashed] (2,0) [partial ellipse=-90:90:0.5755285472390149cm and 1.495203366999636cm];
\draw (2,0) [partial ellipse=90:270:0.5755285472390149cm and 1.495203366999636cm];
\draw [dashed] (-2,0) [partial ellipse=-90:90:0.5755285472390149cm and 1.495203366999636cm];
\draw (-2,0) [partial ellipse=90:270:0.5755285472390149cm and 1.495203366999636cm];
\draw [dashed] (5.5,0) [partial ellipse=-90:90:0.3335344998800339cm and 1.1017010767945346cm];
\draw (5.5,0) [partial ellipse=90:270:0.3335344998800339cm and 1.1017010767945346cm];
\draw [dashed] (-5.5,0) [partial ellipse=-90:90:0.3335344998800339cm and 1.1017010767945346cm];
\draw (-5.5,0) [partial ellipse=90:270:0.3335344998800339cm and 1.1017010767945346cm];
\end{tikzpicture}
\caption{The ancient ovals are an example of ancient asymptotically cylindrical flow.}\label{figure_oval}
\end{figure}
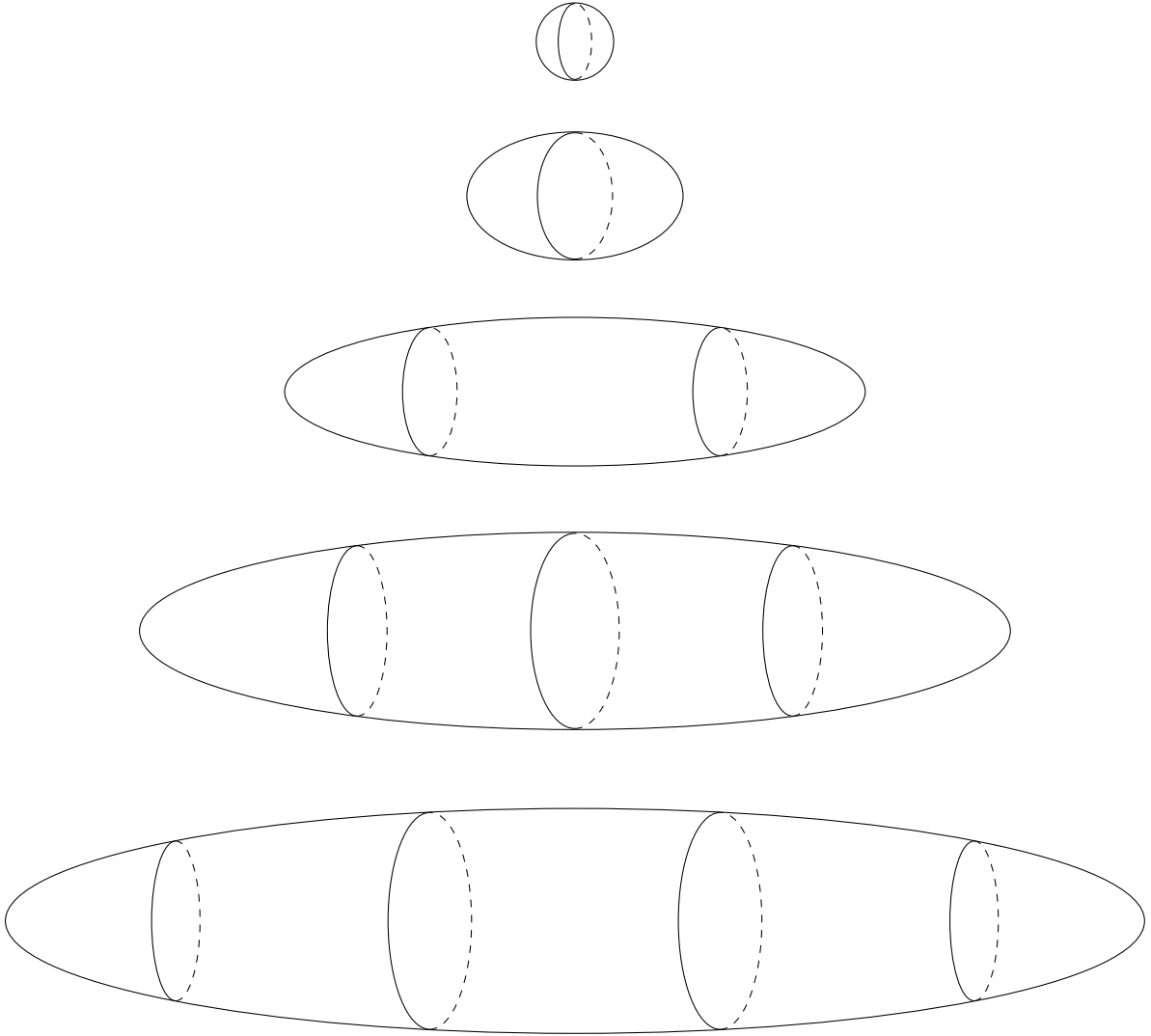

The coarse asymptotics of an ancient integral Brakke flow $\mathcal M = \{\mu_t\}_{t\in (-\infty,T_e(\mathcal{M})]}$, where $T_e(\mathcal{M})\leq \infty$ denotes the \emph{extinction time}, are captured by a so-called \emph{blowdown limit} (aka tangent flow at infinity). To describe this, given any 
$\lambda>0$, let $\mathcal{D}_{\lambda}: \mathbb{R}^{n+1}\times \mathbb{R}\rightarrow \mathbb{R}^{n+1}\times \mathbb{R}$ be the parabolic dilation
\begin{equation}
\mathcal{D}_{\lambda}(x,t)=(\lambda x,\lambda^2 t).
\end{equation}
We denote  by $\Dd_{\lambda}(\mathcal{M})$ the Brakke flow that is obtained 
from $\mathcal{M}$ by parabolically dilating by $\lambda$.\footnote{In more pedantic notation,
$\Dd_{\lambda}(\mathcal{M})=\{{\mu}^\lambda_t\}_{t\in (-\infty,\lambda^2 T)}$ where $\mu^\lambda_t(A)=\lambda^{n}\mu_{\lambda^{-2}t}(\lambda^{-1}A)$.}

\begin{definition}[blowdown limit]
A \emph{blowdown limit} of an ancient Brakke flow $\Mm$ is any limit of the form
\begin{equation}\label{blowdown_flow}
\check{\mathcal{M}}=\lim_{j\rightarrow 0} \mathcal{D}_{\lambda_j}(\mathcal{M}),
\end{equation}
where $\lambda_j$ is a sequenece of positive numbers converging to zero. 
\end{definition} 

In full generality, the limit in \eqref{blowdown_flow} has to be understood in the sense of Brakke flows. However, in the important special case when $\check{\mathcal{M}}$ is smooth with multiplicity one, then thanks to the local regularity theorem \cite{Brakke,White_regularity} the convergence is actually smooth. It follows from Huisken's monotonicity formula \cite{Huisken_monotonicity,Ilmanen_monotonicity} and the Brakke compactness theorem \cite{Brakke, Ilmanen_book} that any ancient integral Brakke flow with finite entropy has at least one blowdown limit, and furthermore that any such blowdown limit $\check{\mathcal{M}}$ is backwardly selfsimilar, i.e.,
\begin{equation}
\check{M}_t = \sqrt{-t}\check{M}_{-1}.
\end{equation} 
The case of interest for the analysis of neck singularities (see Definition \ref{def_neck_sing}) is when $\check{\mathcal{M}}$ (after a suitable orthogonal transformation of $\mathbb{R}^{n+1}$) is a family of round shrinking cylinders as in \eqref{shrinking_cylinders} with $k=n-1$:

\begin{definition}[ancient asymptotically cylindrical flow]\label{def_ancient_flow}
An \emph{ancient asymptotically cylindrical flow} is an ancient, unit-regular,
 integral Brakke flow 
 %$\mathcal{M}$=\{\mu_t\}_{t\in (-\infty,T_e(\mathcal{M})]}$
  in $\mathbb{R}^{n+1}$ that
has some  blowdown limit 
%for $t\to -\infty$ 
consisting (up to a rotation) of the
 round shrinking cylinders $\{S^{n-1}(\sqrt{-2(n-1)t})\times \mathbb{R}\}_{t<0}$. 
\end{definition}

Our main theorem provides a complete classification:

\begin{theorem}[classification of ancient asymptotically cylindrical flows]\label{thm_classification_asympt_cyl}
For every $n\geq 3$, any ancient asymptotically cylindrical flow in $\mathbb{R}^{n+1}$ is either
\begin{enumerate}
\item a round shrinking cylinder,
\item a translating bowl soliton, or 
\item an ancient oval.
\end{enumerate}
\end{theorem}

The most important feature of Theorem \ref{thm_classification_asympt_cyl}, in stark contrast with the prior classification result (see Theorem \ref{thm_class_conv}) from Angenent-Daskalopoulos-Sesum \cite{ADS,ADS2} and Brendle-Choi \cite{BC,BC2}, is that convexity, uniform two-convexity, noncollapsing, connectedness, smoothness, and curvature bounds are implied by the theorem, rather than being its assumptions. This is crucial for the proof of the mean-convex neighborhood conjecture for neck singularities (Theorem \ref{thm_mean_convex_nbd_intro}) and the proof of the uniqueness conjecture for mean curvature flow through neck singularities (Theorem \ref{thm_nonfattening_intro}).\\

\bigskip

\subsection{Relationship with prior classification results}\label{ent_sec}
In this section, we explain the relationship with the prior classification results from \cite{ADS,ADS2,BC,BC2} and \cite{CHH}.

\subsubsection{Relationship with the classification of ancient noncollapsed uniformly two-convex smooth mean curvature flows.}
If $\{M_t\}_{t\in (-\infty,T)}$ is an ancient noncollapsed uniformly two-convex smooth mean curvature flow in $\mathbb{R}^{n+1}$, then it easily follows from Huisken's monotonicity formula \cite{Huisken_monotonicity} and his classification of mean-convex shrinkers (see the beginning of this introduction), combined with the two-convexity assumption, that for $t\to -\infty$ we can take a blowdown limit which must be either (a) a static plane, (b) a family of round shrinking spheres with radius $\sqrt{-2nt}$, or (c) a family of round shrinking cylinders of the form
\begin{equation}
\left\{S^{n-1}(\sqrt{-2(n-1)t})\times\mathbb{R}\right\}_{t<0}.
\end{equation}
In case (a) and (b), by the equality case of the monotonicity formula, the flow $\{M_t\}_{t\in (-\infty,T)}$ itself must be a static plane or a family of round shrinking spheres, respectively. Hence, the only nontrivial case is (c), and we see that Theorem \ref{thm_classification_asympt_cyl} (classification of ancient asymptotically cylindrical flows) of course generalizes Theorem \ref{thm_class_conv} (classification of ancient noncollapsed uniformly two-convex smooth mean curvature flows).\\

A-posteriori, as a consequence of Theorem \ref{thm_classification_asympt_cyl} we obtain:

\begin{corollary}[consequence of our main classification result]
Every ancient asymptotically cylindrical flow is convex, uniformly two-convex, noncollapsed, and smooth.
\end{corollary}

However, we emphasize that it is a-priori completely nonevident -- and in fact completely nonevident almost until the very end of our proof of Theorem \ref{thm_classification_asympt_cyl} -- whether or not ancient asymptotically cylindrical flows  are (mean) convex, uniformly two-convex, noncollapsed and smooth.\\

\subsubsection{Relationship with the classification of ancient low entropy flows.}
Next, let us explain the relationship with the classification result for ancient low entropy flows in $\mathbb{R}^3$ by the first three authors, which we restate here for the reader's convenience:\footnote{For $n=2$, since $\textrm{Ent}[S^1\times\mathbb{R}]>3/2$, one needs the extra technical condition of being \emph{cyclic}, as defined in \cite{White_Currents}. However, for reading the present paper one can safely ignore this notion, since fortunately $\textrm{Ent}[S^{n-1}\times\mathbb{R}]<3/2$ for $n\geq 3$.}

\begin{theorem}[{classification of ancient low entropy flows in $\mathbb{R}^3$ from \cite[Thm. 1.2]{CHH}}]\label{thm_class_low_ent}
Suppose that $\mathcal M$
 is a 
an ancient, unit-regular, cyclic, integral Brakke flow in $\mathbb{R}^3$
 that satisfies the low entropy assumption $\textrm{Ent}[\mathcal{M}]\leq \textrm{Ent}[S^1\times \mathbb{R}]$.
Then $\mathcal M$ 
is either (i) a static plane, (ii) a family of round shrinking spheres, (iii) a family of round shrinking cylinders, (iv) a translating bowl soliton, or (v) an ancient oval.
\end{theorem}

\emph{Entropy} was introduced by Colding-Minicozzi \cite{CM_generic}, and is defined as follows: The entropy of a Radon measure $\mu$ in $\mathbb{R}^{n+1}$ (in particular of a hypersurface $M$ via $\mu=\mathcal{H}^n\llcorner M$) is defined as the supremum of its Gaussian area over all centers and scales, namely
\begin{equation}
\textrm{Ent}[\mu]=\sup_{y\in\mathbb{R}^{n+1},\lambda>0} \frac{1}{(4\pi\lambda)^{n/2}}\int e^{-\tfrac{|x-y|^2}{4\lambda}}\, d\mu(x).
\end{equation}
The \emph{entropy} of a Brakke flow $\mathcal{M}=\{\mu_t\}_{t\in I}$ is then defined as
\begin{equation}
\textrm{Ent}[\mathcal{M}]=\sup_{t\in I} \textrm{Ent}[\mu_t].
\end{equation}

Using Huisken's monotonicity formula \cite{Huisken_monotonicity,Ilmanen_monotonicity} and the important classification result for low entropy shrinkers in dimension $n=2$
 by Bernstein-Wang~\cite{BW_topological_property}, one can easily check:

\begin{proposition}[Essential equivalence for $n=2$]
Let $\mathcal M$ be an ancient unit-regular, cyclic, integral Brakke flow in $\mathbb{R}^3$, 
and suppose that $\mathcal M$ is not a static plane or a family of round shrinking spheres. 
Then $\mathcal M$ is an ancient asymptotically cylindrical flow if and only if $\mathcal M$ is an ancient low entropy flow.
\end{proposition}

In that sense, Theorem \ref{thm_classification_asympt_cyl} (classification of ancient asymptotically cylindrical flows) generalizes to arbitrary $n\geq 3$, the prior classification of ancient low entropy flows in $\mathbb{R}^3$ (Theorem \ref{thm_class_low_ent}).\\

\bigskip

\subsection{The role of entropy in the study of ancient Brakke flows}\label{sec_intro_roleofentropy} In this section, we explain one of the key difficulties that arises in the study of ancient asymptotically cylindrical flows in $\mathbb{R}^{n+1}$ for $n\geq 3$.\\

Let $\mathcal M$ be an ancient asymptotically cylindrical flow in $\mathbb{R}^{n+1}$. It easily follows from Huisken's monotonicity formula \cite{Huisken_monotonicity,Ilmanen_monotonicity} that
\begin{equation}\label{intro_ent_bound}
\textrm{Ent}[\mathcal{M}]\leq \textrm{Ent}[S^{n-1}\times \mathbb{R}].
\end{equation}
For any $X=(x,t)$ in $\mathcal M$ one can take a limit
\begin{equation}
\hat{\mathcal{M}}_X=\lim_{j\rightarrow \infty} \mathcal{D}_{\lambda_{j}}(\mathcal{M}-X),   
\end{equation}
for some sequence $\lambda_j\to \infty$, and any such limit is backwardly selfsimilar. The flow $\hat{\mathcal{M}}_X$ is called a \emph{tangent flow} of $\mathcal{M}$ at $X$, and is a key object in analyzing the singularity formation of the flow. Importantly,
\begin{equation}
\mathrm{Ent}[\hat{\mathcal{M}}_X]\leq \mathrm{Ent}[\mathcal{M}],
\end{equation}
so combined with \eqref{intro_ent_bound}, in order to study partial regularity of ancient asymptotically cylindrical flows, one is lead to study the class of self-similar flows with entropy less than $\mathrm{Ent}[S^{n-1}\times \mathbb{R}]$.

While the role of selfsimilar flows in the singularity analysis is well known, they also play an additional role in the analysis of ancient flows. Indeed, as is explained in \cite{CHN} and Section \ref{sec_sim_back_in_time}, through \emph{quantitative differentiation} one can quantify the equality case of the monotonicity formula to obtain that ancient integral Brakke flows of bounded entropy must be almost selfsimilar away from a controlled number of scales. Ancient selfsimilar flows of entropy less than $\mathrm{Ent}[S^{n-1}\times \mathbb{R}]$ therefore provide invaluable information, not only about the singularity formation, but also on the ``history'' of ancient asymptotically cylindrical flows.\\

A crucial ingredient in \cite{CHH} was the following important result by Bernstein-Wang:

\begin{theorem}[{low entropy shrinkers for $n=2$ by Bernstein-Wang \cite{BW_topological_property}}]\label{thm_bernstein_wang}
The only nontrivial smooth two-dimensional shrinker $\Sigma\subset\mathbb{R}^3$ with $\textrm{Ent}[\Sigma]< \textrm{Ent}[S^{1}\times \mathbb{R}]$
is the round sphere.
\end{theorem}

In particular, this was used in \cite[Sec. 5.1]{CHH} to show that ancient low entropy flows in $\mathbb{R}^3$ are smooth until they become extinct. However, the Bernstein-Wang classification is specifically for $n=2$. Indeed, their classification relies on Brendle's classification of genus zero shrinkers in $\mathbb{R}^3$ \cite{Brendle_genus_zero}.\\      

For $n=3$, the best available structural result for low entropy shrinkers is the following:

\begin{theorem}[structure of low entropy shrinkers for $n=3$ by Bernstein-Wang \cite{BW_topology}]\label{thm_bernstein_wang}
If $\Sigma\subset\mathbb{R}^4$ is any three-dimensional shrinker with $\textrm{Ent}[\Sigma]< \textrm{Ent}[S^{2}\times \mathbb{R}]$ then it is either
\begin{itemize}
\item a compact shrinker diffeomorphic to $S^3$, or
\item a noncompact asymptotically conical shrinker diffeomorphic to $\mathbb{R}^3$.
\end{itemize}
\end{theorem}

In particular, it is a very difficult open problem to determine whether or not there exist nontrivial asymptotically conical shrinkers with entropy less than the cylinder. To circumvent this problem, for their recent proof of the low entropy Sch\"onflies conjecture in $\mathbb{R}^4$ \cite{BW_isotopy}, Bernstein-Wang had to write a slew of auxiliary papers \cite{BW_aux1,BW_aux2,BW_aux3,BW_aux4,BW_aux5,BW_aux6} to deal with the potential scenario of such asymptotically conical shrinkers.\\

For $n\geq 4$, it is believed (but not yet known) that there are many shrinkers with entropy less than the cylinder, and the geometry and topology of such low entropy shrinkers is only partially understood \cite{BW_sharp_bounds,BW_topology}. A complete classification would be no easier than classifying all minimal hypersurfaces in $S^n$ with certain area bounds.\footnote{We thank Jacob Bernstein for a helpful conversation on this topic.} Even for $n=3$, such a classification seems no simpler than a variant of the Willmore conjecture, where  the equation is substituted by some implicit asymptotic information (a classification requires, in particular, to get a lower area bound on the links of the asymptotic cones).\\

While many new ideas (see Section \ref{sec_proof_outline}) are needed in order to overcome the lack of an analogues result to \cite{BW_topological_property} when $n\geq 3$, the most interesting challenge is the following: Standard methods from geometric analysis only imply that ancient asymptotically cylindrical flows are regular away from a set of parabolic Hausdorff dimension at most $n-2$ (see Section \ref{sec_partial_reg}).\footnote{This is seemingly an issue even when $n=2$. See Section \ref{sec_proof_outline} (and in particular, the description of Section \ref{sec_cap_size}) for how smoothness was established for $n=2$.}  Hence, fundamentally new ideas are needed to establish smoothness. We will discuss them in the next two subsections.\\

\bigskip

 \subsection{Moving plane method without assuming smoothness}\label{intro_moving_planes}
     
The moving plane method was introduced by Alexandrov in \cite{Alexandrov} to show that smooth, closed, embedded constant mean curvature surfaces are spheres (see also Hopf's lecture notes on ``Differential geometry in the large'' \cite{Hopf_lect}). From there on it has been a fundamental method in showing that (not necessarily  geometric) elliptic or parabolic problems possess symmetries if they are closed, or if their boundaries or asymptotic behavior are symmetric: see for instance \cite{Serrin,GidasNiNirenberg,Schoen_uniq,Craig_Sternberg,CafarelliGidasSpruck,MSHS,CHH}, as well as the survey \cite{Brezis_symmetry}. \\    

\begin{figure}
\begin{tikzpicture}[x=1cm,y=1cm] \clip(-8,12) rectangle (8,-1);
\shade[left color=blue!50!,right color=blue!10!] (4*1,-4*1) rectangle (4*1.2,4*4);
\draw [samples=100,rotate around={0:(0,0)},xshift=0cm,yshift=0cm,domain=4*0.8245253419765505:8)] plot (\x,{(\x)^2/4});
\draw [dashed] (4*0.305,4*0.25) [partial ellipse=0:180:4*0.18cm and 4*0.03cm];
\draw (4*0.305,4*0.25) [partial ellipse=180:360:4*0.18cm and 4*0.03cm];
%\draw [dashed] (4*0.305,4*0.25) [partial ellipse=0:180:4*0.18cm and 4*0.03cm];
%\draw (4*0.305,4*0.25) [partial ellipse=180:360:4*0.18cm and 4*0.03cm];
\draw [dashed,rotate around={120:(-4*0.64,4*0.35)}] (-4*0.64,4*0.35) [partial ellipse=180:360:4*0.06cm and 4*0.01cm];
\draw [rotate around={120:(-4*0.64,4*0.35)}] (-4*0.64,4*0.35) [partial ellipse=0:180:4*0.06cm and 4*0.01cm];
\draw [dotted,thick,rotate around={240:(4*0.82,4*0.64)}] (4*0.82,4*0.64) ellipse (4*0.024cm and 4*0.004cm);
\draw [dashed] (-4*0.22,4*0.22) [partial ellipse=0:180:4*0.21cm and 4*0.035cm];
\draw (-4*0.22,4*0.22) [partial ellipse=180:360:4*0.21cm and 4*0.035cm];
\draw [dashed] (4*0.0002332491688299718,4*1.5997785993677849) [partial ellipse=0:180:4*1.2571684499174602cm and 4*0.21455991555378162cm];
\draw (4*0.0002332491688299718,4*1.5997785993677849) [partial ellipse=180:360:4*1.2571684499174602cm and 4*0.21455991555378162cm];
\draw [dashed] (4*0.00032078487570581966,4*2.1966671075002657) [partial ellipse=0:180:4*1.4743165895739172cm and 4*0.29007806481583676cm];
\draw (4*0.00032078487570581966,4*2.1966671075002657) [partial ellipse=180:360:4*1.4743165895739172cm and 4*0.29007806481583676cm];
\draw   (4*0.8245253419765505,4*0.6798420395615475)-- (4*0.95,4*0.6);
\draw   (4*0.7819932155552015,4*0.6115133891743638)-- (4*0.95,4*0.6);
%\draw [samples=100,rotate around={0:(0,0)},xshift=0cm,yshift=0cm,domain=4*0.648407203918607:4*0.7819932155552015)] plot (\x,{(\x)^2/4});
\draw   (4*0.648407203918607,4*0.4204319020935459)-- (4*0.6423844814564859,4*0.37711740705108604);
%\draw   (4*0.6150121024263338,4*0.37823988613085924)-- (4*0.6423844814564859,4*0.37711740705108604);
\draw [samples=100,rotate around={0:(0,0)},xshift=0cm,yshift=0cm,domain=-8:-4*0.6408712864207166)] plot (\x,{(\x)^2/4});
\draw   (-4*0.3646156662354182,4*0.08302374114334066)-- (-4*0.31653544754347496,4*0.10019468955154799);
\draw   (-4*0.6408712864207166,4*0.4107160057585441)-- (-4*1,4*0.3);
\draw   (-4*1,4*0.3)-- (-4*0.53991811338522,4*0.29151156916145526);
\draw [samples=100,rotate around={0:(0,0)},xshift=0cm,yshift=0cm,domain=-4*0.31653544754347496:-4*0.20622200492137188)] plot (\x,{(\x)^2/4});
\draw [shift={(-4*0.7027681841100468,-4*0.022034854577463753)}]  plot[domain=0.3:1.0917643671335866,variable=\t]({4*1*0.3533150228543174*cos(\t r)+4*0*0.3533150228543174*sin(\t r)},{4*0*0.3533150228543174*cos(\t r)+4*1*0.3533150228543174*sin(\t r)});
\draw   (-4*0.20622200492137188,4*0.04252751531379033)-- (4*0.1,4*0.3);
\draw   (4*0.1,4*0.3)-- (4*0.2130190048395076,4*0.04537709642281416);
\draw [samples=100,rotate around={0:(0,0)},xshift=0cm,yshift=0cm,domain=4*0.2130190048395076:4*0.5199513674882179)] plot (\x,{(\x)^2/4});
%\draw   (0.3835332778832391,-0.018618189878098844) circle (0.1821277408314506cm);
%\draw   (4*0.6150121024263338,4*0.37823988613085924)-- (4*0.5797243057185343,4*0.3114258682101738);
\draw   (4*0.5797243057185343,4*0.3114258682101738)-- (4*0.5199513674882179,4*0.2703494245528678);
\draw [shift={(4*0.6745638963495582,4*0.2836928070184077)}]  plot[domain=1.9:2.8571037515665014,variable=\t]({4*1*0.09881128798941159*cos(\t r)+4*0*0.09881128798941159*sin(\t r)},{4*0*0.09881128798941159*cos(\t r)+4*1*0.09881128798941159*sin(\t r)});
\draw [shift={(4*0.5386990551836357,4*0.6393654747679093)}]  plot[domain=-1.106267150134368:-0.1,variable=\t]({4*1*0.24488321123101786*cos(\t r)+4*0*0.24488321123101786*sin(\t r)},{4*0*0.24488321123101786*cos(\t r)+4*1*0.24488321123101786*sin(\t r)});
%\draw   [color=blue] (4*1,-4*1) -- (4*1,4*4);
\draw   [color=blue, dashed] (4*1.2,-4*1) -- (4*1.2,4*4);
\draw [<-, color=blue] (4*0.5,4*1.1) -- (4*0.9,4*1.1);
\begin{scriptsize}
\draw[color=blue] (4*0.7,4*1.2) node[scale=1.3] {$moving\; plane$};
\end{scriptsize}
\end{tikzpicture}
\caption{Moving plane method without assuming smoothness}\label{amazing_figure}
\end{figure}
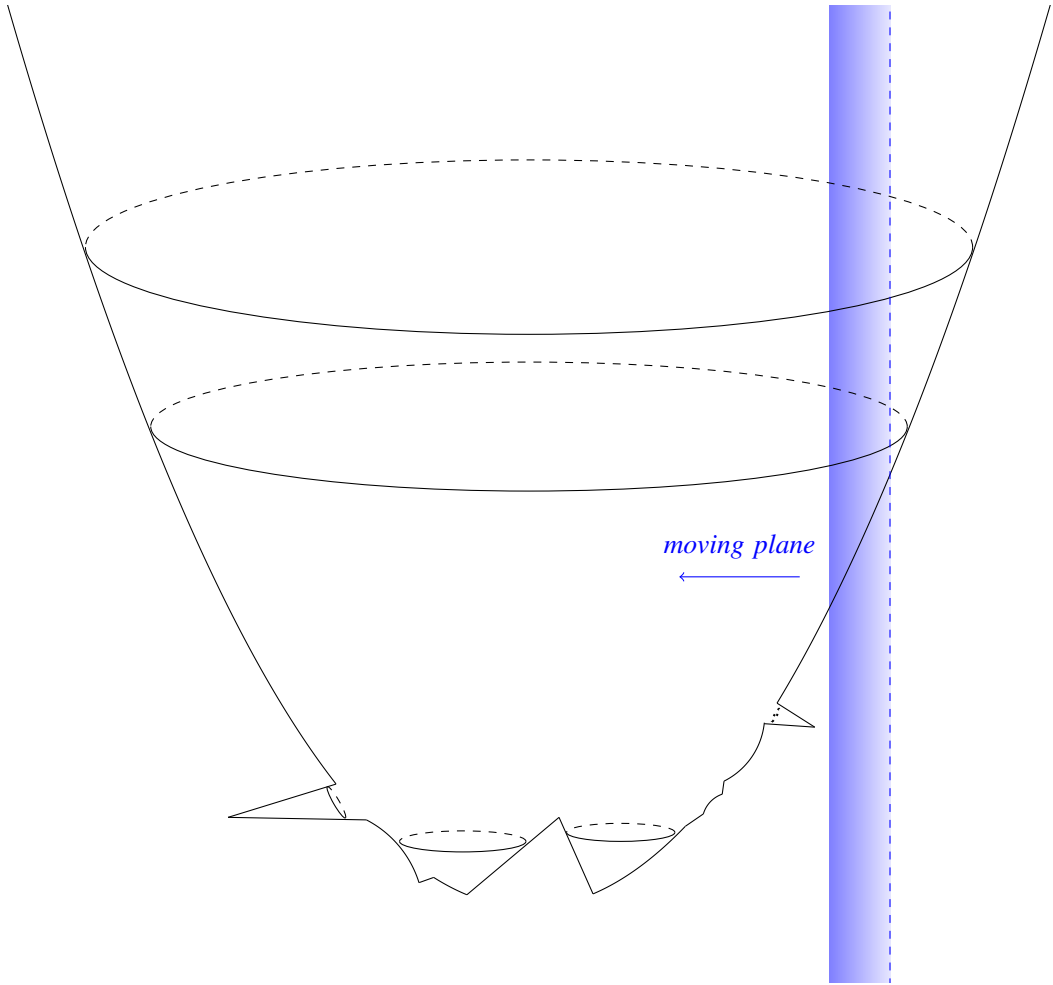

The basic idea of the moving plane method is perhaps best illustrated by Alexandrov's original proof \cite{Alexandrov}: Let $M\subset \mathbb{R}^{3}$ be smooth closed embedded constant mean curvature surface. Take any plane $P$ away from $M$,  and start pushing it, parallel to its normal, towards $M$. Push $P$ until it hits $M$, and then push it further, for as long as $M$ can be reflected across $P$ without intersecting itself. A combination of the \emph{strong maximum principle} and the \emph{Hopf lemma} shows that this process of pushing $P$ will only stop once $M$ is reflection symmetric across $P$. Similarly, a classical example for the moving plane method in the noncompact setting is Schoen's uniqueness proof for the catenoid \cite{Schoen_uniq}.\\

Traditionally, the moving plane method only works for smooth solutions. In our situation, however, it does not seem possible to first establish smoothness by standard methods (see Section \ref{sec_intro_roleofentropy}). To overcome this, we develop a novel variant of the moving plane method, where smoothness and symmetry are established in tandem. Loosely speaking, the method gives rise to the following general principle:

\begin{principle}[Moving plane method without assuming smoothness]\label{moving_planes_princ}
Smoothness and symmetry at infinity (or at the boundary) can be promoted to smoothness and symmetry in the interior.
\end{principle}

More specifically, in the proof of Theorem \ref{thm_classification_asympt_cyl}, after many other steps, we will find ourselves in the following situation: Outside of a cap of controlled size the flow $M_t$ is smooth and paraboloidal, but inside of the cap region it could be singular. This is illustrated schematically in Figure \ref{amazing_figure}. More precisely, the asymptotics are the ones summarized at the beginning of Section \ref{sec_moving_planes} (moving plane method without assuming smoothness). We then show, using the new tools for Brakke flows that will be described in the next subsection, that whenever the moving plane reaches a point, then, unless the moving plane has reached a plane of symmetry of $M_t$, the point must be a smooth point. Consequently, if the moving plane hasn't reached a plane of symmetry, we can push the moving plane further, just like in the smooth moving plane method. In this way, symmetry and smoothness (away from the axis of symmetry) are  derived in tandem.\\

To implement the moving plane method without assuming smoothness (Principle \ref{moving_planes_princ}), we establish a strong maximum principle for Brakke flows and a Hopf lemma for Brakke flows, where smoothness is a consequence rather than being an assumption. This will be described in the next subsection.

\begin{remark}
The moving plane method without assuming smoothness  (Principle \ref{moving_planes_princ}) also has applications for other
geometric problems. We will discuss some of them in \cite{HHW}.
\end{remark}

\bigskip

\subsection{New tools for Brakke flows} In this section, we describe some new tools for Brakke flows. In addition to being key ingredients in the proof of Theorem \ref{thm_classification_asympt_cyl}, these tools are also of independent interest.\\

Denote by $\mathbb{H}\subset \mathbb{R}^{n+1}$ an open halfspace whose boundary $n$-plane contains the origin. Recall first that the classical Hopf lemma says that if $u_1,u_2$ are smooth solutions of a second order parabolic equation, defined in a parabolic ball $P(0,0,r)=B_r(0)\times (-r^2,0]$, such that $u_1(0,0)=u_2(0,0)$, and $u_1(x,t)<u_2(x,t)$ for all $x\in B_r(0)\cap\mathbb{H}$ and $t\in (-r^2,0]$, then $u_1$ and $u_2$ have distinct normal derivatives at $(0,0)$.\\

The first guess about how to generalize this for Brakke flows would be to infer that the tangent flows at $(0,0)$ must be distinct. But actually we can do much better. In essence, we can use the fact that one flow lies above the other one to \emph{conclude} that $(0,0)$ must be a smooth point. Specifically, we prove:

\begin{theorem}[Hopf lemma without assuming smoothness, c.f. Theorem \ref{mirror-theorem}]\label{mirror-theorem_intro}
Let $\mathcal{M}^1,\mathcal{M}^2$ be integral Brakke flows defined in the parabolic ball $P(0,0,r)$.
If
\begin{enumerate}
\item  $(0,0)\in\spt \mathcal{M}^1\cap \spt\mathcal{M}^2$ is a tame point for both flows,
\item $\partial \mathbb{H}$ is \textbf{not} the tangent flow to either $\mathcal{M}^1$ or $\mathcal{M}^2$ at $(0,0)$,
\item and $\reg M^1_t\cap \mathbb{H}$ and $\reg M^2_t\cap \mathbb{H}$ are disjoint for $t\in(-r^2,0 )$, 
\end{enumerate}
then $\mathcal{M}^1$ and $\mathcal{M}^2$ are smooth at $(0,0)$, with distinct tangents.
\end{theorem}

For our Brakke flow version of the moving plane method,
the case of particular interest is when $\mathcal{M}^2$ is the image of $\mathcal{M}^1$ under reflection
in the plane $\partial \mathbb{H}$.

The most important feature of Theorem \ref{mirror-theorem_intro} (Hopf lemma without assuming smoothness) is that smoothness is a conclusion rather than being an assumption. It is thus a fundamental new tool to establish smoothness for Brakke flows. Our new method also applies to other geometric problems.\\

Let us now explain the technical details of the statement of Theorem \ref{mirror-theorem_intro}: If $\mathcal{M}$ is an integral Brakke flow then $\spt \mathcal M$ denotes its support (i.e. all points with Gaussian density $\geq 1$), and  $\reg M_t$ denotes the regular part of the time $t$ slice of the support. A point $X\in\spt\mathcal M$ is called a \emph{tame point} if for each tangent flow of $\mathcal M$ at $X$, the time $-1$ slice is smooth with multiplicity one away from a set of $(n-1)$-dimensional Hausdorff measure zero. This condition is sharp, since the statement clearly fails for triple-junctions. It is easy to check that all ancient asymptotically cylindrical flows are tame (see Corollary \ref{cor_tameness}).\\

We also prove a strong maximum principle for Brakke flows. While more general versions are possible, for our purpose, the following version where one of the flows is assumed to be smooth is sufficient:

\begin{theorem}[Strong maximum principle for Brakke flows, c.f. Theorem \ref{strong_max_Brakke}]\label{strong_max_Brakke_intro}
Let $\mathcal{M}^1$ be a smooth mean curvature flow defined in a parabolic ball $P(X_0,r)$ based at $X_0\in\spt\mathcal{M}^1$, where $r>0$ is small enough such that $\spt \mathcal{M}^1$ separates $P(X_0,r)$ into two open connected components, $\mathcal{U}$ and $\mathcal{U}'$. 
Let $\mathcal{M}^2$ be an integral Brakke flow defined in $P(X_0,r)$ with $X_0\in \spt \mathcal{M}^2$. If $X_0$ has Gaussian density $\Theta_{X_0}(\mathcal{M}^2)<2$, and if
\begin{equation}\label{max_princ_cont_intro}
\spt \mathcal{M}^2 \subseteq \mathcal{U}\cup \spt \mathcal{M}^1,
\end{equation}
then $X_0$ is a smooth point for $\mathcal{M}^2$ and there exists some $\eps>0$ such that 
\begin{equation}\label{conclusion_strong_max_intro}
\spt \mathcal{M}^2 \cap P(X_0,\eps) = \spt \mathcal{M}^1\cap P(X_0,\eps).
\end{equation}
\end{theorem}

Again, the important feature is that smoothness of $\mathcal{M}^2$ is not an assumption but a conclusion. While the present paper is seemingly the first to discuss a singular \textit{strong} maximum principle for mean curvature flow, nonsmooth versions of the strong maximum principle for minimal surfaces have been extensively studied \cite{Solomon_White,Simon_max,Ilmanen_max,Wick_max}, and indeed Theorem \ref{strong_max_Brakke_intro} follows quite easily from the elliptic result from Solomon-White \cite{Solomon_White}. On the other hand, we know of no geometric analytic instance of a singular analogue of the Hopf lemma that has previously appeared.\\

\bigskip

\subsection{Applications of our classification result}\label{intro_application}
In this section, we discuss applications of Theorem \ref{thm_classification_asympt_cyl} (classification of ancient asymptotically cylindrical flows) to mean curvature flow through singularities.\\

Central for this is the notion of a \emph{neck singularity}. Some simple examples of neck singularities are the classical neck-pinch, the degenerate neck-pinch and the doubly degenerate neck-pinch from \cite{AAG,AV_degenerate_neckpinch}. However, the definitions below are much more general. At the first singular time, there are many equivalent ways to describe this. One standard way to phrase the definition is that a mean curvature flow $\mathcal M$ has neck singularity at $X=(x,T)$, if some tangent flow of $\mathcal{M}$ at $X$ is the family of round shrinking cylinders
\begin{equation}
\left\{S^{n-1}(\sqrt{-2(n-1)t})\times\mathbb{R}\right\}_{t<0}
\end{equation}
with multiplicity one. Recall that a \emph{tangent flow} at $X$ is any limit of the rescaled flows $\mathcal{D}_{\lambda_j}(\mathcal{M}-X)$ where $\lambda_j\to \infty$. Cylindrical tangent flows are unique by an important result of Colding-Minicozzi \cite{CM_uniqueness}.\\

If $\mathcal M$ has a neck singularity at $X=(x,T)$, then by the \emph{mean-convex neighborhood conjecture for neck singularities} (see e.g. \cite{Ilmanen_problems,CMP,HershkovitsWhite}, and also \cite{AAG}) it is expected that there is some space-time neighborhood of $X$ in which the flow is mean-convex (possibly after flipping orienation), i.e., that there
 is some $\eps=\eps(X)>0$ such that
\begin{equation}
M_t\cap B_\eps(x) \textrm{ is mean-convex for all  } t \in (T-\eps^2,T+\eps^2).
\end{equation}
We remark that more generally one can also ask whether $S^{n-k}(\sqrt{-2(n-k)t})\times \mathbb{R}^k$ singularities have a mean-convex neighborhood, but in the present paper we only deal with neck singularities, i.e. with the case $k=1$.
 For more extensive background on the mean-convex neighborhood conjecture, see \cite{CHH}.\\
 
 Our main result below is Theorem \ref{thm_mean_convex_nbd_intro} (mean-convex neighborhoods), where we establish the mean-convex neighborhood conjecture for neck singularities at all times.
A major difficulty, which we overcome in our proof, is to exclude certain potential scenarios of nonconvex singularities that are not directly captured by the tangent flow. In particular, we rule out the potential scenario of a degenerate neckpinch with a nonconvex cap, as well as the possibility of nonconvex caps slightly after the neck singularity.\\

For the sake of exposition, let us first state Theorem \ref{thm_mean_convex_nbd_intro} in the special case of the first singular time:

\begin{theorem}[mean-convex neighborhoods at the first singular time]
Let $\mathcal{M}=\{M_t\}_{t\in[0,T)}$ be a mean curvature flow of closed embedded hypersurfaces in $\mathbb{R}^{n+1}$, where $T$ is the first singular time. If the flow has a neck singularity at $X=(x,T)$, then there exists an $\eps=\eps(X)>0$ such that, possibly after flipping the orientation, the flow $\mathcal{M}$ is mean-convex in the two-sided parabolic ball $B(X,\eps)$. Moreover, any nontrivial special limit flow at $X$ is either a round shrinking cylinder or a translating bowl soliton.\footnote{It is easy to see that ancient ovals cannot arise as special limit flow at the first singular time, i.e. they cannot arise as limit of any sequence of rescaled flows $\mathcal{D}_{\lambda_j}(\mathcal{M}-X_j)$ where $\lambda_j\to\infty$ and $X_j=(x_j,t_j)\to X$ with $t_j\leq T$.}
\end{theorem}

Before stating our solution of the mean-convex neighborhood conjecture for neck singularities in arbitrary dimensions at all times, let us first recall some notions about mean curvature through singularities:
For any closed set $K\subset\mathbb{R}^{n+1}$, its \emph{level set flow} $F_t(K)$ is the maximal family of closed sets starting at $K$ that satisfies the avoidance principle \cite{EvansSpruck,CGG,Ilmanen_book}. Now, for any closed embedded hypersurface $M\subset \mathbb{R}^{n+1}$ there are at least three reasonable evolutions through singularities, c.f. \cite{Ilmanen_book,HershkovitsWhite}, namely its level set flow $F_t(M)$, its outer flow $M_t$ and its inner flow $M_t'$. The latter two are defined as follows. Let $K$ be the compact domain bounded by $M$, and let $K':=\overline{\mathbb{R}^{n+1}\setminus K}$. Denote the corresponding level set flows by
\begin{equation}
K_t=F_t(K),\quad \textrm{ and } \quad K_t' = F_t(K').
\end{equation}
Let $\mathcal{K}$ and $\mathcal{K}'$ be their space-time tracks, namely
\begin{align}
\mathcal{K}= \{(x,t)\in \mathbb{R}^{n+1}\times \mathbb{R}_+\, | \, x\in K_t \},\quad \textrm{ and } \quad
\mathcal{K}'= \{(x,t)\in \mathbb{R}^{n+1}\times \mathbb{R}_+\, | \, x\in K_t' \}.
\end{align}
The \emph{outer flow of $M$} and \emph{inner flow of $M$} are then defined by
\begin{align}
M_t=\{x\in \mathbb{R}^{n+1}\, | \, (x,t)\in \partial \mathcal{K} \},\quad \textrm{ and } \quad
M_t'=\{x\in \mathbb{R}^{n+1}\, | \, (x,t)\in \partial \mathcal{K}' \}.
\end{align}
Of course, as long as the evolution is smooth, $M_t,M_t'$ and $F_t(M)$ coincide.\\

Denote by $\mathcal{D}_\lambda(\mathcal{K}-X)$ and $\mathcal{D}_\lambda(\mathcal{K}'-X)$ the flows that are obtained from $\mathcal{K}$ and $\mathcal{K'}$, respectively, by shifting $X$ to the origin, and parabolically rescaling by $\lambda$. The following is the minimal and thus most general definition capturing the formation of a neck singularity under weak mean curvature flow, and simultaneously maintaining the notion of an ``inside'' and an ``outside'':

\begin{definition}[neck singularity]\label{def_neck_sing}
The evolution of a closed embedded hypersurface $M\subset\mathbb{R}^{n+1}$ by mean curvature flow has an inwards (respectively outwards) \emph{neck singularity} at $X\in \mathbb{R}^{n+1}\times \mathbb{R}_{+}$ if the rescaled flow
 $\mathcal{D}_\lambda(\mathcal{K}-X)$ (respectively $\mathcal{D}_\lambda(\mathcal{K}'-X)$) converges for $\lambda\to \infty$ locally smoothly with multiplicity one to a round shrinking solid cylinder $\{\bar{B}^n(\!\!\sqrt{2(n-1)|t|})\times \mathbb{R} \}_{t<0}$, up  to rotation.
 \end{definition}

The following result establishes the mean-convex neighborhood conjecture for neck singularities in arbitrary dimension at all times, and thus generalizes the main result of \cite{CHH} to higher dimensions:

\begin{theorem}[mean-convex neighborhoods]\label{thm_mean_convex_nbd_intro} 
Assume $X=(x,T)$ is a space-time point at which the evolution of a closed embedded hypersurface $M\subset\mathbb{R}^{n+1}$ by mean curvature flow has an inward neck singularity. Then there exists an $\varepsilon=\varepsilon(X)>0$ such that
\begin{equation}
\quad K_{t_2}\cap B(x,\varepsilon)\subseteq K_{t_1}\setminus M_{t_1} 
\end{equation} 
for all $T-\varepsilon^2< t_1< t_2 < T+\varepsilon^2$. Similarly, if the evolution has an outward neck singularity at $X$, then there exists some $\varepsilon=\varepsilon(X)>0$ such that
\begin{equation}
\quad K_{t_2}'\cap B(x,\varepsilon)\subseteq K_{t_1}'\setminus M_{t_1}'
\end{equation}
for all $T-\varepsilon^2< t_1< t_2 < T+\varepsilon^2$. Furthermore, in both cases, any nontrivial limit flow at $X$ is either a round shrinking sphere, a round shrinking cylinder, a translating bowl soliton or an ancient oval.\footnote{More generally, the theorem also holds for mean curvature flow in arbitrary ambient manifolds $N^{n+1}$.}
\end{theorem}

In fact, as a consequence of our proof of Theorem \ref{thm_mean_convex_nbd_intro} (or alternatively as a consequence of the statement of Theorem \ref{thm_mean_convex_nbd_intro} combined with the prior classification from \cite{ADS2,BC2}) we obtain a canonical neighborhood theorem for neck singularites in arbitrary dimensions:

\begin{corollary}[canonical neighborhoods]\label{cor_canonical_neighborhoods}
Assume $X=(x,T)$ is a space-time point at which the evolution of a closed embedded hypersurface $M\subset\mathbb{R}^{n+1}$ by mean curvature flow has a neck singularity. Then for every $\eps>0$ there exists a $\delta=\delta(X,\eps)>0$ with the following significance. For any regular point\footnote{Choosing $\delta>0$ small enough, the singular set of $\mathcal{M}\cap B(X,\delta)$ has parabolic Hausdorff dimension at most $1$. In particular, all points are regular at almost every time.} $X'\in B(X,\delta)$ the flow $\mathcal{M}'=\mathcal{D}_{\lambda}(\mathcal{M}-X')$ which is obtained from $\mathcal{M}$ by shifting $X'$ to the origin and parabolically rescaling by $\lambda=|{\bf{H}}(X')|$ is $\eps$-close in $C^{\lfloor 1/\eps \rfloor}$ in $B_{1/\eps}(0)\times (-1/\eps^2,0]$ to a round shrinking sphere, a round shrinking cylinder, a translating bowl soliton or an ancient oval.
\end{corollary}

The conclusion of Corollary \ref{cor_canonical_neighborhoods} gives a similar structure for singularities as in two-convex mean curvature flow \cite{HuiskenSinestrari_surgery,BH_surgery,HaslhoferKleiner_surgery,ADS,ADS2,BC,BC2} and in 3d Ricci flow \cite{Perelman1,Perelman2,Brendle_ricci1,Brendle_ricci2,ABDS_ricci,BDS_ricci}.\footnote{The classification of noncompact $\kappa$-solutions, as well as rotational symmetry of compact $\kappa$-solutions, has been obtained in very important work by Brendle \cite{Brendle_ricci1,Brendle_ricci2}. Using this, the classification in the compact case has been completed by Angenent, Brendle, Daskalopoulos and Sesum \cite{ABDS_ricci,BDS_ricci}. We also note that Bamler-Kleiner \cite{BK_symmetry} later found an alternative proof of rotational symmetry in the compact case, which uses results and ideas from \cite{Brendle_ricci1}.} However, for two-convex mean curvature flow and 3d Ricci flow convexity is known a priori due to the convexity estimate \cite{HuiskenSinestrari_convexity,White_nature,HaslhoferKleiner_meanconvex} respectively \cite{Hamilton_survey}. Namely, the crucial difference is that in Corollary \ref{cor_canonical_neighborhoods} the convexity is not known a priori but comes as a consequence of our main classification result of ancient asymptotically cylindrical flows (Theorem \ref{thm_classification_asympt_cyl}).\\

Our second main applications concerns uniqueness of mean curvature flow. It has been understood since the 90s that the mean curvature evolution of a hypersurface $M\subset\mathbb{R}^{n+1}$ through singularities in general can be nonunique \cite{Velazquez,AIC,White_ICM}. To capture this, one considers the \emph{discrepancy time}
\begin{equation}
T_{\textrm{disc}}= \inf \{ t > 0 \, | \, \textrm{$M_t$, $M_t'$, and $F_t(M)$ are not all equal} \},
\end{equation}
which is the first time when the evolution becomes nonunique, c.f. \cite{HershkovitsWhite}.\footnote{In particular, no discrepancy implies no fattening. But the best way to capture nonuniqueness is via discrepency.} While the flow through general singularities can be highly nonunique, it has been conjectured (see e.g. \cite{AAG,White_ICM,HershkovitsWhite}, and also \cite{ACK, Carson,BamlerKleiner}) that the flow through neck singularities should be unique. The following result confirms this in arbitrary dimension:

\begin{theorem}[uniqueness]\label{thm_nonfattening_intro}
If $T\in(0,T_{\textrm{disc}}]$, and if all the backward singularities of the outer flow $\{M_t\}$ at time $T$ are neck singularities or spherical singularities, then $T<T_{\textrm{disc}}$. In particular, mean curvature flow through neck singularities and spherical singularities is unique.
\end{theorem}

Theorem \ref{thm_mean_convex_nbd_intro} (mean-convex neighborhoods), Corollary \ref{cor_canonical_neighborhoods} (canonical neighborhoods) and Theorem \ref{thm_nonfattening_intro} (uniqueness) give a precise description of the mean curvature flow around any neck singularity. Previously, the dynamics of mean curvature flow around neck singularities was successfully understood only under additional assumptions, either by making the notion of neck more rigid, or by making further assumptions on the initial hypersurface $M$. In \cite{AAG}, Altschuler, Angenent and Giga gave a full description of such dynamics, both backwards and forwards in time, in the case that $M$ is rotationally symmetric. More recently, using highly sophisticated PDE methods, Gang \cite{Gang1,Gang2}, building upon his earlier work with Knopf and Sigal \cite{GangKnopf,GKS}, was able to successfully analyze the backwards in time behavior of a smooth mean curvature flow around certain nondegenerate neck singularities.\\

Our final application concerns mean curvature flow in higher codimension. In general, as the codimension increases the complexity of the parabolic system increases. However, combining our main classification result (Theorem \ref{thm_classification_asympt_cyl}) with a recent result of Colding-Minicozzi \cite[Thm. 0.9]{CM_complexity} we obtain:

\begin{corollary}[classification of ancient asymptotically cylindrical mean curvature flows in higher codimension]\label{cor_highercodim}
If $M_t^n\subset \mathbb{R}^N$ is an ancient smooth mean curvature flow in arbitrary codimension such that some blowdown for $t\to-\infty$ is $\{S^{n-1}(\sqrt{-2(n-1)t})\times \mathbb{R}\}_{t<0}$, then $M_t$ is contained in an $(n+1)$-dimensional subspace and is either a round shrinking cylinder, a translating bowl soliton, or an ancient oval.
\end{corollary}

Making Corollary \ref{cor_highercodim} applicable for general limit flows would require removing the smoothness assumption. To do so one would have to remove the smoothness assumption in \cite[Thm. 0.9]{CM_complexity}.\\

\bigskip

\subsection{Outline of the proofs}\label{sec_proof_outline}

The strategy of the proof of Theorem \ref{thm_classification_asympt_cyl} (classification of ancient asymptotically cylindrical flows) is related to the one of Theorem \ref{thm_class_low_ent} (classification of ancient low entropy flows), which was proved in \cite{CHH}, but there are many important differences. In order to avoid duplication, we do not include proofs of claims where the argument is a straightforward adaptation of the one appearing in \cite{CHH}. The reader of the present paper will, therefore, have to consult \cite{CHH} quite frequently. On the other hand, we hope that in addition to succinctness, our choice of not spelling out easy adaptions of old proofs will highlight the aspects of this paper that are genuinely new. Correspondingly, this outline will concentrate on the new aspects in this work, and will draw comparisons to \cite{CHH}.\\

In Section \ref{sec_prelim}, we set up the notation and collect some preliminaries.\\

In Section \ref{sec_new_tools}, we establish our new tools for Brakke flows. We first prove the strong maximum principle for Brakke flows (Theorem \ref{strong_max_Brakke_intro}) via reduction to its elliptic analogue \cite{Solomon_White}. The proof of the Hopf Lemma without smoothness (Theorem \ref{mirror-theorem_intro}) is more involved, and the idea is as follows:

 We first derive a weak version of the Hopf lemma, saying that the (not necessarily  smooth) tangents to $\mathcal{M}^1$ and $\mathcal{M}^2$ are not the same, and are thus disjoint in the halfspace $\mathbb{H}$. To do so, we show that if they had the same (nonplanar) time $-1$ slice $\Sigma$, we could construct a positive solution to the linearization of the renormalized mean curvature flow equation on $\Sigma\cap \mathbb{H}$,  which is further bounded, independent of time, on compact subsets of space. The existence of such a solution implies that $\Sigma$ is stable in $\mathbb{H}$, which by a version of a theorem of Brendle \cite{Brendle_genus_zero} implies that $\Sigma$ is a hyperplane.

Once that is done, an adaptation of an argument of Brendle \cite{Brendle_genus_zero} shows that the time $-1$ slices $\Sigma_1$ and $\Sigma_2$ of the tangent flows must both be planar. Roughly speaking, as $\Sigma_1$ and $\Sigma_2$ are critical points of the Gaussian area functional,
\begin{equation}
F[\Sigma]=\int_\Sigma e^{-|x|^2/4},
\end{equation}
they act as barriers,  so  we can construct a hypersurface minimizing $F$ between them. Such a hypersurface is stable in $\mathbb{H}$, and hence must be a halfplane. On the other hand the only two shrinkers which can fit a halfplane between them are planes. 

We remark here that since not only the flows $\mathcal{M}^1$ and $\mathcal{M}^2$, but also their tangents are a priori singular, quite a bit of geometric measure theoretic care is needed in proving Theorem \ref{mirror-theorem_intro} (Hopf lemma without smoothness). Due to this reason (an in particular, issues relating to connectedness of the regular set in a halfspace), the actual argument in Section \ref{sec_new_tools} is organized in the reverse order of the one described above. Namely, we first prove Bernstein-type theorems for singular shrinkers in a halfspace (Theorem \ref{brendle-stable} and Theorem \ref{brendle-connected}), and then run the Jacobi field argument afterwards. Interestingly, due to the presence of the Gaussian conformal factor, the argument works in all dimensions, and not just for $n<7$.\\

In Section \ref{sec_coarse_properties}, we establish several coarse properties of ancient asymptotically cylindrical flows, which we will use in later sections. This section, together with Section \ref{sec_new_tools}, contains the key features that allow us to deal with shrinkers with entropy less than $\mathrm{Ent}[S^{n-1}\times \mathbb{R}]$. In particular, it contains many new results, which have no analogue in \cite{CHH}.

First, in Section \ref{sec_partial_reg}, using ideas from \cite{BW_topological_property} we prove a partial regularity result for ancient asymptotically cylindrical flows. In particular, we see that ancient asymptotically cylindrical flows are tame.

Based on that,  in Section \ref{sec_enc_domain} we show that the flows possess a consistent notion of a domain they bound.

Next, in Section \ref{sec_sim_back_in_time}, we generalize and extend ideas from \cite{BW_sharp_bounds, BW_topology} to give a structural result (Proposition \ref{uniform_r_high_d}) for shrinkers $\Sigma$ with entropy less than or equal to $\mathrm{Ent}[S^{n-1}\times \mathbb{R}]$. Namely, we show that there exists a uniform constant $R_0<\infty$ such that we have the following trichotomy:
\begin{enumerate}
\item $\Sigma$ is a round cylinder, or
\item $\Sigma$ is compact and contained in $B(0,R_0)$, or
\item $\Sigma$ is separating within $B(0,R_0)$.\footnote{Our notion of separating is very closely related to the notion of strongly noncollapsed from \cite{BW_sharp_bounds}.}
\end{enumerate}
 
We then use this structural result in order to study the behavior of the flow $\mathcal M$ around any point $X_0=(x_0,t_0)\in \mathcal{M}$ at dyadic scales $r_j=2^j$. Using quantitative differentiation, we see that there exists an $N<\infty$ such that for every  $X_0\in \mathcal{M}$ and $r_0>0$, there exists a scale $r\in [r_0,2^Nr_0]$ such that $\mathcal{D}_{r^{-1}}(\mathcal{M}-X_0)$ is very close to the flow of one of those shrinkers $\Sigma$ in $B(0,100R_0)\times [-2,-1]$.  If $\Sigma$ is compact, this implies a diameter bound on the connected component of $M_{t_0}$ containing $x_0$. If $\Sigma$ is of  separating type, this implies that $\mathcal{M}$ does not turn extinct prior to time $t_0+r^2$. The trichotomy of cylindrical behavior backwards in time, bounded diameter at the present time or nonextinction in future time, turns out to be a sufficient substitute to the precise classification of Bernstein-Wang in dimension two \cite{BW_topological_property}, which played a central role in the arguments in \cite{CHH}.

Next, in Section \ref{traped_regions},  we use quantitative differentiation  and the local regularity theorem to derive regularity estimates for flows that are trapped in a thin slab over many scales. This result  is perhaps of some independent interest.

Finally, in Section \ref{sec_asympt_cyl_scale}, similarly as in \cite[Sec. 3.3]{CHH}, we record that the hypersurfaces $M_t$ open up slower than any cone of positive angle. \\

If Section \ref{sec_fine_neck_analysis}, we carry out the fine neck analysis on the cylinder. Fortunately, the fine neck analysis from \cite[Sec. 4]{CHH} easily generalizes to $n\geq 3$ (with the only exception that the algebra of $\frak{o}(n)$ is a bit more complicated than the one of $\frak{o}(2)$), so let us only briefly describe the results.

Given $X=(x_0,t_0)\in \mathcal{M}$, consider the renormalized mean curvature flow 
\begin{equation}
\bar{M}_{\tau}^X=e^{\tau/2}(M_{t_0-e^{-\tau}}-x_0).
\end{equation}
Being asymptotic cylindrical implies that $\bar{M}^X_{\tau}$ is a graph of a function $u$ in cylindrical coordinates around the $x_{n+1}$-axis (after rotation) in $B(0,100n)$, satisfying
\begin{equation}\label{u_exp_first}
u(x_{n+1},\theta,\tau)=\sqrt{2(n-1)}+o(1)
\end{equation}
for $\tau\to -\infty$. The spectral analysis implies the following dichotomy for any ancient asymptotically cylindrical flow $\mathcal{M}$, which is not the round shrinking cylinder. Either:  
\begin{enumerate}[(I).]
\item $M_t$ is compact for all $t$, or   
\item For $-t$ sufficiently large, $M_t$ is noncompact and (after suitable choice of ambient coordinates) satisfies $\sup_{p\in M_t}x_{n+1}(p)=\infty$ and 
\begin{equation}\label{intro_psi_def}
\psi(t):=\inf_{p\in M_t} x_{n+1}(p)>-\infty.
\end{equation}
In fact, Theorem \ref{thm Neck asymptotic} (fine neck theorem) implies that there exists some $a=a(\mathcal{M})>0$ such that in  in $B(0,100n)$, the expansion \eqref{u_exp_first} can be improved to 

\begin{equation}\label{u_exp_second}
u(x_{n+1},\theta,\tau)=\sqrt{2(n-1)}+ax_{n+1}e^{\tau/2}+o(e^{\tau/2}),
\end{equation}
for $-\tau \gg \log Z(X)$, where $Z(X)$ denotes the cylindrical scale of the point $X$ - the smallest scale at which the flow $\mathcal{M}$ looks sufficiently cylindrical around $X$. 
\end{enumerate}

Whether an asymptotically cylindrical flow $\mathcal{M}$ falls into category (I) or (II) is completely determined by which mode of the linearization of the renormalized mean curvature flow around the cylinder dominates. Flows of category (I) -- \emph{``the compact case''} --  correspond to the case where the neutral mode dominates, while category (II) -- \emph{``the noncompact case''} -- corresponds to the case where the plus mode dominates.   \\

In Section \ref{sec_cap_size}, similarly as in \cite[Sec. 5, Sec. 6.1]{CHH}, we show that in the noncompact case, the flow $\mathcal{M}$ resembles a translating bowl solution outside of a cap region of controlled size. However, due to the potential scenario of other shrinkers of entropy less than $\mathrm{Ent}[S^{n-1}\times \mathbb{R}]$, one needs to work more to gain less.
As in \cite[Prop. 5.1]{CHH}, the starting point is of to show that the cylindrical scale of points $p\in M_t$ goes to infinity as $x_{n+1}(p)\rightarrow \infty$. In \cite[Sec. 5.1]{CHH}, this and the classification of two-dimensional low entropy shrinkers from \cite{BW_topological_property} quickly gave a global curvature bound, and in particular, implied that tip points, i.e. point $p\in M_t$ with $x_{n+1}(p)=\psi(t)$, c.f. \eqref{intro_psi_def}, move at bounded speed. In our current setting where $n\geq 3$, we use the cylindrical - bounded diameter - no extinction trichotomy of Section \ref{sec_sim_back_in_time} to obtain the much weaker conclusion that $\mathcal{M}$ is eternal (i.e. does not vanish at any finite time) and that its tip position goes to infinity, i.e.
\begin{equation}
\lim_{t\rightarrow \infty} \psi(t)=\infty.
\end{equation}
As a substitute for the pointwise speed bounds we had in \cite[Sec. 5.1]{CHH}, we prove a macroscopic speed bound, according to which
\begin{equation}
\psi(t)-\psi(t') \leq C(t-t'),\qquad \textrm{ provided that }\,\, t\geq t'+1.
\end{equation}

These ingredients turn out to be sufficient in order to show,  as in  \cite[Sec. 5.2, Sec. 6.1]{CHH}, that, up to scaling and shifting, $\psi(t)=t+o(t)$, and that away from a ball of controlled radius $C=C(\mathcal{M})<\infty$ around a tip point $p_t$, the hypersurface $M_t$ is smooth, opens up like a paraboloid, and becomes more and more rotationally symmetric as $x_{n+1}-\psi(t)\rightarrow \infty$. Whether $M_t$ is smooth (and of uniformly bounded curvature) in $B(p_t,C)$ is still unknown at this stage, which is a key difference between our argument for $n\geq 3$ appearing here and the argument for $n=2$ in \cite[Sec. 5]{CHH}.\\

In Section \ref{sec_moving_planes}, we show that the fine paraboloidal asymptotic expansion away from a cap of controlled size, which was concluded in Section \ref{sec_cap_size}, implies genuine rotational symmetry of $\mathcal{M}$ as well as smoothness away from the axis of symmetry. We do so by implementing the moving plane method without assuming smoothness, which was discussed already earlier in the introduction (see Section \ref{intro_moving_planes}).\\

In Section \ref{sec_classification_noncompact}, we conclude that in the noncompact case $\mathcal{M}$ is the bowl. We first show that $\mathcal{M}$ is smooth across the axis of symmetry as well, and then conclude exactly as in \cite{CHH}.

In Section \ref{sec_classification_compact}, we show that in the compact case $\mathcal{M}$ is an ancient oval. The argument is based on the one in \cite{CHH}, and in particular uses the classification in the noncompact case, which is already established at this point of the argument. Some extra care is needed, however, as in higher dimensions we cannot assume a priori that $M_t$ is topologically a sphere, or even that it is connected or smooth. One interesting feature of the high-dimensional treatment is that smoothness, mean-convexity, and sphericality are proved simultaneously. Another interesting feature is that to prove connectedness we have to reinvoke the cylindrical - bounded diameter - no extinction trichotomy from Section \ref{sec_sim_back_in_time}.\\

Finally, in Section \ref{sec_app}, we show how our classification of ancient asymptotically cylindrical flow implies the mean-convex neighborhood conjecture for neck singularities as well as the uniqueness conjecture for mean curvature flow through neck singularities. As in \cite{CHH}, the first step in showing Theorem \ref{thm_mean_convex_nbd_intro} is to construct a unit-regular integral Brakke flow in some neighborhood of $X$, with the same support, for which $\bf H$ does not vanish at regular points, and which has only cylindrical and spherical singularities in that neighborhood. The argument we give here is different from the one in \cite{CHH}, which again relied on the classification of low entropy shrinkers. Our new argument, in which we rescale such that we always see the cylindrical neck far enough back in time, has advantages over the old one, even for $n=2$: While in \cite{CHH} we classified all limit flows at neck singularities for $n=2$, we were not able to relate the axis of their asymptotic cylinder (or the direction of translation in the bowl case) to the axis of the neck singularity. Our argument here shows that, as one may suspect, they are in fact the same.
The rest of the proof of Theorem \ref{thm_mean_convex_nbd_intro} and Theorem \ref{thm_nonfattening_intro} is the same as in \cite{CHH}.\\     
 
\bigskip

\noindent\textbf{Acknowledgments.} KC has been partially supported by NSF Grant DMS-1811267 and KIAS Individual Grant MG078901. RH has been partially supported by an NSERC Discovery Grant (RGPIN-2016-04331) and a Sloan Research Fellowship. OH has been partially supported by a Koret Foundation early career scholar award, and by an AMS-Simons travel grant. BW has been partially supported by NSF grant DMS-1711293. We thank the referees for very detailed and helpful comments.\\

\bigskip

\section{Preliminaries and notation}\label{sec_prelim}

As in \cite[Def. 6.2, 6.3]{Ilmanen_book}, an $n$-dimensional \emph{integral Brakke flow} in $\mathbb{R}^{n+1}$ is a family of Radon measures $\mathcal M = \{\mu_t\}_{t\in I}$ in $\mathbb{R}^{n+1}$ that is integer $n$-rectifiable for almost all times and satisfies
\begin{equation}
\frac{d}{dt} \int \varphi \, d\mu_t \leq \int \left( -\varphi {\bf H}^2 + \nabla\varphi \cdot {\bf H} \right)\, d\mu_t
\end{equation}
for all test functions $\varphi\in C^1_c(\mathbb{R}^{n+1},\mathbb{R}_+)$. Here, $\tfrac{d}{dt}$ denotes the limsup of difference quotients, and $\bf{H}$ denotes the mean curvature vector of the associated varifold $V_{\mu_t}$, which is defined via the first variation formula and exists almost everywhere at almost all times. The integral on the right hand side is interpreted as $-\infty$ whenever it does not make sense literally. All Brakke flows that we encounter in the present paper are \emph{unit-regular} as defined in \cite{White_regularity,SchulzeWhite}. Namely, every spacetime point of density one is a regular point, i.e. the flow is smooth in a two-sided parabolic neighborhood.\\

Given any space-time point $X_0=(x_0,t_0)$, by Huisken's monotonicity formula \cite{Huisken_monotonicity,Ilmanen_monotonicity} we have\footnote{Assuming say that our flow has finite area ratios, which follows from finite entropy, c.f. Proposition \ref{prop_ent_bound}.}
\begin{equation}\label{eq_huisken_mon}
\frac{d}{dt} \int \rho_{X_0}(x,t) \, d\mu_t(x) \leq -\int \left|{\bf H}(x,t)-\frac{(x-x_0)^\perp}{2(t-t_0)}\right|^2 \rho_{X_0}(x,t)\, d\mu_t(x),
\end{equation}
where 
\begin{equation}
\rho_{X_0}(x,t)=\frac{1}{4\pi(t_0-t)} e^{-\frac{|x-x_0|^2}{4(t_0-t)}} \qquad (t<t_0).
\end{equation}

The \emph{density} of $\mathcal M$ at $X_0$ is defined by
\begin{equation}
\Theta_{X_0}(\mathcal M)=\lim_{t\nearrow t_0} \int \rho_{X_0}(x,t) \, d\mu_t(x).
\end{equation}
By the local regularity theorem \cite{Brakke,White_regularity} any space-time point with density close to $1$ is regular.\\

Given any $X$ with $\Theta_X(\mathcal{M})\geq 1$ and $\lambda_i\to \infty$, let $\mathcal{M}^i_X$ be the Brakke flow which is obtained from $\mathcal{M}$ by translating $X$ to the space-time origin and parabolically rescaling by $\lambda_i$. By the compactness theorem for Brakke flows \cite{Ilmanen_book} one can pass to a subsequential limit $\hat{\mathcal{M}}_X$, which is called a \emph{tangent flow at $X$}. By the monotonicity formula, every tangent flow is backwardly selfsimilar, i.e. $\hat{\mathcal{M}}_X\cap \{t\leq 0\}$ is invariant under parabolic dilation $\mathcal{D}_\lambda(x,t)=(\lambda x,\lambda^2 t)$. If $\mathcal{M}$ is ancient, then for any $\lambda_i\to 0$ one can also pass along a subsequence to a backwardly selfsimilar limit $\check{\mathcal{M}}$, which is called a \emph{tangent flow at infinity}.\\

The \emph{support} of a Brakke flow, denoted by $\spt\mathcal{M}$, is the set of all space-time points $X$ with $\Theta_X(\mathcal{M})\geq 1$. By upper-semicontinuity of the density the support is closed. We sometimes conflate the Brakke flow $\mathcal{M}$ and its support in the notation. For any time $t$ we let
\begin{equation}\label{eq_support_t}
M_t=\{x \in \mathbb{R}^{n+1}| (x,t)\in\spt\mathcal{M}\}.
\end{equation}
The \emph{singular set} $\mathcal{S}(\mathcal{M})$ is the set of all $X\in\spt\mathcal{M}$ such that the flow is not smooth in any space-time neighborhood of $X$. The singular set at time $t$ is defined as
\begin{equation}
S_t(\mathcal{M})=\{x\in \mathbb{R}^{n+1}| (x,t)\in \mathcal{S}(\mathcal{M})\}.
\end{equation}

\bigskip

\section{New tools for Brakke flows}\label{sec_new_tools}

In this section, we develop several new tools for Brakke flows. We consider integral Brakke flows, and we sometimes assume that their tangent flows have sufficiently small singular set:

\begin{definition}[tame points and tame flows]\label{def_tame_Brakke}
Let $\mathcal{M}$ be an integral Brakke flow. A point $X\in \spt \mathcal{M}$ is called \emph{tame point} if for each tangent flow of $\mathcal{M}$ at $X$, the time $-1$ slice is smooth with multiplicity one away from a set of $(n-1)$-dimensional Hausdorff measure zero. If every $X\in  \spt \mathcal{M}$ is a tame point, then we call $\mathcal{M}$ a \emph{tame Brakke flow}.
\end{definition}

\begin{remark}
In particular, as we will see in Corollary \ref{cor_tameness} (tameness), every ancient asymptotically cylindrical flow (see Definition \ref{def_ancient_flow}) is a tame Brakke flow (see Definition \ref{def_tame_Brakke}).
\end{remark}

Throughout this section, $\mathbb{H}$ will denote an open halfspace in $\mathbb{R}^{n+1}$ whose boundary $n$-plane contains
the origin.

\subsection{Strong maximum principle for Brakke flows}\label{subsec_strong_max}

We first recall the well-known strong maximum principle for smooth flows.

\begin{proposition}[Strong maximum principle for smooth flows]\label{Strong maximum principle for graphs}
Let $\mathcal{M}^1$ and $\mathcal{M}^2$ be smooth mean curvature flows defined in a parabolic
 ball\footnote{Recall that $P(X_0,r)=B(x_0,r)\times (t_0-r^2,t_0]$ denotes the parabolic ball with center $X_0=(x_0,t_0)$ and radius $r$.}
 $P(X_0,r)$ based at $X_0\in\spt\mathcal{M}^1\cap\spt\mathcal{M}^2$, where $r>0$ is small enough such that $\spt \mathcal{M}^1$ separates $P(X_0,r)$ into two open connected components, $\mathcal{U}$ and $\mathcal{U}'$.
If 
\begin{equation}
\spt \mathcal{M}^2 \subseteq \mathcal{U}\cup \spt \mathcal{M}^1,
\end{equation}
then there exists some $\eps>0$ such that 
\begin{equation}
\spt \mathcal{M}^2 \cap P(X_0,\eps) = \spt \mathcal{M}^1\cap P(X_0,\eps).
\end{equation}

\end{proposition}

Indeed, this easily follows by choosing $\eps>0$ small enough such that the flows are graphical, and applying the strong maximum principle for parabolic partial differential equations. We now derive a strong maximum principle, where smoothness of one of the flows is not an assumption but a conclusion.

\begin{theorem}[Strong maximum principle for Brakke flows]\label{strong_max_Brakke}
Let $\mathcal{M}^1$ be a smooth mean curvature flow defined in a parabolic ball $P(X_0,r)$ based at $X_0\in\spt\mathcal{M}^1$, where $r>0$ is small enough such that $\spt \mathcal{M}^1$ separates $P(X_0,r)$ into two open connected components, $\mathcal{U}$ and $\mathcal{U}'$. 
Let $\mathcal{M}^2$ be an integral Brakke flow defined in $P(X_0,r)$ with $X_0\in \spt \mathcal{M}^2$. If $X_0$ has density $\Theta_{X_0}(\mathcal{M}^2)<2$, and if
\begin{equation}\label{max_princ_cont}
\spt \mathcal{M}^2 \subseteq \mathcal{U}\cup \spt \mathcal{M}^1,
\end{equation}
then $X_0$ is a smooth point for $\mathcal{M}^2$ and there exists some $\eps>0$ such that 
\begin{equation}\label{conclusion_strong_max}
\spt \mathcal{M}^2 \cap P(X_0,\eps) = \spt \mathcal{M}^1\cap P(X_0,\eps).
\end{equation}
\end{theorem}

\begin{proof}
Let $\mathcal{N}$ be a tangent flow to $\mathcal{M}^2$ at $X_0=(x_0,t_0)$, and let $\Sigma$ be 
its time $-1$ slice. Since $\mathcal{M}^1$ is smooth in $P(X_0,r)$,
 from \eqref{max_princ_cont} we conclude that $\Sigma$ is contained in the closure of some open half space $\mathbb{H}$ (the boundary of which is the tangent hyperplane to $M_{t_0}^1$ at $x_0$).
 
Observe that since $\partial \mathbb{H}$ and $\Sigma$ are shrinkers,
 $\partial \mathbb{H} \cap \Sigma\neq\emptyset$.\footnote{It seems to be well-known that
 the supports of any two $n$-dimensional varifold shrinkers in $\RR^{n+1}$ must intersect, 
 but we know of no reference.
 For the case at hand (when one of the shrinkers is planar), this follows from \cite[Thm. 15.1]{white_boundary}. 
 See Theorem~\ref{brendle-connected} below for a related result.}
 Thus from the elliptic strong maximum principle from \cite{Solomon_White} and the hypothesis $\Theta_{X_0}(\mathcal{M}^2)<2$, we see that $\Sigma$ is the multiplicity one hyperplane $\partial \mathbb{H}$. Together with the local regularity theorem for Brakke flows \cite{Brakke,White_regularity}, this yields that $X_0$ is a smooth point for $\mathcal{M}^2$. 

Finally, \eqref{conclusion_strong_max} now follows from Proposition~\ref{Strong maximum principle for graphs} (strong maximum principle for smooth flows).
\end{proof}

\bigskip

\subsection{Singular shrinkers in a half space}
The goal of this subsection is to prove two Bernstein-type theorems (Theorem \ref{brendle-stable} and Theorem \ref{brendle-connected}) for varifold shrinkers in a halfspace. For smooth (two-dimensional) shrinkers such results have been proved by Brendle \cite{Brendle_genus_zero}, so the main point here is to generalize the results to the (potentially) singular setting. As opposed to the rest of this paper, in this subsection, which is more GMT heavy, we will keep the  distinction between a varifold and its support explicit.

\begin{definition}[varifold shrinker]\label{def_sing_shrinker}
A \emph{ varifold shrinker} in $\mathbb{R}^{n+1}$ is an integral $n$-varifold $V$ that has finite entropy and that
 is stationary with respect to the $F$-functional.
\end{definition}

\begin{notation}[regular and singular part]
If $A\subset \mathbb{R}^{n+1}$ is any closed set, we denote by $\reg(A)$ the set of points $x\in A$
for which there is an $r>0$ such that $A\cap \overline{B(p,r)}$ is a smooth, embedded
$n$-dimensional manifold with boundary $A\cap\partial B(p,r)$.  We let $\sing(A)=A\setminus \reg(A)$.
\end{notation}

We start with the following technical lemma.

\begin{lemma}\label{structure}
Suppose that $\Sigma$ is the support of an $n$-dimensional varifold shrinker in $\mathbb{R}^{n+1}$, and suppose that $\Hh^{n-1}(\sing(\Sigma))=0$. 
Let $S$ be a connected component of $\reg(\Sigma)\cap \mathbb{H}$, and set $M=\overline{S}$.
Then
\begin{enumerate}[\upshape (1)]
\item\label{new-small-item} $\Hh^{n}(M\cap\partial \mathbb{H})=0$.
\item\label{dense-item} $\reg(M\cap\mathbb{H})$ is dense in $M$.
\item\label{finite-area-item} $F[M]<\infty$.
\item\label{density-ratio-item}
$
   \sup_{B(p,r)\subset \mathbb{R}^{n+1}} r^{-n}\Hh^n(M\cap B(p,r))  < \infty.
$
\item\label{varifold-item} $M\cap \mathbb{H}$ is the support of a varifold shrinker in $\mathbb{H}$.\footnote{This is not as obvious as it seems. Indeed,  \eqref{varifold-item} is false if $\Sigma$ is allowed to have triple-junction singularities.  For example,
suppose that $\Sigma\cap \mathbb{H}$ is the union of $3$ minimal hypersurfaces that meet smoothly
along their common boundary.  Then the closure of one component of the regular set
will be a smooth minimal surface with nonempty boundary in $\mathbb{H}$, and thus is not the support
of any singular shrinker.}
\end{enumerate}
\end{lemma}

\begin{proof}
Let $\widehat S$ be the connected component of $\reg \Sigma$ that contains $S$.
 (In fact, it is possible to show that $\reg \Sigma$ is connected, so $\widehat S=\reg \Sigma$,
  but we do not need to use that.)
Let $Z$ be the set of points where $\widehat S$ and $\partial \mathbb{H}$ intersect tangentially.  
By Hardt-Simon \cite[Thm. 1.10]{Hardt_Simon},
$Z$ has Hausdorff dimension at most $(n-2)$ and thus 
  $\Hh^{n}(Z)=0$.  Since $(\widehat{S}\setminus Z)\cap \partial \mathbb{H}$ is an $(n-1)$-dimensional manifold,
  it also has $\Hh^n$-measure $0$.  Since $M\cap \partial \mathbb{H}$ is contained
  in the union of $Z$, $(\widehat{S}\setminus Z)\cap \partial \mathbb{H}$, and $\sing(\Sigma)$,
  we see that $\Hh^n(M\cap \partial\mathbb{H})=0$.
  
Assertion~\eqref{dense-item} is trivially true.

Assertions~\eqref{finite-area-item} and~\eqref{density-ratio-item} hold since our varifold shrinker has finite entropy (see Definition~\ref{def_sing_shrinker}).

Invoking heavy GMT machinery \cite{Wick_max,Wick_big}, one can quickly prove assertion~\eqref{varifold-item}:    Let $C$ be the connected component of $\Sigma \cap \mathbb{H}$ that contains $S$.  Since $\Hh^{n-1}(\sing \Sigma)=0$, the maximum principle in \cite{Wick_max} implies that $\reg C$ is connected, and thus that $S=\reg C$ and $C=\overline{S}\cap \mathbb{H}=M\cap \mathbb{H}$.

Alternatively, we can use the following elementary (and standard) argument: Let $Y$ be a smooth vectorfield with compact support in $\mathbb{H}$.
Let $\eps>0$.
Let $D=\sing (M\cap \mathbb{H})$.  Since $\Hh^{n-1}(\sing\Sigma)=0$,
 we can cover $\sing(\Sigma)$
by a locally finite collection of balls $B(p_i,r_i)$ 
such that 
\begin{equation}
\sum_ir_i^{n-1}<\eps.
\end{equation}
By the density bound ~\eqref{density-ratio-item} and the co-area formula, for each $i$, we can find a $\rho_i\in [r_i,2r_i]$ such that
\begin{equation}
  \Hh^{n-1}(M\cap \partial B(p_i,\rho_i)) \leq C r_i^{n-1}.
\end{equation}

By adjusting the $\rho_i$ slightly, we can assume that the spheres $\partial B(p_i,\rho_i)$ are transverse
to each other and to $\reg(\Sigma)$.

Thus if
\begin{equation}
M_\eps:=M\setminus\cup_i B(p_i,\rho_i), 
\end{equation}
then $M_\eps$ (in the space $\mathbb{H}$) is a manifold with piecewise smooth boundary, and
\begin{equation}
   \Hh^{n-1}(\mathbb{H}\cap\partial M_\eps) \leq C\eps.
\end{equation}
Thus
\begin{align}
  \int_{M_\eps} \Div_M Y\,d\Hh^{n} 
  &= -\int_{M_\eps}{\bf H}\cdot Y\,d\Hh^n + O(\Hh^{n-1}(\spt Y \cap \partial M_\eps)) \\
  &= -\int_{M_\eps}{\bf H}\cdot Y\,d\Hh^n + O(\eps).
\end{align}
Letting $\eps\to 0$ gives
\begin{equation}\label{eq_stationary}
  \int_M \Div_M Y\,d\Hh^{n} 
  = -\int_M {\bf H}\cdot Y\,d\Hh^n.
\end{equation}
Since ${\bf H}(x)= -x^\perp/2$ on $\reg M$, equation \eqref{eq_stationary} is precisely the statement 
that the varifold in $\mathbb{H}$ associated to $M\cap \mathbb{H}$ is a varifold shrinker. This finishes the proof of the lemma.
\end{proof}

\begin{theorem}[Bernstein-type theorem, first version]\label{brendle-stable}
Suppose that $M$ is a closed subset of $\overline{\mathbb{H}}$
 with the following properties:
\begin{enumerate}[\upshape (1)]
\item\label{shrinker-assumption} The regular part $\reg(M\cap\mathbb{H})$ of $M\cap\mathbb{H}$ is an
$n$-dimensional stable critical point of the $F$-functional.
\item\label{finite-F-area} $F[M]<\infty$.
\item\label{dense-assumption} $\reg(M\cap\mathbb{H})$ is dense in $M$.
\item\label{sing-assumption} $\Hh^{n-1}(\sing (M\cap \mathbb{H}))=0$. 
\item\label{boundary-assumption} $\Hh^n(M\cap \partial \mathbb{H})=0$.
\item\label{density-assumption} For every bounded subset  $\Omega$ of $\RR^{n+1}$,
\[
   \omega_\Omega:=\sup_{p\in \overline{\Omega}, \, r\le 1} r^{-n}\Hh^n(M\cap B(p,r))  < \infty.
\]
\end{enumerate}
Then $M$ is a union of flat halfplanes.
\end{theorem}

By~\cite{Wick_big}, the hypotheses imply that $\sing(M\cap\mathbb{H})$ has Hausdorff dimension at most $n-7$
and therefore that
\begin{equation}\label{codim-two}
  \Hh^{n-2}(\sing(M\cap\mathbb{H}))=0.
\end{equation}
For the application of Theorem~\ref{brendle-stable} in this paper, one could include~\eqref{codim-two}
as a hypothesis here (and in Theorem \ref{brendle-connected}) and thus~\cite{Wick_big} would not be needed.

Brendle~\cite{Brendle_genus_zero} proved Theorem~\ref{brendle-stable} when $n=2$ and $M$ is a smooth
manifold-with-boundary.
Here we extend Brendle's argument to the singular setting.
The assumptions on $M$ imply that $\reg(M\cap\mathbb{H})$ is orientable
 (see Corollary~\ref{topology-corollary} below).
Thus, the stability assumption~\eqref{shrinker-assumption} means that 
\begin{equation}\label{stability}
-\int_M f Lf\, e^{-|x|^2/4} \geq 0
\end{equation}
for every smooth
 function $f$  that is compactly supported in $\reg(M\cap \mathbb{H})$,
where $L$ is the stability operator for $F$ on $\reg(M\cap \mathbb{H})$:
\begin{equation}\label{L-formula}
L=\Delta -\frac{1}{2}x^\perp\cdot \nabla+\frac{1}{2}+|A|^2.
\end{equation} 

Observe that setting 
\begin{equation}
u(x)=\mathrm{dist}(x,\partial \mathbb{H}),
\end{equation}
we have that
\begin{equation}\label{L_of_u}
Lu=|A|^2u
\end{equation}
on $\reg M$. Hence if we could plug $u$ into the stability inequality \eqref{stability}, 
we would obtain that $|A|=0$ on $\reg M\setminus\partial \mathbb{H}$ and thus that $M$ is flat.
To justify such plugging in, we adapt the methods of Zhu~\cite[Sec. 5]{Zhu} and Brendle~\cite{Brendle_genus_zero}
 to produce suitable cutoff functions.

\begin{lemma}\label{cutoff-lemma}
Under the hypotheses of Theorem~\ref{brendle-stable}, let $Z$ be the union of $\sing(M\cap \mathbb{H})$
and $M\cap\partial\mathbb{H}$.
Given a bounded open subset $\Omega$ of $\RR^{n+1}$ and $0<\eps<1$, 
there is a smooth function $\phi=\phi_\eps:\RR^{n+1}\to [0,1]$
such that
\begin{gather}
\label{oneness}  \text{$\phi(x)=1$ if $\dist(x, Z\cap \overline{\Omega})\ge \eps$}, \\
\label{vanishing} \text{$\phi$ vanishes on an open set containing $Z\cap\overline{\Omega}$},
\end{gather}
and
\begin{equation}\label{cut-off-bound}
   \int_M (u |D\phi|  + u^2 |D^2\phi| ) < \eps,
\end{equation}
where $u(x)=\dist(x,\partial \mathbb{H})$.
\end{lemma}

\begin{proof}
Let $0<\delta<1$.
By Assumption~\eqref{boundary-assumption}, we can
cover $M\cap \partial \mathbb{H}\cap \overline{\Omega}$ by
 finitely many open balls $B_i=B(p_i,r_i)$ (where $1\le i \le k$) 
 with $p_i\in \partial \mathbb{H}$
and $r_i<1$ such that
\begin{equation}\label{first-terms}
   \sum_{i\le k} r_i^n < \delta.
\end{equation}
Thus, since $r_i=\sup_{B_i}u$, 
\begin{equation}\label{first-terms-with-u}
  \sum_{i\le k} (\sup_{B_i}u)r_i^{n-1} = \sum_{i\le k} (\sup_{B_i}u^2) r_i^{n-2} < \delta.
\end{equation}
Let $K= \sing(M\cap \mathbb{H}) \cap \overline{\Omega} \setminus \cup_{i\le k}B_i$.   
By~\eqref{codim-two},
  we can cover $K$ by finitely many balls $B_i=B(x_i,r_i)$ (where $k+1\le i \le \ell$)
with $p_i\in K$ and $r_i<1$ such that
\begin{equation}\label{last-terms}
   \sum_{i>k} r_i^{n-2} < \frac{\delta}{\sup_\Omega (u+1) +\sup_\Omega (u+1)^2}.
\end{equation}
Thus
\begin{equation}
  \sum_{i>k} (\sup_{B_i} u + \sup_{B_i}u^2) r_i^{n-2} < \delta,
\end{equation}
and therefore 
\begin{equation}\label{last-terms-with-u}
 \sum_{i>k} \left((\sup_{B_i}u) r_i^{n-1} + (\sup_{B_i}u^2) r_i^{n-2}\right) < \delta,
\end{equation}
since $r_i<1$.
Combining~\eqref{first-terms-with-u} and~\eqref{last-terms-with-u} gives
\begin{equation}\label{all-terms-with-u}
 \sum_i \left( (\sup_{B_i}u)r_i^{n-1} + (\sup_{B_i}u^2)r_i^{n-2} \right) < 3\delta.
\end{equation}
Let $\psi_i:\RR^{n+1}\to [0,1]$ be smooth functions such that 
\begin{align}
\psi_i&=0 \quad\text{on $B(p_i,2r_i)$}, \\
\psi_i&=1 \quad\text{outside of $B(p_i,3r_i)$}, \\
|D\psi_i| &\le \frac{c}{r_i}, \,\text{and}  \\
|D^2\psi_i| &\le \frac{c}{r_i^2},
\end{align}
where $c$ is a constant depending only on $n$.
Let 
\begin{align}
\psi: \RR^{n+1}\to [0,1], \qquad
\psi(x) = \inf_i \psi_i(x).
\end{align}
Note for almost every $x$ that $\psi(\cdot)=\psi_i(\cdot)$ on a neighborhood of $x$ for some $i$.
Thus
\begin{align}
  |D\psi| &\le \sum_i \frac{c}{r_i} 1_{B(x_i,3r_i)}, \\
  |D^2\psi| &\le \sum_i \frac{c}{r_i^2} 1_{B(x_i,3r_i)}
\end{align}
almost everywhere.
Now convolve $\psi$ with 
a smooth mollifier supported very near the origin to get a smooth function $\phi=\phi_\eps$.
Then, for $\delta$ small enough, $\phi$ will have properties~\eqref{oneness} and~\eqref{vanishing}, and
\begin{align}
|D\phi| &\le \sum_i \frac{c}{r_i} 1_{B(x_i,4r_i)}, \\
|D^2\phi| &\le \sum_i \frac{c}{r_i^2} 1_{B(x_i,4r_i)}.
\end{align}
Thus
\begin{align}
\int_{M\cap\Omega} (u|D\phi| + u^2|D^2\phi|)
&\le C
\sum_i\left((\sup_{B_i}u)\frac{1}{r_i}+ (\sup_{B_i}u^2) \frac{1}{r_i^2} \right) \Hh^n(B_i\cap M)  \\
&\le
C \omega\sum_i  \left(  (\sup_{B_i}u) r_i^{n-1} + (\sup_{B_i}u^2) r_i^{n-2}  \right) \\
&\le
C\omega \delta
\end{align}
by~\eqref{all-terms-with-u}, where $C<\infty$ is a dimensional constant, and $\omega=\omega_{\Omega}<\infty$ is as in Theorem~\ref{brendle-stable}. 
 This completes the proof of Lemma~\ref{cutoff-lemma}.
\end{proof}

\begin{corollary}\label{cutoff-corollary}
Let $f:\RR^{m+1}\to \RR$ be a smooth, compactly supported function
 such that $f=0$
on $\partial \mathbb{H}$.
Then, under the hypotheses of Theorem~\ref{brendle-stable} (Bernstein-type theorem, first version) we have
\begin{equation}\label{cutoff-corollary-equation}
   \int_M f Lf e^{-|x|^2/4}  \le 0.
\end{equation}
\end{corollary}

\begin{proof}
First, note that by assumptions (4) and (5) of Theorem \ref{brendle-stable} (Bernstein-type theorem, first version) $Lf$ is defined $\mathcal{H}^n$-a.e. on $M$, so the integral on the left hand side of \eqref{cutoff-corollary-equation} makes sense. Now, multiplying by a constant, we can assume that $\sup|Df|=1$ and thus (since $f=0$ on $\partial \mathbb{H}$)
that $|f|\le u \le |x|$.
Note that if $\phi$ is any smooth function on $\RR^{n+1}$, then
\begin{equation}\label{difference-bound}
\begin{aligned}
|L(\phi f) - \phi Lf|
&\le
(2|Df| + |x||f|) |D\phi| + f |D^2\phi| 
\\
&\le (2 + |x|^2) |D\phi| + u |D^2\phi|.
\end{aligned}
\end{equation}
Let $\Omega$ be a bounded open set containing the origin and the support of $f$.
  Let $\phi_\eps$ be as in Lemma~\ref{cutoff-lemma}.
Since $|Df|\leq 1$ and $|f|\leq u\leq |x|$, we see from~\eqref{difference-bound} that
\begin{equation}
   \left| \phi_\eps^2f Lf - \phi_\eps f L(\phi_\eps f)\right| \le  Q\left(u|D\phi_\eps| + u^2 |D^2\phi_\eps|\right),
\end{equation}
where $Q= 2+\sup_\Omega|x|^2$. 
Thus by Lemma~\ref{cutoff-lemma}, the integrals
\begin{equation}\label{thing-one}
\int_M \phi_\eps f L (\phi_\eps f) e^{-|x|^2/4}
\end{equation}
and
\begin{equation}
\int_M \phi_\eps^2 f L f e^{-|x|^2/4}
\end{equation}
differ by less than $Q\eps$.  By stability, the integral~\eqref{thing-one} is nonpositive, so
\begin{equation}\label{small_cutoff_diff_1}
   \int_M \phi_\eps^2 f L f e^{-|x|^2/4} \le Q\eps.
\end{equation}

On the other hand, we claim that 
\begin{equation}\label{limit_cutoff}
\lim_{\eps\rightarrow 0}\int_M \phi_\eps^2 f L f e^{-|x|^2/4} = \int_M f L f e^{-|x|^2/4} 
\end{equation}

Indeed, observe that since $M$ satisfies the shrinker equation on the regular part, the intrinsic Laplacian $\Delta$ on $M$, satisfies the pointwise estimate
\begin{equation}
|\Delta f| \leq |D^2f|+|x||Df|. 
\end{equation}
Thus, 
\begin{equation}
\int_M \left|f\left(\Delta f-\frac12 x^{\perp}\cdot \nabla f\right)\right|e^{-|x|^2/4}<\infty.
\end{equation}

The fact that
\begin{equation}
\int_M \phi_\eps^2f \left( \Delta f - \frac12 x^\perp \cdot \nabla f  \right) e^{-|x|/4}
\to
\int_M f \left( \Delta f - \frac12 x^\perp \cdot \nabla f \right) e^{-|x|/4}
\end{equation}
therefore follows from the dominated convergence theorem, while 
\begin{equation}
\int_M \phi_\eps^2 \left(\frac12 + |A|^2 \right) f^2  e^{-|x|^2/4} 
\to
\int_M  \left(\frac12 + |A|^2 \right) f^2 e^{-|x|^2/4} 
\end{equation}
follows from the monotone convergence theorem. This implies \eqref{limit_cutoff}. Together with \eqref{small_cutoff_diff_1} this proves the corollary.
\end{proof}

\begin{proof}[Proof of Theorem~\ref{brendle-stable}]
As in  Brendle~\cite{Brendle_genus_zero}, let $\psi:\mathbb{R}\to [0,1]$ be a smooth, non-increasing cutoff function with $\psi=1$ on $(-\infty,\frac12]$ and $\psi=0$ on $[1,\infty)$. For each $k>1$, taking  
\begin{equation}
f=\psi \left(\frac{\log|x|}{\log k} \right)u
\end{equation} 
we get
\begin{equation}\label{log_cut}
fLf=\psi^2uLu+\psi^2u^2\Delta \psi +2u\psi \nabla u \nabla \psi-\frac{1}{2}\psi u^2 x^{\perp}\cdot \nabla \psi.  
\end{equation}

Note that $\nabla \psi=\psi'\frac{1}{\log k}\frac{\nabla |x|}{|x|}$ and that
\begin{equation}
\Delta \psi =\psi'\frac{1}{\log k}\frac{\langle \Delta x,x \rangle}{ |x|^2}-\psi'\frac{1}{\log k}\frac{|\nabla x|^2}{|x|^2}-\psi'\frac{1}{\log k}\frac{|\nabla |x||^2}{|x|^2}+\psi''\frac{1}{\log^2 k}\frac{|\nabla |x||^2}{|x|^2}.
\end{equation} 

Observe that the last three terms of the right hand side are $O\left(\frac{1}{\log k}\frac{1}{|x|^2}\right)$,
 and that the first term is non-negative, since $\psi'\leq 0$ and $\langle \Delta x , x \rangle=-\tfrac{1}{2}|x^{\perp}|^2$ by the self-shrinker equation.

Similarly, we have that $\nabla \psi=O\left(\frac{1}{\log k}\frac{1}{|x|}\right)$, and that $x^{\perp}\cdot \nabla \psi$ is non-positive.

Putting all of this together with \eqref{log_cut}, and using the 
  facts that $u\leq |x|$, that $|\nabla u|\leq 1$, and that $\nabla \psi$ and $\Delta \psi$ are supported on $\left\{x:\sqrt{k} \leq |x| \leq k\right\}$, we  get 
\begin{equation}
fLf\geq \psi^2uLu-\frac{C}{\log k}\mathds{1}_{\{\sqrt{k}\leq x \leq k\}}.
\end{equation} 
Together with Corollary \ref{cutoff-corollary} we conclude that
\begin{equation}
0\geq \int_M fLfe^{-|x|^2/4} \geq \int_{M\cap \{x: |x|\leq \sqrt{k}\}} |A|^2u^2e^{-|x|^2/4}-\frac{C}{\log k}\int_{M\cap \{x: \sqrt{k} \leq |x|\leq k \}}e^{-|x|^2/4}.
\end{equation} 

Thus, since $F[M]<\infty$, letting $k\to\infty$ yields $|A|\equiv 0$. This proves the theorem.
\end{proof}

\begin{theorem}[Bernstein-type theorem, second version]
\label{brendle-connected}
For $i=1,2$, let $\Sigma_i$ be the support of an $n$-dimensional varifold shrinker
in $\mathbb{R}^{n+1}$.
Suppose that $\Hh^{n-1}(\sing \Sigma_i)=0$.
Let $S_i$ be a connected component of $(\reg \Sigma_i)\cap \mathbb{H}$,
and let $M_i=\overline{S_i}$.
Suppose that $S_1$ and $S_2$ do not
intersect transversely at any point.  Then
\begin{enumerate}
\item $M_1=M_2$, or
\item $\Sigma_1$ and $\Sigma_2$ are flat planes.
\end{enumerate}
\end{theorem}

\begin{proof}
As with Theorem~\ref{brendle-stable}, we adapt an argument of Brendle \cite{Brendle_genus_zero}, who showed that Theorem \ref{brendle-connected} holds in case $\Sigma_1$ and $\Sigma_2$ are smooth two-dimensional surfaces.

If $Y$ is an $n$-rectifiable set of locally finite $n$-dimensional Hausdorff
measure, we let $[Y]$ denote the associated flat chain mod $2$, as in \cite{White_Currents}.

We will now show that if $S_1$ and $S_2$ intersect, then $M_1$ and $M_2$ must coincide. To this end, note first that if $x_0\in S_1\cap S_2$, then since $S_1$ and $S_2$ do not intersect transversely by assumption, we have $T_{x_0}S_1=T_{x_0}S_2$. Write $S_1$ and $S_2$ locally as graphs of functions $u_1$ and $u_2$ over $\Omega:=T_{x_0}S_1\cap B(x_0,\eps)$, and consider the difference $w:=u_2-u_1$.  Suppose towards a contradiction that there are points  $x_{\pm}\in \Omega$ such that $\mathrm{sign}(w(x_{\pm}))=\pm$. By Hardt-Simon \cite[Thm. 1.10]{Hardt_Simon}, the set 
\[
Z=\{x\in \Omega \;|\; w(x)=|\nabla w(x)|=0\}
\]
has Hausdorff dimension at most $(n-2)$, so we can find a curve $\gamma\subset\Omega\setminus Z$ from $x_{-}$ to $x_{+}$. Then, by the intermediate value theorem, there exists a point $p\in \gamma$ such that $w(p)=0$. However, since $p\notin Z$, it follows that $S_1$ and $S_2$ intersect transversely at $p$, contradicting the assumption of the theorem. Hence, $w$ does not change sign, and the strong maximum principle thus yields that $w\equiv 0$ in $\Omega$.  Since $S_1$ and $S_2$ are connected, we conclude that $S_1=S_2$, and thus $M_1=M_2$.\\
  
Thus we may assume that $S_1$ and $S_2$ 
are disjoint.
Let $\Gamma_i = M_i\cap\partial \mathbb{H}$ and 
let $D_i=\sing(M_i)$.
Since $S_i$ is a connected component of $(\reg \Sigma_i)\cap \mathbb{H}$ and $M_i=\overline{S_i}$, it follows that $S_i=\reg(M_i)$, and hence $\reg(M_i)$  is dense in $M_i$.
 Thus, since $D_i$ has $\Hh^{n-1}$ measure $0$ (it is a subset of $\sing\Sigma_i$), 
 we see that
\begin{equation}
  \mathbb{H}\setminus M_i 
\end{equation}
has two connected components, each of which contains $M_i$ in its boundary. (See Lemma~\ref{topology-lemma} below if this is not clear.)
Let $\Omega_1$ be the component of $\mathbb{H} \setminus M_1$ that contains $S_2$,
and let $\Omega_2$ be the component of $\mathbb{H}\setminus M_2$ that contains $S_1$.
Let 
\begin{equation}
   Q = \overline{\Omega_1\cap \Omega_2}.
\end{equation}
(Intuitively, $Q$ is the closed region in $\overline{\mathbb{H}}$ between $M_1$ and $M_2$.)

Consider the class $\Cc$ of $n$-dimensional flat chains mod $2$ in $Q$ that have finite $F$-area and that have
the same boundary as $[M_1]$.
The class is nonempty since it contains $[M_1]$. By the compactness theorem for flat chains mod $2$ (see e.g. \cite[Thm. 5.1]{White_Currents}), we can find a minimizer. Furthermore, by the rectifiability theorem \cite{Ziemer,Fleming,White_rectifiabiliy}, any minimizer is rectifiable. Let $M$ be the support of a flat chain mod $2$ of least $F$-area in the class $\Cc$.

By a general barrier principle (Corollary~\ref{barrier-corollary} below), 
if $p\notin \sing(\Sigma_1)\cup \sing(\Sigma_2)$, 
then there is an $r>0$ such that $[M\cap B(p,r)]$ is $F$-area-minimizing
among {\em all} mod $2$ chains in $\overline{B(p,r)}$ with boundary $\partial[M\cap B(p,r)]$.
(We apply Corollary~\ref{barrier-corollary} to a small open subset $\Omega$ of $\mathbb{R}^{n+1}$ containing $p$
and to relatively closed regions $N_1$ and $N_2$ in $\Omega$
bounded by $\reg \Sigma_1$ and $\reg \Sigma_2$ and to the region $\overline{\mathbb{H}}\cap \Omega$.)

Consequently, 
\begin{equation}\label{quonset}
   M \setminus (D_1\cup D_2 \cup \Gamma_1)  
\end{equation}
is smooth outside of a set $X$ of Hausdorff dimension $\le n-7$, c.f. \cite{Simons}, and solves ${\bf{H}}(x)=-x^\perp/2$ classically in the complement of $X$.
By the smooth strong maximum principle, $M\setminus (X\cup D_1 \cup D_2 \cup\Gamma_1)$ cannot
touch $\partial \mathbb{H}$, so
$
   M\cap \partial \mathbb{H}
$
is contained in $X \cup D_1 \cup D_2 \cup \Gamma_1$
and thus
\begin{equation}
  \Hh^n(M\cap\partial \mathbb{H})=0.
\end{equation}

If $B$ is a ball, then
by the minimizing property of $M$ we have
\begin{equation}\label{F_mini_ball}
F[M\cap B]
\le \frac12 F[\partial (B \cap Q)].
\end{equation}

If $K$ is a compact set, then 
\begin{equation}
   C_K := \sup_{x,y\in K}\frac{e^{-|x|^2/4}}{e^{-|y|^2/4}}< \infty.
\end{equation}
Thus if $B\subset K$, then by~\eqref{F_mini_ball}, 
\begin{equation}
\Hh^n(M\cap B)
\le \frac12C_K\Hh^n(\partial (B \cap Q)).
\end{equation}

Now
\begin{align}
  \partial(B\cap Q) 
  &\subset (B \cap\partial \mathbb{H}) \cup (B\cap M_1) \cup (B\cap M_2) \cup \partial B
\end{align}
and therefore by Lemma~\ref{structure} we get
\begin{equation}
  \Hh^n(\partial (B\cap Q)) \le \omega r^n,
\end{equation}
where $\omega<\infty$ may depend on $M_1$ and $M_2$.
Thus
\begin{equation}
  \Hh^n(M\cap B(p,r)) \le \frac12 C_K\omega r^n
\end{equation}
for $B(p,r)\subset K$.

\bigskip

The minimizing property of $M$ implies that $\reg(M)$ is a stable critical point of the $F$-functional.
Hence by Theorem~\ref{brendle-stable} (Bernstein-type theorem, first version), $M$ is a $n$-dimensional halfplane and thus $\Gamma_1$ is an $(n-1)$-plane in $\partial\mathbb{H}$.
 Together with the fact that $M_1\cap\mathbb{H}$ is the support of a varifold shrinker in $\mathbb{H}$ (by Lemma ~\ref{structure}\eqref{varifold-item}), 
 applying \cite[Thm. 15.1]{white_boundary} we infer that $M_1$ is a union of halfplanes. Since the $S_1$ is connected, it follows that the $M_1$ is a halfplane. Similarly, $M_2$ is a halfplane.

If $\mathbb{H} \cap\, \reg \Sigma_1$ had another connected component $S_1'$, its closure would also be a
flat halfplane (by exactly the same argument).
Thus we have shown: $\Sigma_1\cap \mathbb{H}$ is a union of disjoint halfplanes (with common boundary).
Applying \cite[Thm. 15.1]{white_boundary} again, it follows that $\Sigma_1$ is a union
of halfplanes.  The regularity hypothesis $\Hh^{n-1}(\sing \Sigma_1)=0$ 
then implies that $\Sigma_1$ is a single plane.
The same applies to $\Sigma_2$. This concludes the proof of the theorem.
\end{proof}

\begin{corollary}[connectedness]\label{connected-corollary}
If $\;\Sigma$ is an $n$-dimensional varifold shrinker in $\mathbb{R}^{n+1}$ with 
\begin{equation}
    \Hh^{n-1}(\sing \Sigma)=0,
\end{equation}
then $\mathbb{H}\cap \reg \Sigma$ is connected for every open halfspace $\mathbb{H}$ that contains points of $\reg \Sigma$. In particular $\reg \Sigma$ is connected.
 \end{corollary}

 \begin{proof}
 This follows directly from Theorem~\ref{brendle-connected}.
 \end{proof}

 \begin{corollary}[halfspace property]\label{Half-space-thm}
 Let $\Sigma$ be the support of an $n$-dimensional varifold shrinker in $\mathbb{R}^{n+1}$ with $\Hh^{n-1}(\sing \Sigma)=0$. If $\Sigma \subseteq \{x_1\geq 0\}$, then $\Sigma=\{x_1=0\}$.
 \end{corollary}
 
 \begin{proof}
Suppose towards a contradiction that $\Sigma \nsubseteq  \{x_1= 0\}$. Then, since $\reg \Sigma\subseteq \Sigma$ is dense, we can find a halfspace $\mathbb{H}\neq \{ x_1\geq 0\}$ with $0\in\partial\mathbb{H}$ such that $\reg \Sigma \cap {\mathbb{H}}\neq \emptyset$.
We now consider Theorem \ref{brendle-connected} (Bernstein-type theorem, second version) with $\Sigma_1=\Sigma$, and $\Sigma_2=\{x_1=0\}$, and with $S_1$ any connected component of $(\reg \Sigma_1)\cap \mathbb{H}$.  Since $\Sigma_2$ is a hyperplane, but $\Sigma_1$ is not, this yields that $\reg \Sigma$ intersects $\{x_1=0\}$ transversely at some point (in ${\mathbb{H}}$). But this implies $\Sigma \nsubseteq  \{x_1\geq 0\}$, a contradiction. Hence, we have shown that $\Sigma \subseteq  \{x_1= 0\}$. Using again that our varifold shrinker satisfies $\Hh^{n-1}(\sing \Sigma)=0$ we can easily conclude that $\Sigma=  \{x_1= 0\}$.    
 \end{proof}

\begin{corollary}[transverse intersection]\label{transverse-corollary}
If $\;\Sigma_1\neq\Sigma_2$ are the supports of $n$-dimensional varifold shrinkers in $\mathbb{R}^{n+1}$ with $\Hh^{n-1}(\sing \Sigma_i)=0$,
then there is a nonempty set along which $\reg \Sigma_1$ and $\reg \Sigma_2$ intersect each other transversely.
\end{corollary}

\begin{proof}
By Corollary \ref{Half-space-thm} (halfspace property), it follows that all but at most two open halfspace $\mathbb{H}$ contain points of $\reg \Sigma_1$ and of $\reg \Sigma_2$.
The result now follows from Theorem ~\ref{brendle-connected}.
\end{proof}

\begin{corollary}[improved conclusion]\label{M-to-V-corollary}
 In the conclusion to Theorem~\ref{brendle-connected}, if $M_1=M_2$, then $\Sigma_1=\Sigma_2$.
\end{corollary}

\begin{proof}
Let $p$ be a regular point in $M_1$ and $q$ be a point in $\Sigma_1$.  Since $\reg \Sigma_1$ is connected
and dense in $\Sigma_1$, there is a path $C$ in $\Sigma_1$ joining $p$ to $q$ such that $C\setminus\{p,q\}$
is in $\reg \Sigma_1$.  Since $\Hh^{n-1}(\sing \Sigma_2)=0$ , we can choose $C$ so that $C\setminus\{p,q\}$ is disjoint
from $\sing \Sigma_2$.   By unique continuation, $C\setminus \{p,q\}$ lies in $\reg \Sigma_2$.  Thus $q\in \Sigma_2$.
Since $q\in \Sigma_1$ is arbitrary, $\Sigma_1\subseteq \Sigma_2$.  Likewise, $\Sigma_2\subseteq \Sigma_1$.
\end{proof}

\bigskip

The following barrier principle was used in the proof of Theorem~\ref{brendle-connected}:

\begin{lemma}[barrier principle]\label{barrier-lemma}
Let $\Omega$ be an $(n+1)$-dimensional Riemannian manifold
and $N$ be a closed region in $\Omega$ whose boundary is a smooth minimal
hypersurface.   For $p\in N$, the following
holds for all sufficiently small $r>0$:
\begin{quote}
 if $M$ is a hypersurface\footnote{For this paper, one should work in the class of flat chains mod $2$.
But the lemma and its proof hold in great generality (e.g., for rectifiable currents,
normal currents, flat chains mod $\nu$, etc.)}  in $\overline{B(p,r)}$ with $\partial M\subset N$ and
 if $M$ minimizes area among all hypersurfaces in $\overline{B(p,r)}$ having boundary $\partial M$,
 then $M\setminus \partial M$ is contained in $N\cap B(p,r)$.
\end{quote}
\end{lemma}

\begin{proof}
We may suppose $p\in \partial N$, as otherwise the result is trivially true.
Choose $R>0$ sufficiently small that
\begin{enumerate}[\upshape (i)]
\item $\overline{B(p,R)}$ is compact,
\item $\partial B(p,r)$ is smooth, 
    strictly mean-convex, and transverse to $\partial N$ for all $r<R$, and 
\item  $(\partial N)\cap \overline{B(p,R)}$ is strictly stable.
\end{enumerate}
By the implicit function theorem, there is an open set $U$ of $\Omega$ containing
$(\partial N)\cap B(p,R)$ and a foliation of $U$ by minimal hypersurfaces, one leaf of which
is $U\cap\partial N$.

Now suppose that $r\in (0,R)$ is sufficiently small that $\overline{B(p,r)}$ is contained in $U$.
Let $M$ be a hypersurface as in the statement of the lemma.
By the strong maximum principle \cite{Solomon_White}, the surface
 $M\setminus\partial M$ does not touch $\partial B(p,R)$, nor can it touch a leaf of the foliation
 outside of $N$.  Thus $M\setminus \partial M$ is contained in $N\cap B(p,r)$.
\end{proof}

\begin{corollary}\label{barrier-corollary}
Suppose $N_1, \dots, N_k$ are closed regions in $\Omega$ such that each boundary $\partial N_i$
is a smooth minimal hypersurface.  For each $p\in \cap N_i$, the following holds 
for all sufficiently small $r>0$:
if $\partial M\subset N\cap \overline{B(p,r)}$ and if
 $M$ minimizes area in $\overline{B(p,r)}$ among all surfaces in $\overline{B(p,r)}$ with
 boundary $\partial M$, then $M\setminus \partial M$ is contained in $N\cap B(p,r)$.
\end{corollary}

We also used the following fairly standard topological fact:

\begin{lemma}[separation]\label{topology-lemma}
Let $H$ be a simply connected $(n+1)$-dimensional manifold,
let $D$ be a closed subset with $\Hh^{n-1}(D)=0$, and let $S$ be a smooth (open) $n$-dimensional
manifold (without boundary) properly embedded in $H\setminus D$.
Then $H\setminus (D\cup S)$ is the union of two disjoint open sets $U_1$ and $U_2$,
the boundary of each of which contains $S$.  If $S$ is connected, then $U_1$ and $U_2$ are also connected.
\end{lemma}

\begin{corollary}[orientability]\label{topology-corollary} $S$ has a smooth unit normal vectorfield.
\end{corollary}

\begin{proof}[Proof of Lemma \ref{topology-lemma}]
Since $D$ is closed and $\Hh^{n-1}(D)=0$, every curve or disk in $H$ can be perturbed to
\begin{enumerate}
\item miss $D$
\item intersect $S$ transversally.
\end{enumerate}

Fix a point $p\in H\setminus (S\cup D)$.  Given $q\in H \setminus (S\cup D)$, by (i) and (ii) above, we can 
find a smooth curve $C$ in $H \setminus D$ that joins $p$ to $q$ and that is transverse to $S$. By (i) and (ii) applied to disks, and by standard intersection theory (see e.g. \cite{Samelson_orient}),  the parity of the number of points in $C\cap S$ is independent of the the choice of curve $C$ connecting $p$ and $q$.    
We put $q$ in $U_1$ or $U_2$ according to whether $C$ intersects $S$ in an even or odd number of points. 

Now suppose that $S$ is connected.  If $H \setminus (D\cup S)$ had more than two connected
components, then at least one of those components would not touch $S$.  That component
would then be a connected component of $H\setminus D$, which is impossible since
$H \setminus D$ is connected.
\end{proof}

\bigskip

%%%
%%%  Hopf Portion

\subsection{Hopf lemma without assuming smoothness}\label{sec_hopf}
As before, $\mathbb{H}\subset \mathbb{R}^{n+1}$ denotes an open halfspace whose boundary $n$-plane contains the origin. The goal of this subsection is to prove the following theorem:

\begin{theorem}[Hopf lemma without assuming smoothness]\label{mirror-theorem}
Let $\mathcal{M}^1,\mathcal{M}^2$ be integral Brakke flows defined in the parabolic ball $P((0,0),r)$.
If
\begin{enumerate}
\item  $(0,0)\in\spt \mathcal{M}^1\cap \spt\mathcal{M}^2$ is a tame point for both flows,
\item $\partial \mathbb{H}$ is \textbf{not} the tangent flow to either $\mathcal{M}^1$ or $\mathcal{M}^2$ at $(0,0)$, and
\item $\reg M^1_t\cap \mathbb{H}$ and $\reg M^2_t\cap \mathbb{H}$ are disjoint for $t\in(-r^2,0)$, 
\end{enumerate}
then $\mathcal{M}^1$ and $\mathcal{M}^2$ are smooth at $(0,0)$, with distinct tangents.

More generally, the theorem remains true if we replace $\mathbb{R}^{n+1}$
by an $(n+1)$-dimensional smooth Riemannian manifold $N$, the region $\mathbb{H}$ by an open region in $N$ 
bounded by a smooth hypersurface, and the origin by a point in this hypersurface. 
\end{theorem}

Theorem \ref{mirror-theorem} easily can be reduced to the following
theorem for the renormalized mean curvature flow:

\begin{theorem}\label{renormalized-theorem}
Suppose for $i=1,2$ that $t\in [0,\infty)\mapsto M^i_t$ is a renormalized unit-regular MCF
such that:
\begin{enumerate}[\upshape(i)]
\item\label{unit-hypothesis} Each subsequential limit of $M^i_t$ as $t\to\infty$  is nonempty and is smooth
with multiplicity one away from a closed set of $\Hh^{n-1}$ measure $0$.
\item\label{not-boundary-H-hypothesis}
   Neither $M^1_t$ nor $M^2_t$ converges (as $t\to\infty$) to the plane $\partial \mathbb{H}$.
\item\label{disjoint-hypothesis} For each $t<\infty$,
 $\reg M^1_t\cap\mathbb{H}$ and $\reg M^2_t\cap\mathbb{H}$ are disjoint.
\end{enumerate}
Then $M^1_t$ and $M^2_t$ converge to planes $P_1$ and $P_2$.
\end{theorem}

Recall also, by the local regularity theorem, that convergence of $M^i_t$ to a plane $P$ as $t\to\infty$
 is equivalent to convergence
of $M^i_{t_k}$ to $P$ for some sequence $t_k\to\infty$.

\begin{proof}[Proof of Theorem \ref{mirror-theorem} assuming Theorem \ref{renormalized-theorem}]
Let $\mathcal{M}^1,\mathcal{M}^2$ be as in Theorem \ref{mirror-theorem}, and denote the corresponding renormalized mean curvature flows around $(0,0)$ by $\hat{M}^1_t$ and $\hat{M}^2_t$. It follows from conditions (i)-(iii) of Theorem \ref{mirror-theorem} for $\mathcal{M}^1,\mathcal{M}^2$ that conditions (i)-(iii) of Theorem \ref{renormalized-theorem} hold for $\hat{M}^1_t,\hat{M}^2_t$. Therefore $\hat{M}^1_t$ and $\hat{M}^2_t$ converge to planes.  Thus, $(0,0)$ is a smooth point for both $\mathcal{M}^1$ and $\mathcal{M}^2$. Finally, by condition (iii) of Theorem \ref{mirror-theorem} and the smooth Hopf lemma (see below), this implies that the tangent flows to   $\mathcal{M}^1$ and $\mathcal{M}^2$ at $(0,0)$ are distinct planes.
\end{proof}

In the above proof, we used the following classical result:

\begin{lemma}[smooth Hopf lemma, {c.f. \cite[Lemma 6.7]{CHH}}]\label{Hopf lemma for graphs}
If $u$ and $v$ are smooth graphical solutions of the mean curvature flow in $P(0,\eps)\cap \{x_1\leq 0\}$, such that $u(0,0)=v(0,0)$ and $u< v$ in $P(0,\eps)\cap \{x_1< 0\}$, then $\tfrac{\partial u}{\partial x_1}(0,0)>\tfrac{\partial v}{\partial x_1}(0,0)$.
\end{lemma}

\begin{proof}[Proof of Theorem~\ref{renormalized-theorem}]
Let $\Cc$ be the collection of pairs $(\Sigma_1,\Sigma_2)$ such that there is a sequence $t_k\to\infty$
for which $M^i_{t_k}$ converges to $\Sigma_i$ for $i=1,2$.

Suppose $(\Sigma_1,\Sigma_2)\in \Cc$.  
Then $\Sigma_i\ne \partial\mathbb{H}$ by Hypothesis~\ref{not-boundary-H-hypothesis}, so $\reg \Sigma_i$
 intersects $\partial \mathbb{H}$ transversely
along a nonempty set (by Corollary~\ref{transverse-corollary}), and thus
\begin{equation}\label{nonempty}
    \mathbb{H}\,\cap \,\reg \Sigma_i \ne \emptyset.
\end{equation}
Since $\partial B(0,\sqrt{2n})$ is a shrinker, $\mathbb{H}\cap \reg \Sigma_i$
 and $\mathbb{H}\cap\partial B(0,\sqrt{2n})$
have nonempty intersection by Theorem~\ref{brendle-connected}, and thus
\begin{equation}\label{2-ball}
    W \cap \reg\Sigma_i\ne \emptyset,  
\end{equation}
where 
\begin{equation}\label{def-W}
W = \mathbb{H} \cap B(0, 2\sqrt{2n}).
\end{equation}
We claim that
\begin{equation}\label{transverse}
\text{$\reg \Sigma_1$ and $\reg \Sigma_2$ have no transverse intersection points in $\mathbb{H}$}.
\end{equation}
Otherwise, $\reg M^1_{t_k}$ and $\reg M^2_{t_k}$ would have  transverse
intersection points in $\mathbb{H}$ for all sufficiently large $t_k$,
 contrary to Hypothesis~\ref{disjoint-hypothesis}.
 
By Theorem~\ref{brendle-connected}(Bernstein-type theorem, second version)
and Corollary~\ref{M-to-V-corollary} (improved conclusion), either $\Sigma_1$ and $\Sigma_2$ are both planes
or $\Sigma_1=\Sigma_2$.  In the first case, we are done, so we can assume that
\begin{equation}\label{same}
\Sigma_1=\Sigma_2 \quad\text{for all $(\Sigma_1,\Sigma_2)$ in $\Cc$}.
\end{equation}

If $S$ is a subset of $\RR^{n+1}$ and $p\in S$, let $R(S,p)$ be 
the regularity scale of $S$ at $p$, i.e., the supremum of $r>0$
such that $S\cap B(p,r)$ is a smooth $n$-dimensional manifold (with no boundary in $B(p,r)$) properly
embedded in $B(p,r)$
and such that the norm of the second fundamental form at each point of $S\cap B(p,r)$ is $\le 1/r$.
(If there is no such $r$, we let $R(S,p)=0$.)
We define
\begin{equation}
   f(S) :=\sup \left\{  R(W\cap S, p):  p\in W\cap S \right\},
\end{equation}
where $W$ is as in~\eqref{def-W}.

\begin{claim} Let $\eta_i = \liminf_{t\to\infty} f(M^i_t)$.  Then
\begin{equation}
  \eta:=\min\{\eta_1,\eta_2\} >0.
\end{equation}
\end{claim}

(Using~\eqref{same}, it is not hard to show that $\eta_1=\eta_2$, but we do not need that fact.)

\begin{proof}[Proof of claim]
Choose $t_k\to\infty$ so that $f(M^1_{t_k})\to \eta_1$.
By passing to a subsequence, we can assume that $M^1_{t_k}$ converges to a limit $\Sigma$.
By~\eqref{same}, $M^2_{t_k}$ converges to the same limit $\Sigma$.
By~\eqref{2-ball},
\begin{equation}
   f(\Sigma)>0.
\end{equation}
By the smooth convergence of $M^1_{t_k}$ to $\Sigma$ at the regular points of $\Sigma$
(which follows from the local regularity theorem),
\begin{equation}
   \lim f(M^1_{t_k}) = f(\Sigma).
\end{equation}
Thus $\eta_1>0$. Likewise, $\eta_2>0$.  This completes the proof of the claim.
\end{proof}

Choose $T$ so that $f(M^i_t) > \eta/2$ for all $t\ge T$ and for $i=1,2$.
For $t\ge T$, let
\begin{equation}\label{phi_sup_def}
    \phi(t) = \sup \left\{ \dist(x, M^2_t): 
    \text{$x\in W\cap M^1_t$
      and $R(W\cap M^1_t, x)\ge \eta/2$} \right\}.
\end{equation}
Note that, by definition, for every $x\in M^1_t\cap W$ one has  $R(W\cap M^1_t, x) \leq \dist(x,\partial W)$, so the supremum will be attained at some (not necessarily unique) point  $x_t$, with $\dist(x_t,\partial W) \geq \eta/2$. Moreover, by~\eqref{same}, we have $\lim_{t\to\infty}\phi(t)=0$.   Now, for $k\in \mathbb{N}$, choose $t_k\ge k$ so that
\begin{equation}\label{time-selection}
     \phi(t_k) \ge (1-k^{-1}) \sup_{t\ge t_k} \phi(t).
\end{equation}

By passing to a subsequence, we can assume that $M^1_{t_k}$ converges to a limit $\Sigma$ 
and that $x_{t_k}$ converges to a point $p\in W\cap \reg \Sigma$.  
By~\eqref{same}, $M^2_{t_k}$ also converges to $\Sigma$.  It also follows that
\begin{equation}\label{smoothy}
\begin{gathered}
   \text{$\{ (x,t): t\in \RR, \, x\in M^i_{t_k+t} \}$ converges to $\Sigma\times\RR$, with smooth} \\
   \text{convergence on compact subsets of 
$
   (\RR^{n+1}\setminus \sing \Sigma) \times \RR.
$}
\end{gathered}
\end{equation}

Choose open subsets $\Omega_k$ of $\mathbb{H}\cap\reg \Sigma$ such that
\begin{gather}
p\in \Omega_1 \subset \Omega_2 \subset \dots, \\ 
\Omega_k\subset\subset \mathbb{H}\cap\reg \Sigma, \\
\cup_k\Omega_k=\mathbb{H}\cap\reg \Sigma.
\end{gather}
Since $\mathbb{H}\cap\reg \Sigma$ is connected (by Corollary~\ref{connected-corollary}), we can choose the $\Omega_k$ to be connected.

Let $\nu$ be a unit normal vectorfield on $\reg \Sigma$. 
 (Such a vectorfield exists
by Corollary~\ref{topology-corollary}.)
By the smooth convergence of $M^i_{t_k}$ on the regular portion of $\Sigma$, 
we can, by passing to a subsequence, assume that
there are functions 
\begin{equation}
  u^1_k, u^2_k: \Omega_k \times [-k,k] \to  \RR
\end{equation}
such that
\begin{align}
   \{ q + u^1_k(q,t)\nu(q): q\in \Omega_k\} &\subset M^1_{t_k+t}, \\
   \{ q + u^2_k(q,t)\nu(q): q\in \Omega_k\} &\subset M^2_{t_k+t},
\end{align}
and such that $u^1_k$ and $u^2_k$ converge to $0$ smoothly on compact
subsets of 
\begin{equation}
   (\mathbb{H} \cap \reg\Sigma) \times\RR.
\end{equation}

By relabelling, we can assume that $u^2_k>u^1_k$ on $\Omega_k$. Consider the normalized difference
\begin{equation}
   w_k:=\frac{u^2_k - u^1_k}{ u^2_k( p_{k},0) - u^1_k(p_{k},0)},
\end{equation}
where  $p_k\in \Omega_k$ is such that $p_k+u^1_k(p_k,0)\nu(p_k)=x_{t_k}$. Since $M^1_t$ and $M^2_t$ evolve by renormalized mean curvature flow, the function $w_k$ satisfies a linear second-order parabolic equation. The coefficients depend on $u^1_k$ and $u^2_k$, but for $k\to \infty$ converge smoothly on compact subsets of $(\mathbb{H} \cap \reg\Sigma)\times\mathbb{R}$ to the coefficients of the operator $\partial_t - L$,
where $L$ is  the stability operator~\eqref{L-formula} for the $F$-functional, since the renormalized mean curvature
flow is the gradient flow for $F$. Moreover, by \eqref{time-selection} after passing to another subsequence we have
\begin{equation}
\sup_{t\in [0,k]}w_k(p,t)\leq 2,\qquad \textrm{where } p=\lim_{k\to \infty} x_{t_k}.
\end{equation}
Therefore, applying the Harnack inequality \cite[Corollary 7.42]{Lieberman} we infer that $w_k$ is uniformly bounded on compact subsets of $(\mathbb{H} \cap \reg\Sigma)\times\mathbb{R}$. Together with standard derivative estimates this implies that after passing to another subsequence $w_k$ converges smoothly on compact subsets to a nonnegative solution
\begin{equation}
    \phi: (\mathbb{H}\cap \reg \Sigma) \times \mathbb{R}\to [0,\infty)
\end{equation}    
of the  linearized mean curvature flow equation
\begin{equation}\label{renormalized-linear1}
 ( \partial_t - L)\phi=0.
\end{equation}
Since  $w_k(p_{k},0)=1$, we see that $\phi(p,0)=1$, so $\phi>0$ by the strong maximum principle. Morevoer, using again~\eqref{time-selection}, we see that
\begin{equation}
\sup_{t\geq 0} \phi(p,t) < \infty.
\end{equation}

The existence of such a $\phi$ implies that $\mathbb{H} \cap \reg \Sigma$
 is stable for the $F$-functional.  
(If that is not clear, see Lemma~\ref{stability-lemma} below.)  
Thus, by Theorem~\ref{brendle-stable} (Bernstein-type theorem, first version), $\Sigma\cap \overline{\mathbb{H}}$ is a union of halfplanes. By assumption (i) we have $\mathcal{H}^{n-1}(\sing\Sigma)=0$, so $\Sigma \cap \overline{\mathbb{H}}$ is a single halfplane $P\cap \overline{\mathbb{H}}$, where $P$ is a hyperplane in $\mathbb{R}^{n+1}$. Applying Theorem~\ref{brendle-connected}(Bernstein-type theorem, second version)
and Corollary~\ref{M-to-V-corollary} (improved conclusion) we conclude that $\Sigma=P$.
\end{proof}

\begin{lemma}\label{stability-lemma}
Let $N$ be a smooth, connected Riemannian manifold without boundary
and let $L$ be a second-order, self-adjoint, linear elliptic operator on $N$.
Suppose that there is a smooth, everywhere positive solution 
\begin{equation}
   \phi: N\times [0,\infty) \to \RR
\end{equation}
of the parabolic equation $\partial_t\phi = L\phi$ such that
\begin{equation}
   \sup_{t\ge 0} \phi(p,t)<\infty
\end{equation}
for some point $p\in N$.
Then $N$ is stable for $L$
 in the following sense: if $U$ is a relatively open subset of $N$ with compact
closure and if $\lambda$ is the first Dirichlet eigenvalue of $L$ on $U$,
then $\lambda\ge 0$.
\end{lemma}

\begin{proof}
Let $U\subset\subset N$ be a connected open set containing $p$.
It suffices to show (for every such $U$) that the first eigenvalue $\lambda$ of $L$ on $U$
is nonnegative.
Let
$f$ be the corresponding 
 eigenfunction.  
Then $f$ is nonzero at all points of $U$.
By multiplying by a constant, we can assume that $0<f \le \phi(\cdot,0)$.
Thus by the maximum principle,
\begin{equation}
    e^{-\lambda t} f(\cdot) \le \phi(\cdot,t) 
\end{equation}
for all $t\ge 0$ (since $e^{-\lambda t}f(\cdot)$ solves the same parabolic equation).
In particular,
\begin{equation}
   e^{-\lambda t}f(p) \le \phi(p,t) \le \sup_{t\ge 0}\phi(p,t) < \infty.
\end{equation}
for all $t\ge 0$.
Hence $\lambda \ge 0$.
\end{proof}

%   HERE  

\bigskip

\section{Coarse properties of ancient asymptotically cylindrical flows}\label{sec_coarse_properties}

\subsection{Partial regularity}\label{sec_partial_reg}

In this subsection, we prove a partial regularity result for ancient asymptotically cylindrical flows.

Recall that the \emph{entropy} of a Radon measure $\mu$ in $\mathbb{R}^{n+1}$ is defined as the supremum of the Gaussian area over all centers and scales, namely
\begin{equation}
\textrm{Ent}[\mu]=\sup_{y\in\mathbb{R}^{n+1},\lambda>0} \int \frac{1}{(4\pi\lambda)^{n/2}}e^{-\tfrac{|x-y|^2}{4\lambda}}\, d\mu(x),
\end{equation}
and the \emph{entropy} of an integral Brakke flow $\mathcal{M}=\{\mu_t\}_{t\in I}$ is defined as
\begin{equation}
\textrm{Ent}[\mathcal{M}]=\sup_{t\in I} \textrm{Ent}[\mu_t].
\end{equation}

\begin{proposition}[entropy bound]\label{prop_ent_bound}
Any ancient asymptotically cylindrical flow $\mathcal{M}$ satisfies
\begin{equation}
\textrm{Ent}[\mathcal M]\leq \textrm{Ent}[S^{n-1}\times \mathbb{R}].
\end{equation}
\end{proposition}

\begin{proof}
For any $X_{0}=(x_0,t_0)\in \mathbb{R}^{n+1}\times\mathbb{R}$ and $\rho\in(0,\infty)$ consider the cutoff function
\begin{equation}
\chi_{({X_0},\rho)}(x,t) = \left( 1 - \frac{|x-x_0|^2+2n(t-t_0)}{\rho^2}\right)_+^3.
\end{equation}
Then we have the localized monotonicity inequality
\begin{equation}\label{eq_huisken_mon_loc}
\frac{d}{dt} \int \rho_{X_0}(x,t)\chi_{({X_0},\rho)}(x,t)  \, d\mu_t(x) \leq -\int \left|{\bf H}(x,t)-\frac{(x-x_0)^\perp}{2(t-t_0)}\right|^2 \rho_{X_0}(x,t)\chi_{({X_0},\rho)}(x,t) \, d\mu_t(x).
\end{equation}
Since $\mathcal{M}$ asymptotically cylindrical, and since $X_0$ and $\rho$ are arbitrary, it follows from \eqref{eq_huisken_mon_loc} that $\mathcal{M}$ has bounded area ratios. Hence, we can apply the monotonicity formula \eqref{eq_huisken_mon} without cutoff function, and use again the assumption that $\mathcal{M}$ is asymptotically cylindrical, to conclude that $\textrm{Ent}[\mathcal M]\leq \textrm{Ent}[S^{n-1}\times \mathbb{R}]$.
\end{proof}

\begin{lemma}[stationary cones]\label{smooth_cones}
Let $\mu$ be an integral $3$-rectifiable Radon measure in $\mathbb{R}^4$ with $\textrm{Ent}[\mu]<3/2$. If the associated varifold $V_\mu$ is a stationary cone, then $\mu=\mathcal{H}^3\llcorner P$ for some flat plane $P$.
\end{lemma}
\begin{proof}
This is explained in  \cite[Lem. 4.1 and Footnote 1]{BW_sharp_bounds}. For the reader's convenience, and since part of this statement  is only remarked upon in \cite{BW_sharp_bounds} as a footnote, we will include the proof here as well. 

Let us first show that there are no non-flat $2$-dimensional stationary cones $V_2$ in $\mathbb{R}^3$ with entropy less than $3/2$. Letting $x\in \mathrm{spt}(V_2)-\{0\}$, the tangent cone to $V_2$ at $x$ splits off a line in the $x$ direction, and therefore is of the form $V_1\times \mathbb{R}$, where $V_1$ is a stationary cone in $\mathbb{R}^2$. We have $\mathrm{Ent}[V_1]<3/2$, so $V_1$ is a line and $x$ is a regular point. Therefore, the link of the cone $V_2$ is a smooth closed geodesic in $S^2$, and since such a geodesic has to be a great circle, it follows that $V_2$ is the plane. 

In the three dimensional case, given $\mu$, take $x\in \mathrm{spt}(\mu)-\{0\}$. The tangent cone to $V_{\mu}$ at $x$ splits off a line, and so it is of the form $V_2\times \mathbb{R}$ for a stationary $V_2$ with $\mathrm{Ent}[V_2]<3/2$. As shown above, $V_2$ must be flat, and so $x$ is a regular point. Thus, the link of $V_\mu$ is a smooth closed minimal surface $\Sigma$ in $S^3$. By the entropy assumption we have
\begin{equation}
\frac{\mathrm{Area}(\Sigma)}{\mathrm{Area}(S^2)}=\frac{\mathrm{Ent}[\mu]}{\mathrm{Ent}[\mathbb{R}^3]}<\frac{3}{2}
\end{equation}
On the other hand, it follows from the resolution of the Willmore conjecture \cite[Thm. B]{MN_Willmore} that if $\Sigma$ is not an equitorial sphere then 
\begin{equation}
\frac{\mathrm{Area}(\Sigma)}{\mathrm{Area}(S^2)}\geq \frac{2\pi^2}{4\pi}=\frac{\pi}{2}.
\end{equation}
Thus, $\Sigma$ is an equatorial $2$-sphere, and $\mu=\mathcal{H}^3\llcorner P$ for some flat plane $P$.     
\end{proof}

\begin{theorem}[partial regularity]\label{part_reg_thm}
Let $\mathcal{M}$ be an ancient asymptotically cylindrical flow in $\mathbb{R}^{n+1}$. Then either $\mathcal{M}$ is a round shrinking cylinder, or:
\begin{enumerate}
\item The space-time singular set $\mathcal{S}(\mathcal{M})$ has a parabolic Hausdorff dimension at most $n-2$.
\item The singular set $S_t(\mathcal{M})$ at time $t$ has Hausdorff dimension at most $n-3$ at all times, and at most $n-4$ at almost every time.
\item All self-shrinkers $\Sigma$ appearing as the time $-1$ slice of a tangent flow of $\mathcal{M}$ are smooth away from a set of Hausdorff dimension $n-4$.    
\end{enumerate} 
\end{theorem}

\begin{proof}
Considering the stratification of the singular set as in \cite{White_stratification} it is enough to rule out certain static, quasi-static and shrinking tangent flows.

By Proposition \ref{prop_ent_bound} we have $\textrm{Ent}[\mathcal M]\leq \textrm{Ent}[S^{n-1}\times \mathbb{R}]<3/2$. Moreover, we may assume that the entropy of any tangent flow $\hat{\mathcal{M}}_X$ at any $X\in\mathcal{M}$ is strictly less than $\textrm{Ent}[S^{n-1}\times\mathbb{R}]$, since otherwise, by the equality case of Huisken's monotonicity formula, $\mathcal{M}$ would be a family of round shrinking cylinders.

If a tangent flow $\hat{\mathcal{M}}_X$ at a singular point $X\in \mathcal{M}$ is static or quasi-static, then its time $-1$ slice is a stationary cone. Since $X$ is singular, by Lemma \ref{smooth_cones} this cone can split off at most $n-4$ lines; hence $\hat{\mathcal{M}}_X$ can only contribute to the $(n-4)$-stratum of $S_t(\mathcal{M})$ and to the $(n-2)$-stratum of $\mathcal{S}(\mathcal{M})$.

Consider now the case where a tangent flow $\hat{\mathcal{M}}_X$ is a non-flat self-shrinker, and denote its time $-1$ slice by $\Sigma$. Since $\Sigma$ has entropy less than $3/2$, and since $\Sigma$ is stationary with respect to a metric conformal to the Euclidean metric, any tangent cone to $\Sigma$ is a stationary cone with entropy less than $3/2$. Thus, using Lemma \ref{smooth_cones} and dimension reduction it follows that $\Sigma$ is smooth away from a set of Hausdorff dimension $n-4$. Moreover, since $\Sigma$ has entropy strictly less than $S^2$, by the classification of two-dimensional low entropy shrinkers \cite[Cor. 1.2]{BW_topological_property}, $\Sigma$ cannot splits off $n-2$ lines.

Combining the above facts with \cite[Thm. 9]{White_stratification}, the remaining assertions of the theorem follow.
 \end{proof}
 
\begin{corollary}[tameness]\label{cor_tameness}
Any ancient asymptotically cylindrical flow is a tame Brakke flow (see Definition \ref{def_tame_Brakke}).
\end{corollary}

\begin{proof}
This follows from Theorem \ref{part_reg_thm} (partial regularity).
\end{proof}
 
 \begin{corollary}[extinction time]
If $\mathcal{M}$ is an ancient asymptotically cylindrical flow then exactly one of the following happens:\footnote{We will show later that case (iv) actually cannot occur.}
\begin{enumerate}
\item[(i)] $T_e(\mathcal{M})=\infty$, i.e. the flow $\mathcal M$ is eternal, or
\item[(ii)] $T_e(\mathcal{M})<\infty$ and $\mathcal{M}$ is a round shrinking cylinder, or
\item[(iii)] $T_e(\mathcal{M})<\infty$ and there exist a non-empty set $E\subseteq \mathbb{R}^{n+1}$ of Hausdorff dimension at most $n-3$ such that $\Theta_{(x,T_e(\mathcal{M}))}(\mathcal{M})\geq 1$ if and only if $x\in E$, or
\item[(iv)] $T_e(\mathcal{M})<\infty$ and for every $R<\infty$ there exist $T(R)<T_e(\mathcal{M})$ such that $B(0,R)\cap M_t=\emptyset$ for every $t\in(T(R),T_e(\mathcal{M})]$.
\end{enumerate}
\end{corollary}

\begin{proof}
Suppose that (i), (ii) and (iv) do not hold. Then, the set of points $X$ such that $\Theta_{X}(\mathcal{M})\geq 1$ has an accumulation point $(\bar{x},T_e(\mathcal{M}))$. By upper semi-continuity of the density it follows that $\Theta_{(\bar{x},T_e(\mathcal{M}))}\geq 1$. Define $E\subseteq \mathbb{R}^{n+1}$ as the set of points $x$ such that $\Theta_{(x,T_e(\mathcal{M}))}(\mathcal{M})\geq 1$. Since $\bar{x}\in E$, the set $E$ is non-empty. Moreover, by the definition of a regular point (see Section \ref{sec_prelim}) all  $x\in E$ are singular, since there is no $\eps>0$ such that the flow is smooth in $B(x,\eps)\times [T_e(\mathcal{M})-\eps^2,T_e(\mathcal{M})+\eps^2]$. Hence by Theorem \ref{part_reg_thm} (partial regularity) the Hausdorff dimension of $E$ is at most $n-3$.
\end{proof}

\subsection{Enclosed domain and comparison}\label{sec_enc_domain}

The partial regularity result above allows us to effectively define the domain enclosed by the flow. To this end, view $\mathcal{M}$ as a subset of space-time $\mathbb{R}^{n+1,1}=\mathbb{R}^{n+1}\times\mathbb{R}$ and define $\mathcal{U}$ as the path connected component of $\mathbb{R}^{n+1,1}\setminus \mathcal{M}$ that contains the solid asymptotic cylinder.

 \begin{proposition}[separation]\label{sides}
 Let $X=(x,t)\in \mathcal{M}$ be a regular point and let $r>0$ be small enough such that $\mathcal{M}$ is smooth in the two-sided parabolic ball $B(X,r)=B(x,r)\times (t-r^2,t+r^2)$ and such that $\mathcal{M}$ divides $B(X,r)$ it into two regions $\Omega_+,\Omega_{-}$. Then exactly one of $\Omega_{+}$ and $\Omega_{-}$ is in $\mathcal{U}$.
 %In particular, $\mathcal{U}=\mathrm{Int}(\mathrm{Cl}(\mathcal{U}))$.
 \end{proposition}

\begin{proof}
This is obviously true for the round shrinking cylinder. Assume now $\mathcal{M}$ is not the round shrinking cylinder. Then, by Theorem \ref{part_reg_thm} (partial regularity), the singular set $\mathcal{S}(\mathcal{M})$ of $\mathcal{M}$ has parabolic Hausdorff dimension at most $n-2$, in particular Euclidean Hausdorff dimension at most $n-2$.  Thus, there exists a set $\mathcal{M}'\subseteq \mathcal{M}$ such that $\mathcal{M}'$ is a smooth open $(n+1)$-dimensional manifold, such that $\mathrm{Cl}(\mathcal{M}')=\mathcal{M}$, and such that $\mathcal{M}-\mathcal{M}'$ has Euclidean Hausdorff dimension at most $n-2$. Hence, any space-time curve connecting $\Omega_+$ and $\Omega_-$ can be perturbed such that it avoids $\mathcal{M}-\mathcal{M}'$, and likewise any smooth disc bounding a closed curve can be perturbed such that it avoids $\mathcal{M}-\mathcal{M}'$ and meets $\mathcal{M}'$ transversally. By a standard intersection theory argument (c.f. \cite[Thm. 4.3]{BW_sharp_bounds}, \cite{Samelson_orient})) this implies that at most one of $\Omega_\pm$ is in $\mathcal{U}$.

Similarly, thanks again to the small singular set (and the small singular set of tangent flows), there exists a path $\gamma \subseteq \mathcal{M}$, connecting $X$ to the asymptotic cylinder and passing solely through regular points. Therefore, at least one of $\Omega_\pm$ is in $\mathcal{U}$.
\end{proof}

We set
\begin{equation}
 U_t=\{x\in \mathbb{R}^{n+1}\, |\, (x,t)\in \mathcal{U}\},
\end{equation}
and call
\begin{equation}\label{eq_enclosed_t}
K_t=U_t\cup M_t
\end{equation}
the \emph{domain enclosed by $M_t$}.

\begin{proposition}[comparison]\label{containment}
Suppose $\{N_t\}_{t\in [t_0,t_1]}$ is a smooth mean curvature flow of compact hypersurfaces. If $N_{t_0}\subseteq U_{t_0}$ (resp. $N_{t_0}\subseteq \mathbb{R}^{n+1}\setminus K_{t_0}$), then $N_{t}\subseteq U_{t}$ (resp. $N_{t}\subseteq \mathbb{R}^{n+1}\setminus K_{t}$) for all $t\in[t_0,t_1]$.
\end{proposition}

\begin{proof}
This follows directly from the avoidance principle \cite[Sec. 10]{Ilmanen_book} applied to $M_t$ and $N_t$.
\end{proof}

\subsection{Almost self-similarity backwards in time}\label{sec_sim_back_in_time}

We start with the following structural result for low entropy self-shrinkers, which is a refinement of the results in \cite[Sec. 4]{BW_sharp_bounds}.

\begin{proposition}[structure of low entropy shrinkers]\label{uniform_r_high_d}
For every $n\geq 3$ and $\rho> 2n$ there exist constants $R_0=R_0(n,\rho)\in [2\rho,\infty)$ and $c_0=c_0(n,\rho)>0$, such that every (potentially singular) $n$-dimensional self-shrinker $\Sigma\subseteq \mathbb{R}^{n+1}$ with $\mathrm{Ent}[\Sigma] \leq \mathrm{Ent}[S^{n-1}\times\mathbb{R}]$  is one of the following:
\begin{enumerate}
\item a round shrinking cylinder $S^{n-1}\times\mathbb{R}$,
\item a compact self-shrinker (potentially singular) contained in $B(0,R_0)$,
\item a noncompact self-shrinker with the property that there exists some $y\in B(0,R_0)\cap \Sigma$ and $s\in [2\rho,R_0]$, such that $\Sigma$ is smooth in  $B(y,s)$ with regularity bounded from below by $c_0$, and such that $\Sigma$ separates $B(y,s)$ into two connected components, both of which contain a ball of radius $\rho$.
\end{enumerate}
\end{proposition}

\begin{proof} 
Since $\textrm{Ent}[S^k]<3/2$ for $k\geq 2$, it is enough to show by induction on $n\geq 2$ that the assertion holds for all $n$-dimensional shrinkers $\Sigma$ satisfying the assumptions $\mathrm{Ent}[\Sigma] \leq \mathrm{Ent}[S^{n-1}]$ and $\mathrm{Ent}[\Sigma]<3/2$. For $n=2$ by \cite[Cor. 1.2]{BW_topological_property} the only such two-dimensional shrinkers are the flat plane and the round shrinking sphere.

Given $n\geq 3$, assume that $\Sigma$ is a non-compact $n$-dimensional shrinker with $\mathrm{Ent}[\Sigma] \leq \mathrm{Ent}[S^{n-1}\times\mathbb{R}]$. Taking the corresponding self-shrinking flow $\mathcal{N}$, it is easy to see that $C=\{x\in \mathbb{R}^{n+1}\;|\; \Theta_{(x,0)}(\mathcal{N})\geq 1\}$ forms a nontrivial set-theoretic cone. Taking any $0\neq p\in C$, any tangent flow $\hat{\mathcal{N}}_{(p,0)}$ at $(p,0)$ splits off a line. Denoting by $\Sigma'$  the time $(-1)$ slice of $\hat{\mathcal{N}}_{(p,0)}$, we therefore have that $\Sigma'=\Sigma'_{n-1}\times \mathbb{R}$, where $\Sigma'_{n-1}$ is an $(n-1)$-dimensional shrinker with $\mathrm{Ent}[\Sigma'_{n-1}] \leq \mathrm{Ent}[S^{n-1}]$.

By the induction hypothesis, one of  (i) - (iii) applies to $\Sigma'_{n-1}$. By the entropy bound, $\Sigma'_{n-1}$ can not be a round shrinking cylinder $S^{n-2}\times\mathbb{R}$. If $\Sigma'_{n-1}$ is compact, then by \cite[Thm. 2.5]{Zhu} it is the round shrinking sphere $S^{n-1}$. The equality case of the monotonicity formula therefore implies that $\Sigma$ is the round shrinking cylinder $S^{n-1}\times\mathbb{R}$. Finally, if $\Sigma'_{n-1}$ satisfies (iii), the self-similarity of $\Sigma$ (and $\hat{\mathcal{N}}_{(p,0)}$ being a tangent flow) implies the existence of $R_0(\Sigma)$ and $c_0(\Sigma)$ as in (iii), which at this stage may depend on $\Sigma$.

Now, suppose towards a contradiction there are compact $\Sigma_i$ that are not contained in $B(0,i)$. Since every shrinker contains a point in $B(0,\sqrt{2n})$ and since the entropy is bounded, we can apply Allard's compactness theorem \cite{Allard} to pass to a subsequential limit $\Sigma$, which will be a noncompact shrinker with $\mathrm{Ent}[\Sigma] \leq \mathrm{Ent}[S^{n-1}\times\mathbb{R}]$. This shrinker can not be the cylinder, as the cylinder is isolated in the space of shrinkers by \cite{CIM}. Hence, $\Sigma$ is of type (iii). But by comparison with spheres of radius $2n$ (using \cite[Prop. 4.3]{BW_sharp_bounds}) for $i$ large enough this contradicts the fact that $\Sigma_i$ becomes extinct in a point at time $0$.

Similarly, suppose towards a contradiction there is a sequence of $\Sigma_i$ of type (iii) such that for every $i$, we have $R_0(\Sigma_i)\geq i$ (respectively $c_0(\Sigma)\leq i^{-1}$). Pass to a subsequential limit $\Sigma$. Then, arguing as above, $\Sigma$ can not be the cylinder or compact. Hence, $\Sigma$ is of type (iii). But then there are $c>0$, $R<\infty$, $y\in \Sigma\cap B(0,R)$ and $s\leq R$ such that $\Sigma$ is smooth in $B(y,s)$ with regularity scale bounded below by $c$, and such that $\Sigma$ separates $B(y,r)$ into two components, both of which contain a ball of radius $4\rho$. This gives a contradiction for $i$ large enough, and concludes the proof of the proposition.
\end{proof}

\begin{convention}\label{conv_eps}
We now fix a small positive constant $\eps < (100R_0)^{-1}$ where $R_0$ is the constant from Proposition \ref{uniform_r_high_d} (structure of low entropy shrinkers) corresponding to $\rho=3n$. We furthermore assume that $\eps$ is small enough so that Proposition \ref{thm_finding_sim} (almost selfsimilarity) from below applies. 
\end{convention}

Given a Brakke flow $\mathcal{M}$, a point $X=(x,t)\in \mathcal{M}$, and a scale $r>0$, we denote by
\begin{equation}
\mathcal{M}_{X,r}=\mathcal{D}_{1/r}(\mathcal M -X)
\end{equation}
the Brakke flow which is obtained from $\mathcal{M}$ by shifting $X$ to the origin and parabolically rescaling by $1/r$.

The following definition captures how close $\mathcal{M}_{X,r}$ is to a selfsimilar flow as in Proposition \ref{uniform_r_high_d} (structure of low entropy shrinkers). Since the selfsimilar solutions can be singular, we measure closeness in a hybrid way (strong closeness in smooth regions and weak closeness in singular regions).

\begin{definition}[almost selfsimilarity]\label{def_alm_selfsim}
We say that a Brakke flow $\mathcal{M}$ is
\begin{itemize}
\item \emph{$\varepsilon$-cylindrical around $X$ at scale $r$}  if $\mathcal{M}_{X,r}$ is $\varepsilon$-close in $C^{\lfloor1/\varepsilon \rfloor}$ in $B(0,1/\varepsilon)\times [-2,-1]$ to the evolution of a round cylinder with radius $\sqrt{-2(n-1)t}$ and center at the origin.
\item \emph{$\varepsilon$-compact around $X$ at scale $r$} if there is some compact shrinker $\Sigma$ with $\mathrm{Ent}[\Sigma] \leq \mathrm{Ent}[S^{n-1}\times\mathbb{R}]$, such that for all $t\in [-2,-1]$ we have that $(\mathcal{M}_{X,r})_t\cap B(0,1/\varepsilon) \subseteq B(\sqrt{-t}\Sigma,\eps)$.
\item \emph{$\varepsilon$-separating around $X$ at scale $r$} if there are a noncompact shrinker $\Sigma$ with $\mathrm{Ent}[\Sigma] \leq \mathrm{Ent}[S^{n-1}\times\mathbb{R}]$, a point $y$ and a radius $s$ as in item (iii) of Proposition \ref{uniform_r_high_d} (structure of low entropy shrinkers), such that for all $t\in [-2,-1]$ we have that $(\mathcal{M}_{X,r})_t\cap B(0,1/\varepsilon) \subseteq B(\sqrt{-t}\Sigma,\eps)$ and that $\mathcal{M}_{X,r}$ is $\varepsilon$-close in $C^{\lfloor1/\varepsilon \rfloor}$ in $\{(x',t')\in\mathbb{R}^{n+1}\times [-2,-1]\;|\; \;x'\in B(\sqrt{-t'}y,\sqrt{-t'}s)\}$ to $\sqrt{-t}\Sigma$.
\end{itemize}
We say that  $\mathcal{M}$ is \emph{$\varepsilon$-selfsimilar around $X$ at scale $r$} if it is $\varepsilon$-cylindrical, $\varepsilon$-compact, or $\varepsilon$-separating around $X$ at scale $r$.
\end{definition}

Given any $X=(x,t)\in\mathcal{M}$, we analyze the solution around $X$ at the diadic scales $r_j=2^j$, where $j\in \mathbb{Z}$. 

\begin{theorem}[almost selfsimilarity]\label{thm_finding_sim}
For any small enough $\eps>0$ and any ancient asymptotically cylindrical flow $\mathcal M$ that is not the round shrinking cylinder the following holds:
\begin{enumerate}
\item There exists a positive integer $N=N(\eps)<\infty$, such that for any integers $j_1,j_2$ with $j_2-j_1\geq N$, and any $X\in \mathcal{M}$, there is some integer $j$ with $j_1\leq j \leq j_2$ such that $\mathcal{M}$ is $\varepsilon$-selfsimilar around $X$ at scale $r_j$.
\item For any $X\in \mathcal{M}$ there exists a largest integer $J(X)\in\mathbb{Z}$ such that
\begin{equation*}
\textrm{$\mathcal M$ is not $\varepsilon$-cylindrical around $X$ at scale $r_j$ for all $j<J(X)$},
\end{equation*}
and for all $\eps'>0$ there exists a positive integer $N'=N'(\eps,\eps')<\infty$ such that
\begin{equation*}
\textrm{$\mathcal M$ is $\varepsilon'$-cylindrical around $X$ at scale $r_j$ for all $j\geq J(X)+N'$}.
\end{equation*} 
\end{enumerate}
\end{theorem}

\begin{proof}
For integral Brakke flows $\mathcal{M}=\{\mu_t\}$ and $\mathcal{N}=\{\nu_t\}$ with bounded entropy we consider their pseudo-distance
 \begin{equation}
d_{\mathcal{B}}(\mathcal{M},\mathcal{N})=\sum_{i,j}\frac{1}{2^{i+j}}\frac{\big|\int \phi_{i}d\mu_{t_{j}}-\int \phi_{i}d\nu_{t_{j}}\big|}{1+\big|\int \phi_{i}d\mu_{t_{j}}-\int \phi_{i}d\nu_{t_{j}}\big|},
\end{equation}
where $\{\phi_{i}\}$ is dense set of functions in $C_c(B(0,2/\eps))$, and $\{t_{j}\}$ is a dense set of times in $[-4,-1]$. We recall from \cite[Sec. 2.4]{CHN} that $d_{\mathcal{B}}$ is compatible with the convergence of integral Brakke flows and strictly positive definite when restricted to selfsimilar solutions.

Now, by Huisken's monotonicity formula and quantitative differentiation (see \cite[Sec. 3.1]{CHN}), for any $\hat{\eps}>0$ there exists an $\hat{N}(\hat{\eps})<\infty$ such that for any $X\in\mathcal{M}$ in our ancient asymptotically cylindrical flow, and any $j\in \mathbb{Z}\setminus E_X$, where the set of exceptional scales $E_X$ satisfies $|E_X|\leq \hat{N}$, we have that
\begin{equation}\label{weak_selfsim}
d_{\mathcal{B}}(M_{X,r_j},\mathcal{N})<\hat{\eps}
\end{equation}
for some self-shrinker $\mathcal{N}$ with $\mathrm{Ent}[\mathcal{N}] \leq \mathrm{Ent}[S^{n-1}\times \mathbb{R}]$.

Recall that the structure of such shrinkers $\mathcal{N}=\{\sqrt{-t}\Sigma\}_{t<0}$ is given by Proposition \ref{uniform_r_high_d} (structure of low entropy shrinkers). Using the local regularity theorem and the clearing out lemma \cite{Brakke}, the weak $\hat{\eps}$-closeness to $\mathcal{N}$ from \eqref{weak_selfsim} can be upgraded to $\varepsilon$-selfsimilarity around $X$ at scale $r_j$ (see Definition \ref{def_alm_selfsim}), provided $\hat{\eps}=\hat{\eps}(\eps)>0$ is small enough. This proves the first assertion.

The second assertion follows from a similar argument (c.f. \cite[proof of Thm. 3.5]{CHH}), but for convenience of the reader let us provide the details. First, note that since by \cite{CIM} the round shrinking cylinder is isolated in the space of shrinkers, there exists some $\eps_{\textrm{CIM}}>0$ such that
\begin{equation}\label{cyl_isol}
d_{\mathcal{B}}(\mathcal{N},\mathcal{C})>\eps_{\textrm{CIM}}
\end{equation}
for any self-shrinker $\mathcal{N}$ with $\mathrm{Ent}[\mathcal{N}] \leq \mathrm{Ent}[S^{n-1}\times \mathbb{R}]$ that is not a round shrinking cylinder $\mathcal{C}$.
Assume from now on that $\eps\ll \eps_{\textrm{CIM}}$. If there was a sequence $j\to -\infty$ such that $\mathcal{M}$ is $\eps$-cylindrical around $X$ at scale $r_j$, then it would follow that some tangent flow at $X$ is a round shrinking cylinder. However, since some tangent flow at $-\infty$ is a round shrinking cylinder, by the equality case of Huisken's monotonicity formula this would yield that the flow $\mathcal{M}$ itself is a round shrinking cylinder, contradicting our assumption. Hence, remembering again that $\mathcal{M}$ is asymptotically cylindrical, it follows that there is a largest integer $J(X)\in \mathbb{Z}$ such that $\mathcal{M}$ is not $\eps$-cylindrical around $X$ at scale $r_j$ for all $j<J(X)$. Now, using the definition of $\eps$-cylindrical we see that Huisken's monotone quantity at scale $r_{J(X)}$ satisfies
\begin{equation}
\Theta_X(\mathcal{M},r_{J(X)})\geq \mathrm{Ent}[S^{n-1}\times \mathbb{R}]-\delta(\eps),
\end{equation}
where $\delta(\eps)\to 0$ as $\eps\to 0$. Fix a small constant $\eps_0\ll\eps_{\textrm{CIM}}$. Since on the other hand $\mathrm{Ent}[\mathcal{M}]\leq \mathrm{Ent}[S^{n-1}\times \mathbb{R}]$, using Huisken's monotonicity formula and quantitative rigidity (see \cite[Sec. 3.1]{CHN}), possibly after decreasing $\eps$, we infer that $\mathcal{M}$ is $\eps_0$-selfsimilar at scale $r_j$ for all $j\geq J(X)$. Remembering \eqref{cyl_isol}, for sufficiently small $\eps$ this yields that $\mathcal{M}$ is $\eps_1$-cylindrical at scale $r_j$ for all $j\geq J(X)$, where $\eps_1$ can be made as small as we want by choosing $\eps_0$ sufficiently small.\\
Finally, the same argument as above together with quantitative differentiation (see \cite[Sec. 3.1]{CHN}) yields that the quality of the necks in fact improves as we go further back in time, namely if $\mathcal{M}$ is $\eps_1$-cylindrical at scale $r_j$, then it is $\eps_1/2$-cylindrical at scale $r_{j'}$ for all $j'\geq j+N_1$, where $N_1=N_1(\eps_1)<\infty$. This implies the remaining assertion, and thus concludes the proof of the theorem.
\end{proof}

\begin{definition}[cylindrical scale]
The \emph{cylindrical scale} of $X\in\mathcal{M}$ is defined by
\begin{equation}
Z(X)=2^{J(X)}.
\end{equation}
\end{definition}

We conclude this subsection by proving one more proposition that will be useful in Section \ref{sec_cap_size}.

\begin{proposition}[{$\eps$-compact and $\eps$-separating}]\label{prop_useful_for45}
Let $\mathcal{M}$ be an ancient asymptotically cylindrical flow, and let $X_0=(x_0,t_0)\in \mathcal{M}$. Then:
\begin{enumerate}
\item If $\mathcal{M}$ is $\varepsilon$-compact around $X_0$ at scale $r$, then the connected component of $M_{t_0}$ containing $x_0$ is contained in $B(x_0,r/\eps)$.
\item If $\mathcal{M}$ is $\varepsilon$-separating around $X_0$ at scale $r$, then the extinction time satisfies $T_{e}(\mathcal{M})\geq t_0+{r^2}$, and there exists a point $x\in B(x_0,r/\eps)$ such that $(x,t_0+r^2)\in \mathcal{M}$.
\end{enumerate}
\end{proposition}

\begin{proof}
The first assertion follows from Proposition \ref{uniform_r_high_d} (structure of low entropy shrinkers) and the choice of $\varepsilon$ (Convention \ref{conv_eps}) by the avoidance principle \cite[Sec. 10]{Ilmanen_book}.

Next, note that Proposition \ref{uniform_r_high_d} (structure of low entropy shrinkers) and the choice of $\varepsilon$ (Convention \ref{conv_eps}) imply that both $K_{t_0-r^2}\cap B(x_0,r/\eps)$ and $B(x_0,r/\eps)\setminus K_{t_0-r^2}$ contain a ball of radius $2nr$. By comparison (Corollary \ref{containment}) this yields the second assertion.
\end{proof}

\bigskip

\subsection{Regularity of trapped regions}\label{traped_regions} In this section we prove that if some region of an ancient asymptotically cylindrical flow is trapped in a thin slab (either between two planes or between two cylindrical shells) over many scales, then its regularity scale is bounded from below.

For a set $A\subset \mathbb{R}^{n+1}$, a point $x\in\mathbb{R}^{n+1}$, and a radius $r>0$, we denote by $\mathrm{Th}(A,x,r)$ the infimum over all $w$ such that $(A-x)\cap B(0,r)$ after suitable rotation is contained in the slab $\{ |x_{n+1}|\leq w\}$.

\begin{lemma}[thickness of shrinkers]
There exists a constant $\delta_0>0$, such that every nonplanar shrinker $\Sigma\subset\mathbb{R}^{n+1}$ with $\mathrm{Ent}[\Sigma]\leq \mathrm{Ent}[S^{n-1}\times\mathbb{R}]$ satisfies
\begin{equation}
\mathrm{Th}(\Sigma,0,2n)\geq \delta_0.
 \end{equation}
\end{lemma}

\begin{proof}
Suppose towards a contradiction that there is a sequence $\Sigma_i\subset\mathbb{R}^{n+1}$ of nonplanar shrinkers  with $\mathrm{Ent}[\Sigma_i]\leq \mathrm{Ent}[S^{n-1}\times\mathbb{R}]$ such that 
$
\mathrm{Th}(\Sigma_i,0,2n)\leq i^{-1}.
$
Since $\Sigma_i\cap \overline{B}(0,\sqrt{2n})\neq \emptyset$ by comparison with spheres,  it follows from Allard's compactness theorem \cite{Allard} that the  $\Sigma_i$ subconverge to a shrinker $\Sigma$ with  $\mathrm{Ent}[\Sigma]\leq \mathrm{Ent}[S^{n-1}\times\mathbb{R}]$ that satisfies $\Sigma\cap \overline{B}(0,\sqrt{2n})\neq \emptyset$ and $\Sigma \cap B(0,2n) \subseteq P$ for some hyperplane $P$.  
By Theorem \ref{part_reg_thm} (partial regularity), we have $\mathcal{H}^{n-3} (\sing \Sigma) =0$, and hence in particular the regular set of $\Sigma$ is connected. Thus, choosing a regular point $p\in \Sigma \cap B(0,2n)$ and using smooth unique continuation, similarly as in the proof of Corollary \ref{M-to-V-corollary}, it follows that $\Sigma=P$. However, by the local regularity theorem, the flat plane is isolated in the space of shrinkers (see e.g. \cite[Prop. 3.2]{White_regularity}), so $\Sigma_i$ cannot converge to $\Sigma$. This gives the desired contradiction, and thus proves the lemma.
\end{proof}

Recall that the regularity scale $R(X)$ is defined as the maximal radius $r\geq 0$ such that $|A|\leq 1/r$ in the parabolic ball $P(X,r)$.
 
 \begin{proposition}[regions trapped between planes]\label{c0_plane}
There exist constants $\delta>0$ and $c>0$ with the following significance: If $X_0=(x_0,t_0)\in \mathcal{M}$ is a point on an ancient asymptotically cylindrical flow and if $r_0>0$ is such that
\begin{equation}\label{slab_bound}
\mathrm{Th}(M_{t_0-r^2},x_0,2nr) \leq \delta r
\end{equation}
for every $r \in [r_0,2^Nr_0]$, where $N$ is the constant from Theorem \ref{thm_finding_sim} (almost selfsimilarity), then 
\begin{equation}\label{Ecker_huisken}
R(X_0) \geq c r_0.
\end{equation}
\end{proposition}
 
 \begin{proof}

By Theorem \ref{thm_finding_sim} (almost selfsimilarity), there exists a scale $r\in [r_0,2^Nr_0]$ at which $\mathcal{M}$ is $\eps$-self similar around $X_0$ to a self-shrinker $\Sigma$. Taking $\delta=\delta_0/2$ from the previous lemma (and tacitly assuming that we fixed $\eps$ small enough), equation \eqref{slab_bound} implies that $\Sigma$ is the hyperplane (we note that being $\eps$-close to a plane was included as special case of $\eps$-separating in Definition \ref{def_alm_selfsim}). 
The bound \eqref{Ecker_huisken} now follows from White's local regularity theorem \cite{White_regularity}.
 \end{proof}
 
\begin{proposition}[regions trapped between cylindrical shells]\label{quant_strat_c0_promote}
There exist constants $\sigma>0$, $c>0$ and $\mathcal{T}>-\infty$ with the following significance. Suppose that $\mathcal{M}$ is an asymptotically cylindrical flow, with $(0,0)\in \mathcal{M}$, and that $t_0 \leq \mathcal{T}$. If for every $t\leq t_0$ and $x\in M_t\cap B(0,\sqrt{4n|t|})$ we have
\begin{equation}\label{c0_cyl_coarse}
\Big|\tfrac{\sqrt{x_1^2+\ldots +x_n^2}}{\sqrt{-t}}-\sqrt{2(n-1)}\Big| \leq \sigma,
\end{equation}
then for every  $t\leq t_0$ and $x\in M_t\cap B(0,\sqrt{3n|t|})$ we have
\begin{equation}\label{R_c0_promote}
R(x,t)\geq c \sqrt{-t}.
\end{equation}

\end{proposition}

 \begin{proof}
Let $\sigma$ be so small such that 
 \begin{equation}
\frac{2\sigma}{\delta}\cdot 2^N = \frac{\sigma^{2/3}}{2n},
 \end{equation}
 where $N$ is  the constant from Theorem \ref{thm_finding_sim} and $\delta$ is the constant from Proposition \ref{c0_plane}. Consider the increasing function 
 \begin{equation}
g(r)=\frac{r}{\sqrt{-t+r^2}},
 \end{equation}
which measures the quotient between the scale of the cylinder and the scale measured from $X=(x,t)$. 
Letting $r_0$ be such that $g(r_0)=\frac{2\sigma}{\delta}$, we see that $r_0\cong \frac{\sigma}{\delta}\sqrt{-t}$ if $|t|$ is sufficiently large. Therefore, by Proposition \ref{c0_plane}, if the postulated bound does not hold, then there exists some $r\in [r_0,2^Nr_0]$ such that 
\begin{equation}\label{no_slab}
\mathrm{Th}(M_{t-r^2},x,2nr) \geq \delta r.
\end{equation}

Letting $r_1$ to be such that $g(r_1)=\frac{\sigma^{2/3}}{2n}$ we see that
 \begin{equation}
2^N=\frac{g(r_1)}{g(r_0)}=\frac{r_1}{r_0}\cdot \frac{\sqrt{-t+r_0^2}}{\sqrt{-t+r_1^2}}\leq \frac{r_1}{r_0},
 \end{equation}
so $r\leq r_1$. Thus, by the monotonicity of $g$
 \begin{equation}
2nr=2ng(r)\sqrt{-t+r^2} \leq 2ng(r_1)\sqrt{-t+r^2}=\sigma^{2/3}\sqrt{-t+r^2},
 \end{equation}
so
\begin{equation}\label{lhs_c_0bd}
\mathrm{Th}(M_{t-r^2},x,2nr) \leq\mathrm{Th}(M_{t-r^2},x,\sigma^{2/3}\sqrt{-t+r^2}).
\end{equation}
Similarly,
\begin{equation}\label{rhs_c_0bd}
\delta r =\delta g(r)\sqrt{-t+r^2} \geq \delta g(r_0)\sqrt{-t+r^2}=2\sigma \sqrt{-t+r^2},
\end{equation}
and so combining \eqref{no_slab},\eqref{lhs_c_0bd}, and \eqref{rhs_c_0bd}, we obtain
\begin{equation}
\mathrm{Th}(M_{t-r^2},x,\sigma^{2/3}\sqrt{-t+r^2})) \geq 2\sigma \sqrt{-t+r^2} .
\end{equation}
This is contradictory to the fact that for the unit sphere $S^{n-1}$  and a point $p\in S^{n-1}$,
 \begin{equation}
\mathrm{Th}(S^{n-1},p,\eta) \cong \eta^2
 \end{equation}
for $\eta\ll 1$ (indeed, instead of $S^{n-1}$ we could have taken any smooth hypersurface of bounded geometry and any point on it).
 \end{proof}

\subsection{Asymptotic cylindrical scale}\label{sec_asympt_cyl_scale}

Let $\mathcal M$ be an ancient asymptotically cylindrical flow which is not the round shrinking cylinder.  By Colding-Minicozzi \cite{CM_uniqueness} the axis of the asymptotic cylinder is unique. We remark that Colding-Minicozzi  only explicitly stated the results for blowups, but their proof also applies for blowdowns (see also \cite[Prop. 4.1]{GH} for a uniform axis-tilt estimate for blowdowns). We can assume without loss of generality that the axis is in $x_{n+1}$-direction. Moreover, after translating and scaling, we can assume that $(0,0)\in\mathcal{M}$ and $Z(0,0)\leq 1$.

\begin{proposition}[asymptotic cylindrical scale]\label{prop_reg_growth}
For every $\delta>0$ there exists $\Lambda=\Lambda(\delta)<\infty$, such that $Z(p,0)\leq \delta |p|$ for all $p\in M_{0}$ with $|p|\geq \Lambda$.
\end{proposition}

\begin{proof} If the assertion fails for some $\delta>0$, then there is a sequence $\mathcal{M}^i$ of ancient asymptotically cylindrical flows with $Z(0,0)\leq 1$ and a sequence of points $p_i\in M_0^i$ with $|p_i|\to \infty$ and $Z(p_i,0)\geq \delta |p_i|$. Let $\mathcal{M}^i:=\mathcal{D}_{1/|p_i|}(\mathcal M^i)$ and pass to a subsequential limit $\mathcal{M}^\infty$. By construction, $\mathcal{M}^\infty$ is an ancient integral Brakke flow with entropy at most $\textrm{Ent}[S^{n-1}\times\mathbb{R}]$. Moreover, applying Theorem \ref{thm_finding_sim} (almost selfsimilarity) along the approximating sequence we infer that $\mathcal{M}^\infty$ has a cylindrical singularity at the space-time origin. Hence, by the equality case of Huisken's monotonicity formula,  $\mathcal{M}^\infty$ is a round shrinking cylinder that becomes extinct at time $0$. After passing to a subsequence we can assume that $p_i/|p_i|$ converges to a point $q$. If $q$ lies on the $x_{n+1}$-axis we obtain a contradiction with $\liminf_{i\to\infty}\tfrac{Z(p_i,0)}{|p_i|} >0$, and if $q$ does not lie on the  $x_{n+1}$-axis we obtain a contradiction the fact that the cylinder becomes extinct at time $0$. This proves the proposition.
\end{proof}

\begin{corollary}[barrier for the normalized flow]\label{cor_barrier}
There exists an even smooth function $\varphi :\mathbb{R}\to \mathbb{R}_+$ with $\lim_{ z\to \pm \infty}\varphi'(z)=0$ such that, given $X_0=(x_0,t_0)\in\mathcal{M}$, the normalized flow
$\bar M^{X_0}_\tau := e^{\frac{\tau}{2}}( M_{-e^{-\tau}} - x_0)$,
where $\tau=-\log (t_0-t)$, satisfies
\begin{equation}
\bar{M}^{X_0}_\tau \subseteq \left\{ \sqrt{x_1^2+\ldots+x_n^2} \leq \varphi(x_{n+1}) \right\}
\end{equation}
for $\tau\leq \mathcal{T}(Z(X_0))$, where $\mathcal{T}(Z(X_0))>-\infty$ is a constant that only depends on the cylindrical scale $Z(X_0)$. In particular, any potential ends must be in direction $x_{n+1}\to \pm \infty$. 
\end{corollary}

\begin{proof}
This follows from Proposition \ref{prop_reg_growth} arguing similarly as in \cite[Sec. 3.3]{CHH}.
\end{proof}

\bigskip

\section{Fine neck analysis for ancient asymptotically cylindrical flows}\label{sec_fine_neck_analysis}

\subsection{Setting up the fine neck analysis}\label{sec_fine_neck_setup}

Let $\mathcal M$ be an ancient asymptotically cylindrical flow, which is not the round shrinking cylinder. Given any $X_0=(x_0,t_0)\in \mathcal M$, we consider the normalized flow
\begin{equation}
\bar M^{X_0}_\tau = e^{\frac{\tau}{2}} \, \left( M_{{t_0}-e^{-\tau}} - x_0\right).
\end{equation}

By Theorem \ref{thm_finding_sim} and the uniqueness of the axis from Colding-Minicozzi \cite{CM_uniqueness} (see also \cite[Prop. 4.1]{GH}), we can assume that the rescaled flow converges for $\tau\to -\infty$ to the cylinder
\begin{equation}\label{eq_conv_axis}
\Sigma = \left\{x\in\mathbb{R}^{n+1}\, |\, x_1^2+\ldots + x_n^2=2(n-1)\right\}.
\end{equation}
Moreover, the convergence is uniform in $X_0$ once we normalize such that $Z(X_0)\leq 1$, see again \cite{CM_uniqueness} and \cite[Prop. 4.1]{GH}. Hence, we can find universal functions $\sigma(\tau)>0$ and $\rho(\tau)>0$ with
\begin{equation}\label{univ_fns}
\lim_{\tau \to -\infty} \sigma(\tau)=0,\quad\lim_{\tau \to -\infty} \rho(\tau)=\infty, \quad \textrm{and} -\rho(\tau) \leq \rho'(\tau) \leq 0,
\end{equation}
such that $\bar M^{X_0}_\tau$ is the graph of a function $u(\cdot,\tau)$ over $\Sigma \cap B_{2\rho(\tau)}(0)$ with
\begin{equation}\label{small_graph}
\|u(\cdot,\tau)\|_{C^4(\Sigma \cap B_{2\rho(\tau)}(0))} \leq \sigma(\tau) \, \rho(\tau)^{-1}.
\end{equation}
With the goal of deriving precise asymptotics for $u$, which capture the deviation from the exactly round cylinder, we will now set up a fine neck analysis as in \cite{ADS,BC,CHH}.

\bigskip

In the following, we denote by $C<\infty$ and $\mathcal{T}>-\infty$ constants that can change from line to line, and can depend on various other quantities, but are independent of the center point $X_0$ with $Z(X_0)\leq 1$. We also fix a nonnegative smooth function $\chi$ satisfying $\chi(z)=1$ for $|z| \leq \frac{1}{2}$ and $\chi(z)=0$ for $|z| \geq 1$, and set
\begin{equation}
\hat{u}(x,\tau)=u(x,\tau)\chi\!\left(\frac{x_{n+1}}{\rho(\tau)}\right).
\end{equation}

\bigskip

We recall from Angenent-Daskalopoulos-Sesum that there are $n$-dimensional shrinkers
\begin{align}\label{ADS_KM}
\Sigma_a &= \{ \textrm{hypersurface of revolution with profile } r=u_a(x_{n+1}), 0\leq x_{n+1} \leq a\},\\
\tilde{\Sigma}_b &= \{ \textrm{hypersurface of revolution with profile } r=\tilde{u}_b(x_{n+1}), 0\leq x_{n+1} <\infty\},\nonumber
\end{align}
as illustrated in \cite[Fig. 1]{ADS}, see also \cite{KM} and \cite[Sec. 8]{ADS} for a detailed description. These shrinkers can be used for barrier arguments as well as for calibration arguments.

\begin{proposition}[{c.f. \cite[Prop. 4.3]{CHH}, \cite[Prop. 2.3]{BC}, \cite[Prop. 2.6]{BC2}, \cite[Lem. 4.7]{ADS}}]\label{Gaussian density analysis}
The graph function $u$ satisfies the integral estimates
\begin{equation}
\int_{\Sigma \cap \{|x_{n+1}| \leq L\}} e^{-\frac{|x|^2}{4}} \, |\nabla u(x,\tau)|^2 \leq C \int_{\Sigma \cap \{|x_{n+1}| \leq \frac{L}{2}\}} e^{-\frac{|x|^2}{4}} \, u(x,\tau)^2
\end{equation}
and 
\begin{equation}
\int_{\Sigma \cap \{\frac{L}{2} \leq |x_{n+1}| \leq L\}} e^{-\frac{|x|^2}{4}} \, u(x,\tau)^2 \leq CL^{-2} \int_{\Sigma \cap \{|x_{n+1}| \leq \frac{L}{2}\}} e^{-\frac{|x|^2}{4}} \, u(x,\tau)^2
\end{equation}
for all $L \in [L_0,\rho(\tau)]$ and $\tau \leq \mathcal{T}$, where $L_0$ is a numerical constant.
\end{proposition}

\begin{proof}

Since $\mathcal{M}$ is an ancient asymptotically cylindrical flow, Huisken's monotonicity formula yields
\begin{equation}\label{eq_cal_in1}
\int_{\bar{M}^{X_0}_\tau}e^{-|x|^2/4}\leq \int_{\Sigma}e^{-|x|^2/4}\, .
\end{equation}
Thus, using the shrinker foliation from \eqref{ADS_KM} as a calibration as in \cite[proof of Prop. 4.2]{CHH} and applying the divergence theorem we infer that
\begin{equation}\label{eq_cal_in2}
 \int_{\Sigma\cap \{|x_{n+1}|\geq L\}}e^{-|x|^2/4}\leq \int_{\bar{M}_\tau\cap \{|x_{n+1}|\geq L\}}e^{-|x|^2/4} + CL^{-1} \int_{\Sigma\cap \{|x_{n+1}|= L\}}e^{-|x|^2/4} u^2,
\end{equation}
where we also used Corollary \ref{cor_barrier} to ensure that the region under consideration is indeed foliated by the family of shrinkers, provided $L_0$ is large enough and $\mathcal{T}$ is negative enough. Let us justify in more detail why the divergence theorem is indeed applicable in our possibly singular setting.
Set
\begin{equation}
\bar K^{X_0}_\tau = e^{\frac{\tau}{2}} \, \left( K_{{t_0}-e^{-\tau}} - x_0\right),
\end{equation}
where $K_t$ is the domain enclosed by $M_t$ from \eqref{eq_enclosed_t}, and let $\nu$ be the outward pointing normal on $\reg \bar{M}^{X_0}_{\tau}$. Denoting by $Z$ the solid cylinder with $\partial Z=\Sigma$, we  set $\Delta_{\tau}=\Delta_{\tau}^{\textrm{out}}\cup \Delta_{\tau}^{\textrm{in}}$, where
\begin{equation}
\Delta_{\tau}^{\textrm{out}}= \mathrm{Int}(K^{X_0}_\tau) - Z,\;\;\;\;\;\;\; \Delta_{\tau}^{\textrm{in}}= Z-\textrm{Int}(K^{X_0}_\tau).
\end{equation}
 By Theorem \ref{part_reg_thm} (partial regularity) we have $\mathcal{H}^{n}(\bar{M}^{X_0}_\tau)=0$. Thus, by  \cite[Section 4, Condition (4) implies condition (3)]{Federer_sur} for every $R>2L$ the set $\Delta_{\tau}^{\mathrm{out}}\cap \{|x_{n+1} \geq L|\}\cap \{|x| \leq R\}$ (respectively $\Delta_{\tau}^{\mathrm{in}}\cap \{|x_{n+1} \geq L|\}\cap \{|x| \leq R\}$) and its boundary indeed satisfy the divergence theorem.

Having established \eqref{eq_cal_in1} and \eqref{eq_cal_in2}, the rest of the proof is similar as in \cite[proof of Prop. 2.6]{BC2}.
\end{proof}

Since $\bar{M}_\tau$ moves by normalized mean curvature flow, the evolution of the graph function $u$ is governed by the linear operator
\begin{equation}\label{def_oper_ell}
\mathcal{L} = \Delta_{\Sigma}  - \tfrac{1}{2}  x^{\text{\rm tan}}\cdot \nabla + 1
\end{equation}
on the cylinder $\Sigma$. The nonlinear error can be estimated either pointwise, or in the Gaussian $L^2$-norm
\begin{equation}\label{def_norm}
\|f\|_G = \left(\int_\Sigma  (4\pi)^{-\frac{n}{2}} e^{-\frac{|x|^2}{4}} \, f^2 \right)^{1/2}.
\end{equation}

More precisely, we have:

\begin{lemma}[{c.f. \cite[Lem. 4.4, 4.5]{CHH}, \cite[Lem. 2.4, 2.5]{BC}}]\label{Error u-PDE}
The graph function $u(x,\tau)$ satisfies 
\begin{equation}
\left| (\partial_\tau-\mathcal{L}) u \right| \leq C\sigma(\tau)\rho^{-1}(\tau)\left( |u| + |\nabla u|\right)\qquad\qquad (\tau\leq \mathcal{T}).
\end{equation}
Moreover, the truncated graph function $\hat{u}(x,\tau) = u(x,\tau) \, \chi \big ( \frac{x_{n+1}}{\rho(\tau)} \big )$ satisfies 
\begin{equation}
\| (\partial_\tau-\mathcal{L}) \hat{u} \|_G \leq C\rho^{-1} \, \|\hat{u}\|_G \qquad\qquad (\tau\leq \mathcal{T}).
\end{equation}
\end{lemma}

\begin{proof}
The first assertion simply follows from linearizing the normalized mean curvature flow over the cylinder and using \eqref{small_graph}. Using this, the second assertion follows similarly as in \cite[proof of Lem. 4.5]{CHH}, where we now use
Proposition \ref{Gaussian density analysis} to estimate the Gaussian $L^2$-norm of the error terms.
\end{proof}  

Let us recall a few facts from \cite{BC2} about the operator $\mathcal L$ defined in \eqref{def_oper_ell}. In cylindrical coordinates this operator takes the form
\begin{equation}
\mathcal L f=\frac{\partial^2}{\partial x_{n+1}^2} f + \frac{1}{2} \, \Delta_{S^{n-1}} f - \frac{1}{2} \, x_{n+1} \, \frac{\partial}{\partial x_{n+1}} f + f.
\end{equation}
Denote by $\mathcal{H}$ the Hilbert space of all functions $f$ on $\Sigma$ such that $\| f\|_G<\infty$, where $\|\; \|_G$ is the Gaussian $L^2$-norm defined in \eqref{def_norm}. 
Analysing the spectrum of $\mathcal L$, the Hilbert space $\mathcal H$ can be decomposed as
\begin{equation}
\mathcal H = \mathcal{H}_+\oplus \mathcal{H}_0\oplus \mathcal{H}_-,
\end{equation}
where $\mathcal{H}_+$ is spanned by the $n+2$ positive eigenmodes $1, x_1, \ldots,x_{n+1}$ (here, $x_1,\ldots, x_{n+1}$ denotes the restriction of the Euclidean coordinate functions to $\Sigma$), and $\mathcal{H}_0$ is spanned by the $n+1$ zero-modes $x_{n+1}^2-2,x_1x_{n+1},\ldots,x_nx_{n+1}$. We have
\begin{align} 
&\langle \mathcal{L} f,f \rangle_G \geq \tfrac{1}{2} \, \|f\|_G^2 & \text{\rm for $f \in \mathcal{H}_+$,} \nonumber\\ 
&\langle \mathcal{L} f,f \rangle_G = 0 & \text{\rm for $f \in \mathcal{H}_0$,} \\ 
&\langle \mathcal{L} f,f \rangle_G \leq - \tfrac{1}{n-1}  \, \|f\|_G^2 & \text{\rm for $f \in \mathcal{H}_-$.} \nonumber
\end{align}

Now, with the aim of splitting the fine neck analysis into two cases, we consider the functions 
\begin{align}
&U_+(\tau) := \|P_+ \hat{u}(\cdot,\tau)\|_G^2, \nonumber\\ 
&U_0(\tau) := \|P_0 \hat{u}(\cdot,\tau)\|_G^2,\label{def_U_PNM} \\ 
&U_-(\tau) := \|P_- \hat{u}(\cdot,\tau)\|_G^2, \nonumber
\end{align}
where $P_+, P_0, P_-$ denote the orthogonal projections to $\mathcal{H}_+,\mathcal{H}_0,\mathcal{H}_-$, respectively. Using Lemma \ref{Error u-PDE} we obtain
\begin{align} 
&\frac{d}{d\tau} U_+(\tau) \geq U_+(\tau) - C\rho^{-1} \, (U_+(\tau) + U_0(\tau) + U_-(\tau)), \nonumber\\ 
&\Big | \frac{d}{d\tau} U_0(\tau) \Big | \leq C\rho^{-1} \, (U_+(\tau) + U_0(\tau) + U_-(\tau)), \label{U_PNM_system}\\ 
&\frac{d}{d\tau} U_-(\tau) \leq -\frac{2}{n-1}U_-(\tau) + C\rho^{-1} \, (U_+(\tau) + U_0(\tau) + U_-(\tau)). \nonumber
\end{align}

\begin{proposition}[{c.f. \cite[Sec. 4.1]{CHH}}]\label{Plus or Neutral}
Either the plus mode is dominant, i.e.
\begin{equation}\label{plus_dom}
U_- + U_0 \leq C\rho^{-1}U_+,
\end{equation}
or the neutral mode is dominant, i.e.
\begin{equation}\label{neut_dom}
U_+ + U_{-}=o(U_0).
\end{equation}
Moreover, which of the two cases happens only depends on $\mathcal{M}$, and not on $X_0$ and $\rho$.
\end{proposition}

\begin{proof}
This follows from \eqref{U_PNM_system} by applying the Merle-Zaag ODE-lemma \cite{MZ} similarly as in \cite[Sec. 4.1]{CHH}.
\end{proof}

\subsection{Fine analysis in the plus mode}\label{sec_plus_mode}

In this section, we assume that the plus mode is dominant. As explained above this means that after fixing a center $X_0\in\mathcal{M}$ with $Z(X_0)\leq 1$, and a graphical scale function $\rho$, we have that
\begin{equation}\label{assumption_plus_dom}
U_-+U_0 \leq C\rho^{-1}U_+
\end{equation}
for all $\tau\leq \mathcal{T}$. As before, $C<\infty$ and $\mathcal{T}>-\infty$ denote constants that can change from line to line and are independent of the point $X_0\in\mathcal{M}$ with $Z(X_0)\leq 1$. The main goal of this section is to prove Theorem \ref{thm Neck asymptotic}, which shows that all necks open up slightly in the $x_{n+1}$-direction in a very specific way.

\subsubsection{Graphical radius}

To get started, using \eqref{U_PNM_system} and \eqref{assumption_plus_dom} we compute
\begin{align}
\frac{d}{d\tau} U_+ \geq U_+ - C\rho^{-1} \, U_+.
\end{align}
Integrating this differential inequality, for every $\mu>0$ we get
\begin{align}
U_+(\tau) \leq Ce^{(1-\mu)\tau}
\end{align}
for all $\tau\leq \mathcal{T}(\mu)$.
Recalling that $U_+=\|P_+\hat{u}\|_G^2$ and using \eqref{assumption_plus_dom} we infer that
\begin{equation}
\| \hat u \|_G \leq Ce^{\frac{(1-\mu)\tau}{2}}.
\end{equation}
By interior estimates and interpolation this implies
\begin{align}\label{eq u weak estimate}
||u(\cdot,\tau)||_{C^{n+2}(\Sigma\cap \{|x_{n+1}|\leq 10L_0\})} \leq Ce^{\frac{(1-\mu)\tau}{2}}
\end{align}
for all $\tau \leq \mathcal{T}(\mu)$.

We will now show that the normalized flow is graphical over an exponentially large domain:

\begin{proposition}[{c.f. \cite[Prop. 4.10]{CHH}}]\label{prop4.10}
For $\tau \leq \mathcal{T}$ the normalized mean curvature flow $\bar{M}_\tau$ can be written as graph of a function $v(\cdot,\tau)$ over $\Sigma\cap \{ |x_{n+1}|\leq e^{-\tau/10}\}$ with the estimate
\begin{equation}
||v ||_{C^6(\Sigma \cap \{ |x_{n+1}|\leq e^{-\tau/10}\} )} \leq Ce^{\tau/10}.
\end{equation}
\end{proposition}

\begin{proof}
First, thanks to \eqref{eq u weak estimate} we can use the $n$-dimensional shrinkers from \eqref{ADS_KM} as barriers similarly as in \cite[proof of Prof. 4.9]{CHH} to infer that the rescaled mean curvature flow $\bar{M}_\tau$ satisfies the $C^0$-estimate
\begin{equation}\label{eq C^0 estimate}
\sup_{\bar{M}_\tau \cap \{|x_{n+1}| \leq 2e^{-\frac{\tau}{10}}\}} \big |x_1^2+\ldots+x_n^2-2(n-1)\big| \leq  C e^{\frac{\tau}{10}}.
\end{equation}
The geometric meaning of \eqref{eq C^0 estimate} is that the normalized flow is trapped between two cylindrical shells. Thus, the regularity scale of the normalized flow in this region clearly is bounded above. On the other hand, thanks to Proposition \ref{quant_strat_c0_promote} (regularity of trapped regions) the regularity scale of the normalized flow in this region is also bounded below, i.e. we get
\begin{equation}\label{equ_regscale}
C^{-1}\leq R(\bar{p},\tau)\leq C
\end{equation}
for $|x_{n+1}(\bar{p})|\leq \tfrac{3}{2}e^{-\tfrac{\tau}{10}}$, provided that $\tau$ is sufficiently negative.

Finally, having established \eqref{equ_regscale}, the $C^0$-estimate \eqref{eq C^0 estimate} can be upgraded to a graphical $C^6$-estimate by the same argument as in \cite[second half of the proof of Prop. 4.10]{CHH}.
\end{proof}

We now repeat the process from Section \ref{sec_fine_neck_setup} with improved functions $\rho$ and $\sigma$. Namely, by Proposition \ref{prop4.10} we can choose
\begin{equation}
\rho(\tau)=e^{-\tau/20},\qquad \sigma(\tau)=Ce^{\tau/20},
\end{equation}
and write $M_\tau$ as graph of a function $u(\cdot,\tau)$ defined over the exponentially large domain $\Sigma \cap B_{2\rho(\tau)}$, such that it satisfies the estimate \eqref{small_graph} for $\tau\leq\mathcal{T}$.

\begin{proposition}\label{thm_sharp_decay_plus}
For $\tau \leq \mathcal{T}$ the function $\hat{u}(x,\tau)=u(x,\tau)\chi\!\left(\frac{x_{n+1}}{e^{-\tau/20}}\right)$ satisfies the estimate
\begin{equation}
\|\hat u\|_G  \leq Ce^{\frac{\tau}{2}}.
\end{equation}
In particular, we have 
\begin{align}
\sup_{\bar M_{\tau}\cap \{ |x_{n+1}|\leq 10L_0 \}  }|x_1^2+\ldots+x_n^2-2(n-1)| \leq Ce^{\frac{\tau}{2}}.
\end{align}
\end{proposition}

\begin{proof}
The argument from \cite[proof of Prof. 4.11]{CHH} works in any dimension.
\end{proof}

\subsubsection{Constant functions cannot be dominant}\label{const_non_don_sec}

It is useful to further decompose
\begin{equation}
P_+=P_{1/2}+P_1,
\end{equation}
where 
 $P_{1/2}$ is the projection to the span of $x_1,\cdots,x_{n+1}$, and $P_1$ is the projection to multiples of $1$. Accordingly, we can decompose
\begin{equation}
U_+:=\|P_{+} \hat{u}(\cdot,\tau)\|_G^2=\|P_{1/2} \hat{u}(\cdot,\tau)\|_G^2 +\|P_{1} \hat{u}(\cdot,\tau)\|_G^2=:U_{1/2}+U_{1}.
\end{equation}
Using the assumption that the plus mode is dominant, and Lemma \ref{Error u-PDE}, we obtain
\begin{align}
\Big|\frac{d}{d\tau} U_{1/2}-U_{1/2}\Big| \leq C\rho^{-1}(U_{1/2}+U_{1}),\\
\Big|\frac{d}{d\tau} U_1-2U_1\Big| \leq C\rho^{-1}(U_{1/2}+U_{1}).
\end{align}
Hence, applying the Merle-Zaag ODE-lemma \cite{MZ} we infer that either the $U_{1/2}$ is dominant, i.e.
\begin{equation}\label{subcase1}
U_1=o(U_{1/2}),
\end{equation}
or the constant function $1$ is dominant, i.e.
\begin{equation}\label{subcase2}
U_{1/2}\leq C\rho^{-1} U_1.
\end{equation}

\begin{proposition}\label{thm_const_not_dom}
It must be the case that $U_1=o(U_{1/2})$.
\end{proposition}

\begin{proof}
The case \eqref{subcase2} can be excluded similarly as in \cite[proof of Prop. 4.12]{CHH}.
\end{proof}

\bigskip

\subsubsection{The fine neck theorem}

By Proposition \ref{thm_sharp_decay_plus} and Proposition \ref{thm_const_not_dom} we can now assume that 
\begin{equation}\label{assumption_plus_dom1}
U_-+U_0 \leq C\rho^{-1}U_+,
\end{equation}
and
\begin{equation}\label{u12dom}
U_1=o(U_{1/2}),
\end{equation}
where $\rho(\tau)=e^{-\tau/20}$.
Recall in particular that Proposition \ref{thm_sharp_decay_plus}  gives
\begin{equation}\label{decay_hatuu}
\|\hat u\|_G \leq Ce^{\frac{\tau}{2}}.
\end{equation}
Moreover, using in addition equation \eqref{u12dom} and the assumption that our solution is not the round shrinking cylinder, we see that
\begin{equation}\label{12isdom}
\lim_{\tau\to -\infty} e^{-\tau}U_{1/2} >0.
\end{equation}

\begin{lemma}[{c.f. \cite[Lem. 4.13]{CHH}}]\label{lemma_u_C2_coarse estimate} For $\tau\leq \mathcal{T}$ we have
\begin{equation}\label{u_C2_coarse estimate}
||u(\cdot,\tau)||_{C^{n}(\Sigma \cap \{|x_{n+1}|\leq 10L_0\})} \leq C e^{ \frac{40}{81}\tau}.
\end{equation}
\end{lemma}

\begin{proof}
Since the $C^4$-norm of $u$ is small, this follows from \eqref{decay_hatuu} by linear parabolic estimates, similarly as in \cite[proof of Lem. 4.13]{CHH}.
\end{proof}

\bigskip

We will now express $P_+\hat u\in\mathcal{H}_+$ as linear combination of the $n+2$ eigenfunctions $1,x_1,\cdots,x_{n+1}$. Namely, let
\begin{align}\label{coeffs_abcd}
&a^X(\tau)=2^{\frac{n-3}{4}}( \tfrac{e\pi}{n-1})^{\frac{1}{4}(n-1)}|S^{n-1}|^{-\frac{1}{2}}\int x_{n+1} \, \hat u^X(x,\tau)  e^{-\frac{|x|^2}{4}},\nonumber\\
&b_i^X(\tau)=2^{\frac{n-3}{4}}(1-\tfrac1n)^{-\frac{1}{2}}( \tfrac{e\pi}{n-1})^{\frac{1}{4}(n-1)}|S^{n-1}|^{-\frac{1}{2}} \int x_i \,  \hat  u^X(x,\tau)   e^{-\frac{|x|^2}{4}},\\
&c^X(\tau)=2^{\frac{n-1}{4}}( \tfrac{e\pi}{n-1})^{\frac{1}{4}(n-1)}|S^{n-1}|^{-\frac{1}{2}} \int  \hat u^X(x,\tau)e^{-\frac{|x|^2}{4}},\nonumber
\end{align}
where $1\leq i\leq n$, and where the superscript $X$ is to remind us that all these coefficients depend (a priori) on $X$.
Then, we have
\begin{align}
P_+\hat u^X=a^X x_{n+1}+\sum_{i=1}^nb_i^Xx_i+c^X. \label{def P_+hat u}
\end{align}
Moreover, $U_+^X=\|P_+\hat u^X\|_G^2$ is given by a sum of coefficients squared:
\begin{align}
U_+^X=e^{-\frac{n-1}{2}}\pi^{-\frac{n-1}{2}}2^{-\frac{n-3}{2}}(n-1)^{\frac{n-1}{2}}|S^{n-1}| \Big(|a^X|^2+(1-\tfrac1n)\sum^{n}_{i=1}|b_i^X|^2+ \tfrac12|c^X|^2\Big) \label{U_+ in terms of a,b,c,d}.
\end{align}

\begin{theorem}[{Fine neck theorem, c.f. \cite[Thm. 4.15]{CHH}}]\label{thm Neck asymptotic}
Let $\mathcal M$ be an ancient asymptotically cylindrical flow that is not a round shrinking cylinder. If the plus mode is dominant, then there are constants $\bar{a}=\bar{a}(\mathcal M)\neq 0$, $C=C(\mathcal M)<\infty$ and a decreasing function $\mathcal{T}:\mathbb{R}_+\to\mathbb{R}_{-}$ (depending on $\mathcal M$) with the following significance.

For every $X\in\mathcal M$ the graph function $u^X(\cdot,\tau)$ of the normalized flow $\bar{M}^X_\tau$ satisfies the estimates\footnote{We remind the reader that $L_0<\infty$ is a large numerical constant that has been fixed in Proposition \ref{Gaussian density analysis}.}
\begin{equation}\label{main_thm_est1}
\|e^{-\frac{\tau}{2}}\hat{u}^X(x,\tau)-\bar a x_{n+1}-\bar b_1^X x_1-\cdots-\bar b_n^X x_n\|_G   \leq C e^{\frac{\tau}{40}},
\end{equation}
and
\begin{align}\label{main_thm_est2}
\sup_{|x_{n+1}|\leq 10L_0}\big| e^{-\frac{\tau}{2}}{u}^X(x,\tau)-\bar a x_{n+1}-\bar b_1^X x_1-\cdots-\bar b_n^X x_n \big| \leq C e^{\frac{\tau}{160}}
\end{align}
for $\tau \leq \mathcal{T}(Z(X))$. Here, the constant $\bar{a}$ is independent of $X$, and $\bar b_i^X$ are numbers that may depend on $X$ and satisfy
\begin{equation}\label{main_thm_est3}
|\bar b_1^X|+\cdots + |\bar b_n^X|\leq C.
\end{equation}
\end{theorem}

\begin{proof}
First, by analyzing the evolution ODEs for the coefficients defined in \eqref{coeffs_abcd} similarly as in \cite[proof of Prop. 4.14]{CHH}, we infer that
\begin{equation}\label{coeff_est2}
|e^{-\tfrac{\tau}{2}} a^X(\tau)-\bar{a}^X|+\sum_{i=1}^{n}|e^{-\tfrac{\tau}{2}}b_i^X(\tau)-\bar{b}_i^X| \leq Ce^{\tfrac{\tau}{20}},\qquad |c^X(\tau) |\leq Ce^{\tfrac{11}{20}\tau},
\end{equation}
where $\bar{a}^X, \bar{b}_1^X,\cdots, \bar{b}_n^X$ are real numbers that might depend on $X$.

Consider the difference
\begin{equation}
D^X=\hat u -e^{\frac{\tau}{2}}\left(\bar a x_{n+1}+\bar b_1^X x_1+\cdots+\bar b_n^X x_n\right)
\end{equation}
Using \eqref{def P_+hat u} and \eqref{coeff_est2} we see that
\begin{equation}
|D^X|\leq |\hat{u}^X - P_{+}\hat{u}^X| + C(1+|x_{n+1}|)e^{\tfrac{11}{20}\tau}.
\end{equation}
Since by \eqref{assumption_plus_dom1} and \eqref{decay_hatuu} we have
\begin{equation}
U_{-}+ U_{0}\leq Ce^{\tfrac{21}{20}\tau},
\end{equation}
it follows that
\begin{equation}\label{est_for_DX}
\|D^X\|_{G} \leq Ce^{\frac{21}{40}\tau},
\end{equation}
which proves \eqref{main_thm_est1} modulo the claim about the coefficients.

Combining \eqref{decay_hatuu}, \eqref{U_+ in terms of a,b,c,d} and \eqref{coeff_est2} we see that
\begin{equation}\label{bc bound}
|\bar{b}_i^X|\leq C,
\end{equation}
which proves \eqref{main_thm_est3}.

We recall that $e^{-\frac{\tau}{2}}u$ corresponds to the original scale. Hence, if instead of $X=(x,t)$ we consider the new origin $X'=(x',t)$, where
\begin{equation}\label{eq_recenter}
x'=x+\sqrt{2(n-1)} \sum_{i=1}^n \bar b^X_i e_i,
\end{equation}
then the estimate \eqref{est_for_DX} simplifies to
\begin{equation}\label{est_simplified}
\|\hat u^{X'}(x,\tau) - e^{\tau/2} \bar{a}^X x_{n+1} \|_G \leq Ce^{\frac{21}{40}\tau},
\end{equation}
i.e. the estimate \eqref{main_thm_est1} holds with $\bar{a}^{X'}=\bar{a}^X$, $\bar b^{X'}=0$, and $\bar c^{X'}=0$. If $\bar{a}^X=0$, then \eqref{est_simplified} implies $\|\hat u^{X'} \|_{\mathcal{H}}^2 \leq Ce^{\frac{21}{20}\tau}$, contradicting \eqref{12isdom}. Here, we have used Proposition \ref{Plus or Neutral}, as well as Proposition \ref{thm_const_not_dom}, to show that, even after re-centering, the $\frac{1}{2}$ mode dominates. Hence, $\bar{a}^X\neq 0$. Since the estimate \eqref{est_simplified} holds for any $X$ and since $\bar{a}^X$ does not vanish for any $X$, we see that $\bar a^X=:\bar{a}$ is independent of $X$.

It remains to prove the pointwise estimate \eqref{main_thm_est2}. To this end, we start with
\begin{align}
\|D^X \|_{L^2(\Sigma\cap \{ |x_{n+1}|\leq 10L_0\})} \leq C\| D^X \|_G \leq Ce^{\frac{21}{40}\tau}.
\end{align}
Next, combining Lemma \ref{lemma_u_C2_coarse estimate} and inequality \eqref{bc bound} yields
\begin{equation}
\|D^X\|_{H^n(\Sigma\cap \{|x_{n+1}|\leq 10L_0\}} \leq C e^{\frac{41}{80}\tau}.
\end{equation}
Hence, applying Agmon's inequality
\begin{equation}
\|u\|_{L^{\infty}}\leq C\|u\|_{L^2}^{\frac{1}{2}}\|u\|_{H^n}^{\frac{1}{2}},
\end{equation}
we conclude that
\begin{equation}
\sup_{|x_{n+1}|\leq 10L_0} |D^X |\leq Ce^{\frac{81}{160}\tau}.
\end{equation}
This finishes the proof of the theorem.
\end{proof}

\bigskip

After a change of coordinates we can assume without loss of generality that our ancient low entropy flow $\mathcal M$ satisfies $\bar{a}=\bar{a}(\mathcal M)=\sqrt{(n-1)/2}$. Then, after recentering as above, the fine neck theorem (Theorem \ref{thm Neck asymptotic}) tells us that the graph $u^X(\cdot,\tau)$ of the rescaled flow $\bar{M}_\tau^X$ satisfies
\begin{align}\label{main_thm_est2pp}
\sup_{|x_{n+1}|\leq 10L_0}\big| e^{-\frac{\tau}{2}}u^X(x,\tau)- {\small \sqrt{(n-1)/2}}\; x_{n+1}\big| \leq C e^{\frac{\tau}{160}}
\end{align}
for $\tau\leq \mathcal{T}(Z(X))$.

\begin{corollary}[{c.f. \cite[Cor. 4.16]{CHH}}]\label{fine_neck_cor}
If the rescaled flow $\bar{M}_\tau^X=\partial \bar{K}_\tau^X$ satisfies \eqref{main_thm_est2pp}, then 
\begin{align}
&\bar{K}_\tau^X \cap \{  x_{n+1}\leq -L_0\}\subseteq \{ x_1^2+\ldots x_n^2 \leq 2(n-1)\},\\
&\bar{K}_\tau^X \cap \{  x_{n+1}\geq + L_0\}\supseteq \{ x_1^2+\ldots x_n^2 \leq 2(n-1)\},
\end{align}
and
\begin{align}
&\inf_{\bar{M}_\tau^X} x_{n+1} \leq -\mu e^{-\tau/2}
\end{align}
for $\tau\leq \mathcal{T}(Z(X))$, where $\mu>0$ is a numerical constant.
In particular, the unrescaled mean curvature flow $\mathcal{M}=\{M_t\}$ satisfies
\begin{equation}
\inf_{p\in M_t} x_{n+1}(p) >-\infty,\quad \textrm{ and }\quad \sup_{p\in M_t} x_{n+1}(p) =\infty.
\end{equation}
\end{corollary}

\begin{proof}
This follows from barrier arguments similarly as in \cite[proof of Cor. 4.16]{CHH}.
\end{proof}

\bigskip

\subsection{Fine analysis in the neutral mode}\label{sec_fine_neutral}

In this section, we assume that the neutral mode is dominant. The main goal is to prove Theorem \ref{thm_rotation} and Corollary \ref{thm_compact}, which show that the solution is compact with a precise inwards quadratic expansion.\\

Given any center $X_0\in \mathcal M$, there exists some functions $\sigma$ and $\rho$ satisfying \eqref{univ_fns}, such that \eqref{small_graph} holds, and we have
\begin{align}\label{Neutral mode}
&U_-+U_+=o(U_0), && \big|\partial_\tau U_0\big| \leq o(U_0)
\end{align}
for $\tau\leq \mathcal{T}$. In this subsection, $C<\infty$ and $\mathcal{T}>-\infty$ denote constants that can change from line to line, and that may also depend on $X_0\in\mathcal M$.
To distinguish the initial choice of $\rho$, we set 
\begin{equation}\label{rho_0_eq}
\rho_0(z)=\rho(z).
\end{equation}
We will later use improved scale functions, but $\rho_0$ will never change.

\bigskip

\subsubsection{Graphical radius}
To begin with, we consider the positive function
\begin{equation}\label{def_alpha}
\alpha(\tau)=\left(\int_{\Sigma \cap \{|x_{n+1}|\leq L\}}u^2(x,s) (4\pi)^{-\frac{n}{2}} e^{-\frac{|x|^2}{4}}\right)^{1/2}.
\end{equation}

\begin{lemma}[{c.f. \cite[Lem. 4.17]{CHH}}]\label{alpha0 bound L2}
For $L \in [L_0,\rho(\tau)]$, we have the estimate
\begin{align}
\alpha(\tau)^2 \leq \int_{\Sigma \cap \{|x_{n+1}|\leq L\}} u^2(x,\tau) (4\pi)^{-\frac{n}{2}}e^{-\frac{|x|^2}{4}} \leq C\alpha(\tau)^2.
\end{align}
\end{lemma}

\begin{proof}
The proof is similar to  \cite[Lem. 4.17]{CHH}.
\end{proof}

Now, define an increasing continuous function by
\begin{equation}\label{beta1 def}
\bar{\alpha}(\tau)= \sup_{\sigma \leq \tau}\alpha(\sigma).
\end{equation}
By standard interior estimates, we have 
\begin{equation}\label{C0 bound by beta1}
|u|(x,\tau) \leq C\bar{\alpha}(\tau)
\end{equation}
for $|x_{n+1}|\leq L_0$ and $\tau\leq\mathcal{T}$.

For technical reasons, it will be best to work with a monotone function $\beta$, which simultaneously has controlled derivatives. 
To this end, we define
\begin{equation}
\beta(\tau)=\sup_{\sigma \leq \tau}\left(\int_{\Sigma} u^2(x,\sigma)\chi^2\big( \tfrac{x_{n+1}}{\rho_0(\sigma)}\big)(4\pi)^{-\frac{n}{2}}e^{-\frac{|x|^2}{4}}\right)^{1/2},
\end{equation}
where we recall that $\rho_0$ is defined in \eqref{rho_0_eq} to be the original graphical scale. 
Clearly, $\beta$ is a locally Lipschitz, increasing function. By equation \eqref{Neutral mode}, we have $\beta'=o(\beta)$ at almost every time, so in particular
\begin{align}\label{beta3 gradient}
0\leq \beta'(\tau)\leq \tfrac{1}{5}\beta(\tau).
\end{align}
Moreover, by Lemma \ref{alpha0 bound L2} we have
\begin{equation}
\bar{\alpha}(\tau) \leq \beta(\tau) \leq C \alpha(\sigma) \leq C\bar{\alpha}(\tau),
\end{equation}
where $\sigma$ is chosen such that the second inequality holds. To recapitulate, we have obtained
\begin{equation}\label{beta23 ratio}
\bar{\alpha}(\tau)\leq \beta(\tau)\leq C\bar{\alpha}(\tau).
\end{equation}

\begin{proposition}[{c.f. \cite[Prop. 4.18]{CHH}}]\label{C0 estimate beta3}
There are constants $c>0$ and $C<\infty$ such that
\begin{equation}
|u|(x,\tau) \leq C\beta(\tau)^{\frac{1}{2}}
\end{equation}
whenever $|x_{n+1}| \leq c\beta(\tau)^{-\frac{1}{4}}$ and $\tau \leq \mathcal{T}$.
\end{proposition}

\begin{proof}
The proof is similar to  \cite[Prop. 4.18]{CHH}.
\end{proof}

Similarly as in Proposition \ref{prop4.10}, the $C^0$-estimate from Proposition \ref{C0 estimate beta3} can be upgraded to a $C^4$-estimate. Hence, we can now repeat the process from Section \ref{sec_fine_neck_setup} with better functions $\rho$ and $\sigma$. Namely, by we can now choose
\begin{equation}\label{improved_rho}
\rho(\tau)=\beta(\tau)^{-\frac{1}{5}}, \textrm{ and } \sigma(\tau)=\beta(\tau)^{\frac{1}{5}},
\end{equation}
and write $\bar{M}_\tau$ as a graph of a function $u$ over $\Sigma\cap B_{2\rho(\tau)}$ such that
\begin{equation}\label{C4 estimate}
\|u(\cdot,\tau)\|_{C^4(\Sigma\cap B_{2\rho(\tau)}(0))} \leq \rho(\tau)^{-2}
\end{equation}
for $\tau\leq \mathcal T$. Note that by equation \eqref{beta3 gradient} the derivative $\rho'$ indeed satisfies
\begin{equation}\label{contr_der}
-\rho(\tau) \leq \rho'(\tau) \leq 0,
\end{equation}
as required by condition \eqref{univ_fns}. From now on we work with the function
\begin{equation}
\hat{u}(x,\tau)=u(x,\tau)\chi\! \left( \frac{x_{n+1}}{\rho(\tau)}\right),
\end{equation}
where $\rho$ is the improved graphical radius from \eqref{improved_rho}.

\begin{proposition}[{c.f. \cite[Prop. 4.19]{CHH}}]\label{prop_improved_rho}
There are constants $\gamma>0$ and $c>0$ such that
\begin{equation}
\rho(\tau) \geq c |\tau|^\gamma
\end{equation}
holds for $\tau \leq \mathcal{T}$.
\end{proposition}

\begin{proof}
The proof is similar to  \cite[Prop. 4.19]{CHH}.
\end{proof}

\subsubsection{Expansion in terms of neutral eigenfunctions}

For the function
\begin{equation}
\hat{u}(x,\tau)=u(x,\tau)\chi \big( \tfrac{x_{n+1}}{\rho(\tau)}\big),
\end{equation}
where $\rho$ is the improved graphical scale from \eqref{improved_rho}, we let
\begin{align}
& U_+=\|P_+\hat u\|_G^2, && U_0=\|P_0\hat u\|_G^2, && U_-=\|P_-\hat u\|_G^2.
\end{align}
By assumption \eqref{Neutral mode} and Proposition \ref{Plus or Neutral} we have
\begin{equation}\label{neutral mode in check u}
 U_+ + U_- =o(U_0).
\end{equation}
Therefore, we can expand
\begin{equation}\label{hatu_expansion}
\hat u = \sum_{I=0}^{n}\alpha_I\psi_I + o(|\vec\alpha|),
\end{equation}
where
\begin{equation}
\vec{\alpha}(\tau)=(\alpha_0(\tau),\cdots,\alpha_{n}(\tau))
\end{equation}
are time dependent coefficients, and where $\psi_0,\cdots,\psi_{n}$ are the $n+1$ normalized zero eigenfunctions of $\mathcal L$,
explicitly
\begin{align}
\psi_{0} &=2^{\frac{n-7}{4}}( \tfrac{e\pi}{n-1})^{\frac{1}{4}(n-1)}|S^{n-1}|^{-\frac{1}{2}}(x_{n+1}^2-2),\\
\psi_i&=2^{\frac{n-5}{4}}(1-\tfrac1n)^{-\frac{1}{2}}( \tfrac{e\pi}{n-1})^{\frac{1}{4}(n-1)}|S^{n-1}|^{-\frac{1}{2}} x_i {x_{n+1}},
\end{align}
with $1\leq i \leq n$. Moreover, Lemma \ref{alpha0 bound L2} and equation \eqref{neutral mode in check u} yield that 
\begin{equation}\label{alpha0 and alpha}
C^{-1}\alpha(\tau) \leq |\vec{\alpha}|(\tau) \leq C\alpha(\tau)
\end{equation} 
for $\tau\leq\mathcal{T}$, where $\alpha(\tau)$ is the function defined in equation \eqref{def_alpha}.\\

In the following lemma we Taylor expand the normalized mean curvature flow to second order:

\begin{lemma}[{c.f. \cite[Lemm. 4.20]{CHH}}]\label{lemma_taylor} 
The function $\hat{u}(x,\tau)=u(x,\tau)\chi \big( \tfrac{x_{n+1}}{\rho(\tau)}\big)$ satisfies
\begin{equation}
\partial_\tau \hat u = \mathcal L \hat u  -\tfrac{u^2}{2\sqrt{2(n-1)}}-\tfrac{|\overline{\nabla}u|^2}{2(n-1)\sqrt{2(n-1)}}-\tfrac{u\overline{\Delta}u}{(n-1)\sqrt{2(n-1)}} + E,
\end{equation}
where the error term can be estimated by
\begin{align}\label{est_eq}
|E|\leq &C\chi(|u|+|\nabla u|)^2(|u|+|\nabla u|+|\nabla^2u|)\nonumber\\
&+ C |\chi'|\rho^{-1} \big(|\nabla u|+|x_{n+1}||u|\big)\nonumber\\
&+ C |\chi''|\rho^{-2}|u|\nonumber\\
&+ C \chi(1-\chi)\big(|u|^2+|\nabla u|^2+|\nabla^2u|^2\big).
\end{align}
\end{lemma}

\begin{proof}
Consider the normalized mean curvature flow for graphs over the cylinder
\begin{align}
\partial_\tau u=&
\frac{\Big[1+\frac{|\overline{\nabla}u|^2}{(\sqrt{2(n-1)}+u)^2}\Big]\partial_{x_{n+1}}^2u-\frac{\overline{\nabla}^2u(\overline{\nabla} u,\overline{\nabla} u)}{(\sqrt{2(n-1)}+u)^4}-\frac{(\partial_{x_{n+1}}u)( \partial_{x_{n+1}}|\overline{\nabla} u|^2)}{(\sqrt{2(n-1)}+u)^2}-\frac{|\overline{\nabla}u|^2}{(\sqrt{2(n-1)}+u)^3}}{1+(\partial_{x_{n+1}} u)^2+|\overline\nabla u|^2(\sqrt{2(n-1)}+u)^{-2}}\nonumber\\
&  + \frac{\overline{\Delta}u}{(\sqrt{2(n-1)}+u)^2}-\frac{n-1}{\sqrt{2(n-1)}+u }+\frac{1}{2}\left(\sqrt{2(n-1)}+u-x_{n+1}\partial_{x_{n+1}}u\right),
	 \end{align}
 where $\overline \nabla$ denotes the Levi-Civita connection on $S^{n-1}$, c.f.  \cite[eqn. (A.3), (A.4)]{GKS}. Hence,
 \begin{align}
\partial_\tau u= \partial_{x_{n+1}}^2u-\tfrac{|\overline{\nabla}u|^2}{(\sqrt{2(n-1)}+u)^3} + \tfrac{\overline{\Delta}u}{(\sqrt{2(n-1)}+u)^2}-\tfrac{n-1}{\sqrt{2(n-1)}+u }+\tfrac{1}{2}\left(\sqrt{2(n-1)}+u-x_{n+1}\partial_{x_{n+1}}u\right)+E_1,
	 \end{align}
where
\begin{equation}
|E_1|\leq C|\nabla u|^2(|\nabla u|+|\nabla^2 u|).
\end{equation}
Therefore,
\begin{equation}
\partial_\tau u=\mathcal{L}u+\mathcal{Q}(u)+E_2,
\end{equation}
where
\begin{equation}
\mathcal{Q}(u)=-\tfrac{u^2}{2\sqrt{2(n-1)}}-\tfrac{|\overline{\nabla}u|^2}{2(n-1)\sqrt{2(n-1)}}-\tfrac{u\overline{\Delta}u}{(n-1)\sqrt{2(n-1)}},
\end{equation}
and
\begin{equation}
|E_2|\leq C|\nabla u|^2(|u|+|\nabla u|+|\nabla^2 u|)+C|u|^2|\nabla^2 u|+C|u|^3.
\end{equation}
Next, using \eqref{contr_der} we obtain
\begin{equation}
|\partial_\tau \hat u-\mathcal{L}\hat u-\chi(\partial_\tau u-\mathcal{L}u)|\leq \rho^{-2}|\chi''||u|+2\rho^{-1}|\chi'|(|\nabla u|+|x_{n+1}||u|).
\end{equation}
Since $\chi$ only depends on $x_{n+1}$, we also get
\begin{equation}
|\mathcal{Q}(\hat u)-\chi\mathcal{Q}(u)|=|\chi^2\mathcal{Q}( u)-\chi\mathcal{Q}(u)|=\chi(1-\chi)|\mathcal{Q}(u)|.
\end{equation}
Putting everything together yields the desired result.
\end{proof}

\begin{proposition}[{c.f. \cite[Prop. 4.21]{CHH}}]\label{prop_error_term}
The error term $E$ from Lemma \ref{lemma_taylor} satisfies the estimate
\begin{equation}
|\langle E,\psi_I\rangle |\leq C \beta(\tau)^{2+\frac{1}{5}}
\end{equation}
for $\tau\leq\mathcal{T}$.
\end{proposition}

\begin{proof}
The proof is similar to  \cite[proof of Prop. 4.21]{CHH}.
\end{proof}

\subsubsection{The inwards quadratic neck theorem} The following proposition shows that the rotations, which are captured by the coefficients $\alpha_1,\ldots,\alpha_n$ from the expansion \eqref{hatu_expansion}, are rapidly decaying.

\begin{proposition}[{axis tilt decay, c.f. \cite[Prop. 4.24]{CHH}}]\label{prop_axis_tilt}
There exists a constant $\eta>0$ such that for all $\tau\leq\mathcal{T}$ we have
\begin{equation}
|\bar{\nabla} u|\leq e^{-\eta |\tau|^\gamma}
\end{equation}
for all $x\in \Sigma\cap \{ |x_{n+1}|\leq c|\tau|^\gamma\}$.
\end{proposition}

\begin{proof}
The proof is similar as in \cite[proof of Prop. 4.24]{CHH}, but we use the higher dimensional neck-improvement theorem from \cite{BC2} instead of the one from \cite{BC}.
\end{proof}

\begin{corollary}[{c.f. \cite[Cor. 4.25]{CHH}}]\label{cor_rota.decay} The coefficients $\alpha_1,\cdots,\alpha_n$ from the expansion \eqref{hatu_expansion} satisfy
\begin{equation}
\sum_{i=1}^n|\alpha_i|\leq Ce^{-\eta |\tau|^\gamma},
\end{equation}
for $\tau\leq \mathcal{T}$.
\end{corollary}

\begin{proof}
Consider the average
\begin{equation}
v(z):=\left(2(n-1)\right)^{-(n-1)/2}|S^{n-1}|^{-1}\int_{\Sigma\cap \{ x_{n+1}=z\}} \hat{u}.
\end{equation}
Since
\begin{equation}
\int_{\Sigma\cap \{ x_{n+1}=z\}} v\psi_i=0,
\end{equation}
using Proposition \ref{prop_axis_tilt} we can estimate
\begin{equation}
|\alpha_i|\leq | \langle \hat{u}-v, \psi_i\rangle|\leq  || \hat{u}-v ||_{\mathcal{H}}\leq Ce^{-\eta |\tau|^\gamma}.
\end{equation}
This proves the corollary.
\end{proof}

The next proposition gives an ODE for the coefficient $\alpha_0$ from the expansion \eqref{hatu_expansion}.

\begin{proposition}[{c.f. \cite[Prop. 4.26]{CHH}}]\label{prop_ODEs}
The coefficient $\alpha_0$ from the expansion \eqref{hatu_expansion} satisfies
\begin{align}
\tfrac{d}{d\tau} \alpha_0 &= -A\alpha_0^2  + o(\beta^2)+O(e^{-\eta|\tau|^\gamma}),\label{ODE1}
\end{align}
where $A=\frac{1}{\sqrt{n-1}}( \frac{2e\pi}{n-1})^{\frac{1}{4}(n-1)}|S^{n-1}|^{-\frac{1}{2}}$.
\end{proposition}

\begin{proof}

Using Lemma \ref{lemma_taylor}, Proposition \ref{prop_error_term}, and $\bar{\nabla} \psi_0=0$, we compute
\begin{align}
\tfrac{d}{d\tau} \alpha_{0} &= \langle \partial_\tau \hat u , \psi_{0} \rangle\\
&=\left\langle \mathcal L \hat u  -\tfrac{u^2}{2\sqrt{2(n-1)}}+\tfrac{|\overline{\nabla}u|^2}{2(n-1)\sqrt{2(n-1)}}-\tfrac{\text{div}_{S^{n-1}}( u\overline{\nabla}u)}{(n-1)\sqrt{2(n-1)}}+ E , \psi_{0}\right\rangle\\
&=\left\langle -\tfrac{u^2}{2\sqrt{2(n-1)}}+\tfrac{|\overline{\nabla}u|^2}{2(n-1)\sqrt{2(n-1)}} + E , \psi_{0} \right\rangle\\
&=-\tfrac{1}{2\sqrt{2(n-1)}}\alpha_0^2\langle \psi_0^2,\psi_0\rangle + o(\beta^2)+O(e^{-\eta |\tau|^\gamma}),
\end{align}
where in the last step we used Proposition \ref{prop_error_term}, Proposition \ref{prop_axis_tilt} and Corollary \ref{cor_rota.decay}. Hence,
\begin{equation}
\int \psi_0^3 (4\pi)^{-\frac{n}{2}}e^{-\frac{|x_{n+1}|^2+2(n-1)}{4}}=2^{\frac{n+5}{4}}(\tfrac{e\pi}{n-1})^{\frac{n-1}{4}}|S^{n-1}|^{-\frac{1}{2}}.
\end{equation}
implies the assertion.
\end{proof}

\begin{theorem}[{Inwards quadratic neck theorem, c.f. \cite[Thm. 4.28]{CHH}}]\label{thm_rotation}
For $\tau\leq \mathcal{T}$ the coefficients from the expansion \eqref{hatu_expansion} satisfy 
\begin{align}
\alpha_0(\tau)=\frac{-(n-1)^{\frac{1}{2}} (\frac{2e\pi}{n-1})^{-\frac{1}{4}(n-1)}|S^{n-1}|^{\frac{1}{2}}+o(1)}{| \tau|},
\end{align}
and
\begin{align}
\sum_{i=1}^n|\alpha_i|=o(|\alpha_0|).
\end{align}
\end{theorem}

\begin{proof} We can show $\limsup_{\tau \to -\infty}|\tau|^{10}\beta(\tau)=\infty$ as the proof of \cite[Lem. 4.27]{CHH}. Then, using Proposition \ref{prop_axis_tilt} and Proposition \ref{prop_ODEs}, the proof is similar to \cite[Thm. 4.28]{CHH}.
\end{proof}

\begin{corollary}[{c.f. \cite[Cor. 4.29 and Cor. 4.30]{CHH}}]\label{thm_compact}
For $\tau\leq\mathcal{T}$ we have
\begin{equation}\label{eq_cont_in_cyl}
\bar{K}_\tau^{X_0} \cap \{|x_{n+1}|\geq L_0\} \subset \{x_1^2+\ldots+x_n^2\leq 2(n-1)\}.
\end{equation}
Moreover, $K_t$ is compact for all $t$.
\end{corollary}

\begin{proof}
Using the inwards quadratic neck theorem (Theorem \ref{thm_rotation}) and the barriers $\tilde{\Sigma}_b$ from \eqref{ADS_KM}, similarly as in \cite[proof of Cor. 4.29]{CHH}, we see that \eqref{eq_cont_in_cyl} holds. In particular, considering the unrenormalized flow, given any $X_0=(x_0,t_0)\in \mathcal{M}$, we can find some $t_1\in (-\infty,t_0)$ and $L<\infty$ so that $(K_{t_1}-x_0)\cap \{ |x_{n+1}|\geq L\}$ is contained inside the solid cylinder of radius $\sqrt{2(t_0-t_1)}$. By \cite[Lem. 4.31]{CHH} and the comparison principle any mean curvature flow that outside of a ball is contained in the cylinder must be compact after the cylinder becomes extinct, namely $K_t$ is compact for all $t\geq t_0$. Since $X_0$ was arbitrary, this proves that $K_t$ is compact for all $t$.
\end{proof}

\bigskip

\section{Immortality, cap size control and asymptotics}\label{sec_cap_size}

Throughout this section, $\mathcal M$ will always be an ancient asymptotically cylindrical flow where the plus mode is dominant.
We will frequently use the fine neck theorem (Theorem \ref{thm Neck asymptotic}) and its corollary (Corollary \ref{fine_neck_cor}). By an affine change of coordinates, we can assume without loss of generality that
\begin{equation}
\bar{a}=\sqrt{\frac{n-1}{2}},
\end{equation}
that the axis of the asymptotic cylinder is the $x_{n+1}$-axis, and that
\begin{equation}
\min_{p\in M_0}x_{n+1}(p) =0 \textrm{ is attained at the origin.}
\end{equation}

\subsection{Speed of the tip and immortality}\label{sec_tip}
Let us consider the \emph{height of the tip function}
\begin{equation}\label{eq_height_tip}
\psi(t):= \min_{p\in M_t} x_{n+1}(p).
\end{equation}
The goal of this section is to derive certain estimates for $\psi$, and to show that the flow $\mathcal{M}$ is eternal, i.e. to prove that $T_e(\mathcal{M})=\infty$. To this end, we start with the following immediate observation.

\begin{lemma}[{c.f. \cite[Lemma 5.7]{CHH}}]\label{lemma_basic_tip}
The function $\psi$ is well-defined, strictly increasing, and satisfies
\begin{equation}
\lim_{t\to -\infty} \psi(t)=-\infty.
\end{equation}
\end{lemma}

\begin{proof}
This follows directly from Corollary \ref{cor_barrier}, Corollary \ref{fine_neck_cor}, and the strong maximum principle.
\end{proof}

The following proposition shows that tip points have controlled cylindrical scale.

\begin{proposition}[cylindrical scale of tip points]\label{fast_tip}
There exists a constant $Q=Q(\mathcal{M})<\infty$ such that every $p\in M_t$ with $x_{n+1}(p)=\psi(t)$ satisfies $Z(p,t)\leq Q$.
\end{proposition}

\begin{proof}
The argument is based on the one in \cite[proof of Prop. 5.8]{CHH}, but due to the lack of regularity/curvature control one needs to be a bit more careful.

If the assertion fails, we can find a sequence $X_j=(p_j,t_j)\in \mathcal M$ with $x_{n+1}(p_j)=\psi(t_j)$ such that $Z(X_j)\to \infty$. Consider the parabolically rescaled flow $\mathcal{M}^j=\mathcal{D}_{Z(X_j)^{-1}}(\mathcal{M}-X_j)$. Note that $\mathcal{M}^j$ has expansion parameter $\bar{a}^j=Z(X_j)^{-1}\bar{a}\to 0$. Up to a subsequence, we can pass to a limit $\mathcal{M}^\infty$. It is easy to see that  $\mathcal{M}^\infty$ is an ancient asymptotically cylindrical flow.

We first observe that $\mathcal{M}^\infty$ is not a round shrinking cylinder. Indeed, if $(0,0)$ was on the axis of such a round shrinking cylinder, this would contradict our scaling by the cylindrical scale. If on the other hand $(0,0)$ was not on the axis, then the cylinder would become extinct at a strictly positive time, contradicting the fact that $M^{\infty}_t$ is contained in a half-space for all $t>0$.

Next, suppose towards a contradiction that $\mathcal{M}^{\infty}$ has a dominant zero mode. Then, by the inwards quadratic neck theorem (Theorem \ref{thm_rotation}) at all sufficiently large scales the neck centered at $(0,0)$ bends inwards. For $j$ large, this contradicts the fine neck theorem (Theorem \ref{thm Neck asymptotic}) for $\mathcal{M}^j$ with center $(0,0)$.

By the above, $\mathcal{M}^{\infty}$ must be an ancient asymptotically cylindrical flow with dominant plus mode. Hence, by the fine neck theorem (Theorem \ref{thm Neck asymptotic}), it has an expension parameter $a^{\infty}=a^{\infty}(\mathcal{M}^\infty)\neq 0$. However, it follows from the fine neck theorem and its proof that $\bar{a}^j\rightarrow a^{\infty}$, contradicting $\bar{a}^j\rightarrow 0$.  This finishes the proof of the proposition.
\end{proof}

\begin{proposition}[{c.f. \cite[Prop. 5.1]{CHH}}]\label{cylind_growth}
For every $\Lambda < \infty$ there exists a $\rho=\rho(\mathcal M,\Lambda)<\infty$ with the following significance. If $(p_0,t_0),(p_1,t_1)\in \mathcal{M}$ are such that  $Z(p_0,t_0)\leq \Lambda$, $t_1\in [t_0-1,t_0+1]$, and
\begin{equation}
x_{n+1}(p_1)-x_{n+1}(p_0)\geq \rho,
\end{equation}
then
\begin{equation}
Z(p_1,t_1) > \Lambda.
\end{equation}
\end{proposition}

\begin{proof}
This follows from the fine neck theorem (Theorem \ref{thm Neck asymptotic}) arguing similarly as in \cite[proof of Prop. 5.1]{CHH}.
\end{proof} 

In contrast to \cite[Sec. 5.2]{CHH} we do not have any global curvature bound at our disposal at this stage. The following simple, but crucial, lemma serves as a partial substitute:

\begin{lemma}[Macroscopic speed limit]\label{macro_speed}
There exists a constant $C=C(\mathcal{M})<\infty$ such that  
\begin{equation}
\psi(t)-\psi(t') \leq C(t-t'+1)
\end{equation}
for all $t'\leq t< T(\mathcal{M})$.
\end{lemma}

\begin{proof}
For $t_0< T_e(\mathcal{M})$ let $p_{t_0}$ be a point at the tip at time $t_{0}$. 
Consider a tip point $p_{t_1}$ at time $t_1=t_{0}-1$. By Proposition \ref{fast_tip} (cylindrical scale of tip points) we have $Z(p_i,t_i) \leq Q$, and so applying Proposition \ref{cylind_growth} with $\Lambda=Q$, we obtain that $x_{n+1}(p_{t_0})-x_{n+1}(p_{t_1}) \leq \rho$. Iterating this, we get that 
\begin{equation}\label{tip_dif}
\psi(t)-\psi(t') \leq C(t-t'),
\end{equation}
whenever $t-t' \geq 1$. Together with the monotonicity of $\psi$ (see Lemma \ref{lemma_basic_tip}) this proves the assertion. 
\end{proof}

\begin{theorem}[immortality]\label{eternal}
The flow $\mathcal{M}$ is immortal, i.e. $T_e(\mathcal{M})=\infty$. Moreover, the height of the tip function satisfies $\lim_{t\rightarrow \infty} \psi(t)=\infty$.
\end{theorem}

\begin{proof}
Suppose towards a contradiction that $T_e(\mathcal{M})<\infty$. Then, by the macroscopic speed limit lemma (Lemma \ref{macro_speed}) we get that $\psi(T_e(\mathcal{M})):=\lim_{t\rightarrow T_e(\mathcal{M})} \psi(t)<\infty$. 
Let $p\in\mathbb{R}^{n+1}$ be a ``tip point'' of $\mathcal{M}$ at its finite extinction time, i.e. a point such that $(p, T_e(\mathcal{M}))\in \mathcal{M}$ and $\psi(T_e(\mathcal{M}))=x_{n+1}(p)$.
By the fine neck theorem (Theorem \ref{thm Neck asymptotic}), there exists some $t_0<T_e(\mathcal{M})$ such that for every  $t\leq t_0$  we see a fine-neck around $(p,T_e(\mathcal{M}))$. Let $N$ be the spatial connected component of $M_{t_0}$ containing this fine neck. Note that $N$ is non-compact by Corollary \ref{fine_neck_cor}. Consider the space-time points $X=(x,t_0)$ for $x\in N$. Proposition \ref{prop_useful_for45} implies $\mathcal{M}$ is not $\eps$-compact around $X$ at any scale, and also that $\mathcal{M}$ is not $\eps$-separating around $X$ at scales larger than $\sqrt{T_e(\mathcal{M})-t_0}$. Thus, using also Proposition \ref{uniform_r_high_d} and Theorem \ref{thm_finding_sim}, we see that that $Z(X)$ is uniformly bounded from above over $x\in N$; this however contradicts Proposition \ref{cylind_growth}, as $N$ contains points with arbitrarily large $x_{n+1}$ by Corollary \ref{fine_neck_cor}. Hence, the flow $\mathcal{M}$ is eternal, i.e.  $T_e(\mathcal{M})=\infty$.

Finally, suppose towards a contradiction that $L:=\lim_{t\rightarrow \infty}\psi(t)<\infty$. Consider pairs of times $t_0<t_1$ such that $\psi(t_i)\geq L-1$, and denote by $p_{t_i}$ a tip point at time $t_i$. By Proposition \ref{fast_tip} (cylindrical scale of tip points), and the fine neck theorem (Theorem \ref{thm Neck asymptotic}), provided $t_1-t_0$ is sufficiently large, at time $t_0$ we see a fine neck centered at $(p_{t_1},t_1)$. This contradicts the fact that $p_{t_0}$ is a tip point.
\end{proof}

\bigskip

\subsection{Cap size control}\label{sec_cap_size_asymptotics}

Let $\mathcal{M}$ be an ancient asymptotically cylindrical flow with dominant plus mode, normalized as in the previous subsection. For each $t\in\mathbb{R}$ select a point $p_t\in M_t$ with $x_{n+1}(p_t)=\psi(t)$.

\begin{theorem}[{Cap size control, c.f. \cite[Thm.  5.9]{CHH}}]\label{thm_asympt_par}
There exists a constant $C=C(\mathcal{M})<\infty$, such that for $t\in\mathbb{R}$ every point in $M_t\setminus B_C(p_t)$ lies on a fine neck. In particular, $M_t$ has exactly one end.

Moreover, $M_t\setminus B_C(p_t)$ is the graph of a function $r$ in cylindrical coordinates around the $x_{n+1}$-axis satisfying  
\begin{equation}\label{expansion_cylindrical}
r(t,x_{n+1},\omega)=\sqrt{2(n-1)(x_{n+1}-\psi(t))}+o\left(\sqrt{x_{n+1}-\psi(t)}\right)
\end{equation}
for $x_{n+1}\geq \psi(t)+C$, and the height of the tip function $\psi$ satisfies
\begin{equation}\label{height_est}
\psi(t)= t + o(|t|).
\end{equation}
\end{theorem}
\begin{proof}

Let $p\in M_t$. By Theorem \ref{eternal} (immortality) there exists a time $t_{\ast}\geq t$ such that $\psi(t_{\ast})=x_{n+1}(p)$. 
By Proposition \ref{fast_tip} (cylindrical scale of tip points)  we have $Z(p_{t_{\ast}})\leq Q$. Moreover, using the macroscopic speed limit lemma (Lemma \ref{macro_speed}),  we infer that there is some $c=c(\mathcal{M})>0$ such that
\begin{equation}
t_{\ast}-t \geq c(x_{n+1}(p)-\psi(t)),
\end{equation}
provided that $x_{n+1}(p)-\psi(t)$ is sufficiently large. Hence, applying the fine neck theorem (Theorem \ref{thm Neck asymptotic}) and Proposition \ref{prop4.10} with center $(p_{t_{\ast}},t_{\ast})$ we see that if $x_{n+1}(p)-\psi(t)$ is sufficiently large, then $p$ lies on a fine neck. Together with Corollary \ref{cor_barrier}, Corollary \ref{fine_neck_cor}, and Theorem \ref{eternal} (immortality), this proves the existence of a constant $C=C(\mathcal M)<\infty$ such that for all $t\in\mathbb{R}$ every $p\in M_t\setminus B_C(p_t)$ lies on a fine neck. In particular, this shows that $M_t$ has exactly one end.

The expansion \eqref{expansion_cylindrical} now follows from integrating the fine neck estimate, and \eqref{height_est} follows from comparison with scaled bowls (see \cite[proof of Thm.  5.9]{CHH} for how these two things are done). 
\end{proof}

\bigskip

\subsection{Fine expansion away from the cap}

Let $\mathcal{M}$ be an ancient asymptotically cylindrical flow with dominant plus mode, normalized as in the previous subsections. The goal of this subsection is to prove Theorem \ref{thm_neck_asympt}, which shows that the cylindrical end becomes rotationally symmetric at very fast rate, and also controls the distance of the cap from the $x_{n+1}$-axis uniformly in time.\\

Given a point $q\in \mathbb{R}^{n+1}$ and a direction $w \in S^n$, we denote by $\{R_\alpha \, :\, {1\leq \alpha \leq \tfrac{(n-1)n}{2}}\}$ a \emph{normalized set of rotation vector fields} that corresponds to an orthonormal basis of the space of the rotations around the axis $W=\{q+wt:t \in \mathbb{R}\}$, namely
\begin{equation}
R_\alpha(x) = SJ_\alpha S^{-1}(x-q), \qquad \textrm{ where } J_\alpha = \begin{bmatrix} \hat J_\alpha & 0 \\  0 & 0  \end{bmatrix}\in  \textrm{so}(n+1) ,
\end{equation}
where $S\in \textrm{SO}_{n+1}$ is any rotation matrix with $Se_{n+1}=w$, and $\{\hat J_\alpha:1\leq \alpha \leq \frac{(n-1)n}{2}\}$ is an orthonormal basis of $\textrm{so}(n)$.

\begin{definition}[{c.f. \cite[Def. 4.3]{BC2}}]
A point $X=(x,t)\in \mathcal{M}$ with $H(X)>0$ is called \emph{$\delta$-symmetric} if there exists a normalized set of rotation vector fields $\{R_\alpha:1\leq \alpha \leq \frac{(n-1)n}{2}\}$ such that
\begin{equation}
\max_\alpha|R_\alpha(x)|H(X) \leq 10n,
\end{equation}
and
\begin{equation}
\max_\alpha|\langle R_\alpha,\nu \rangle \, H | \leq \delta \textrm { in the parabolic ball } P(X,10H^{-1}(X)).
\end{equation}
\end{definition}

The following proposition shows that $\mathcal M$ becomes $\delta$-symmetric at a very fast rate if one moves away from the cap.

\begin{proposition}[{c.f. \cite[Prop. 6.2]{CHH}}]
There exist a constant $C=C(\mathcal{M})<\infty$, such that if $X=(x,t) \in \mathcal{M}$ is any point with
\begin{equation}
x_{n+1}-\psi(t)  \geq C,
\end{equation}
then 
\begin{equation}
\textrm{$X$ is $\left(x_{n+1}-\psi(t)\right)^{-300}$-symmetric.}
\end{equation}
\end{proposition}

\begin{proof}
This follows by combining the cylindrical expansion \eqref{expansion_cylindrical} from Theorem \ref{thm_asympt_par} with the neck-improvement theorem from Brendle-Choi \cite[Thm. 4.4]{BC2}, similarly as in \cite[proof of Prop. 6.2]{CHH}.
\end{proof}

\begin{corollary}[{strong symmetry, c.f. \cite[Cor. 6.3]{CHH}}]\label{cor_strong_symm}
There exist a constant $C=C(\mathcal{M})<\infty$ with the following significance. If $X=(x,t) \in \mathcal{M}$ is any point with $x_{n+1}-\psi(t)  \geq C$,
then there exist a direction $w_X \in S^n$ and a point $q_X\in\mathbb{R}^{n+1}$ with
\begin{equation}
|w_X-e_{n+1}|\leq \tfrac{1}{100},\qquad \langle q_X,e_{n+1}\rangle = x_{n+1},
\end{equation}
such that each normalized rotation vector field $R_{X,\alpha}(y)=S_{ X}J_\alpha S_{ X}^{-1}(y-q_X)$, where $S_{X}\in\textrm{SO}_{n+1}$ with $S_{X}e_{n+1}=w_X$, satisfies the estimate
\begin{equation}
\sup_{P(X,10H^{-1}(X))}
|\langle R_{X,\alpha},\nu \rangle \, H |\leq  \left(x_{n+1}-\psi(t)\right)^{-300}.
\end{equation}
\end{corollary}

\begin{proof}
Since fine necks are very close to the asymptotic cylinder, we always have $|w_X-e_{n+1}|\leq \tfrac{1}{100}$.  In addition, by moving $q_X$ along the axis $W$ we can always  arrange  that $\langle q_X,e_{n+1}\rangle = x_{n+1}$.   Hence,  the  corollary  follows  from  the proposition.
\end{proof}

\begin{definition}\label{Strong sym definition}
We call any triple  $(X,w_X,q_X)$ that satisfies the conclusion of Corollary \ref{cor_strong_symm} a \emph{strongly symmetric triple}.
\end{definition}

The following lemma shows that nearby strongly symmetric triples at the same time align well with each other.

\begin{lemma}[{alignment, c.f. \cite[Lemm. 6.5]{CHH}}]\label{Local estimate for strong symmetry}
There exists a constant $C=C(\mathcal{M})<\infty$ with the following significance. If $(X,w_{ X},q_{ X})$ and $( Y,w_{{Y}},q_{ Y})$ are strongly symmetric triples with $ X=( x, t)$, $Y=( y, t)$ and $ | x- y|H(X)\leq 1$, then
\begin{equation}\label{local axis estimate}
|w_X-w_Y | \leq C ( x_{n+1}-\psi(t))^{-300},
\end{equation}
and
\begin{equation}\label{local center estimate}
\sum_{i=1}^n |\langle q_{ X} - q_{Y},e_i\rangle | \leq C( x_{n+1}-\psi( t))^{-\frac{599}{2}}.
\end{equation}
\end{lemma}

\begin{proof}
Without loss of generality, after suitable rotations and translations, we can assume that $t=0$, $x_{n+1}=0$, $w_{{X}}=e_{n+1}$, $q_{ X}=0$, $w_Y=-\sin\varphi e_n+\cos \varphi e_{n+1}$ and
\begin{equation}
S_{{Y}}=\begin{bmatrix} I & 0 & 0 \\ 0 & \cos \varphi & -\sin \varphi \\ 0 & \sin\varphi & \cos \varphi \end{bmatrix},
\end{equation}
where $I$ is the $(n-1)\times (n-1)$ unit matrix and $0\leq \varphi \leq \tfrac{1}{10}$.

\bigskip

We express $\mathcal{M}\cap P({X},10H^{-1}({X}))$ in cylindrical coordinates over the $x_{n+1}$-axis, namely we parametrize as
\begin{equation}
(\omega,x_{n+1},t)\mapsto \big(r(\omega,x_{n+1},t)\omega,x_{n+1}\big).
\end{equation}
In these coordinates, one can directly compute that 
\begin{align}
\nu=&\frac{\left( \omega-r^{-1}\nabla^{S^{n-1}} r,-r_{x_{n+1}}\right)}{\sqrt{1+r^{-2}|\nabla^{S^{n-1}} r|^2+|r_{x_{n+1}}|^2}},
\end{align}
which, using also $w_{{X}}=e_{n+1}$ and $q_X=0$, yields
\begin{equation}
\langle R_{X,\alpha},\nu\rangle=\frac{-\langle \hat{J}_\alpha\omega,\nabla^{S^{n-1}} r\rangle }{\sqrt{1+r^{-2}|\nabla^{S^{n-1}} r|^2+|r_{x_{n+1}}|^2}}.
\end{equation}

Since $\{\hat J_{\alpha}:1\leq \alpha \frac{n-1)n}{2}\}$ is an orthonormal basis of $\textrm{so}(n)$, at each point with $\nabla^{S^{n-1}}r \neq 0$ there exists a unit vector $\lambda=(\lambda^1,\cdots,\lambda^{\frac{(n-1)n}{2}})\in \mathbb{R}^{\frac{(n-1)n}{2}}$ such that
\begin{equation}
\frac{\nabla^{S^{n-1}}r }{|\nabla^{S^{n-1}}r|}=\lambda^\alpha\hat J_{\alpha} \omega.
\end{equation}
Thus, we obtain
\begin{equation}
\big|\langle \lambda^\alpha R_{X,\alpha},\nu\rangle\big|= \frac{|\nabla^{S^{n-1}} r|}{\sqrt{1+r^{-2}|\nabla^{S^{n-1}} r |^2+|r_{x_{n+1}}|^2}}.
\end{equation}

Since $X$ is the center of a (fine) neck, we have
\begin{equation}\label{eq_neck_simple}
r^{-2}|\nabla^{S^{n-1}} r|^2+|r_{x_{n+1}}|^2\leq 10\varepsilon,
\end{equation}
and
\begin{equation}
\frac{1-\varepsilon}{r}\leq H(X) \leq \frac{1+\varepsilon}{r}.
\end{equation}
Combining these equations with Corollary \ref{cor_strong_symm} (strong symmetry) we infer that
\begin{equation}
\frac{|\nabla^{S^{n-1}} r|}{r}\leq 2 (-\psi(0))^{-300}.
\end{equation}
Together with Theorem \ref{thm_asympt_par} this yields the estimate
\begin{equation}\label{eq_est_r_theta}
|\nabla^{S^{n-1}} r| \leq C r^{-599}
\end{equation}
in the parabolic ball $P({X},10H^{-1}( X))$.

\bigskip

Now, choosing $\alpha$ of the special form $(i,n)$, where $1\leq i\leq n-1$, we consider $J_{(i,n)}=-e_i\otimes e_n +e_n \otimes e_i$, and the corresponding rotation vector field
\begin{equation}
R_{Y,(i,n)}(x)=S_{ Y}J_{(i,n)} S_{ Y}^{-1}(x-q_{Y})
\end{equation}
with center $q_Y=(q_1,\cdots,q_{n+1})$ and axis $w_Y=-\sin \varphi\, e_n+\cos\varphi \,e_{n+1}$ as above. A direct computation yields
\begin{multline}
R_{Y,(i,n)}(x)=-\big[ (x_n-q_n)\cos\varphi+(x_{n+1}-q_{n+1})\sin\varphi\big]e_i \\
+(x_i-q_i)\cos\varphi e_n  +(x_i-q_i)\sin\varphi e_{n+1}.\label{rotation vector comparison I}
\end{multline}

Arguing as above, using Theorem \ref{thm_asympt_par} and Corollary \ref{cor_strong_symm} we obtain
\begin{equation}\label{eq_est_ky}
|\langle R_{Y,(i,n)}, \nu\rangle |  \leq Cr^{-599}
\end{equation}
in the parabolic ball $P(X, 8 H^{-1}(X))$.
In addition, we have the rough estimate
\begin{equation}\label{q position bound}
 |q_Y|  \leq 10  r.
\end{equation}
Now, from equation \eqref{rotation vector comparison I}, using the estimates \eqref{eq_neck_simple}, \eqref{eq_est_r_theta}, \eqref{eq_est_ky} and \eqref{q position bound}, we infer that
\begin{equation}\label{rotation vector comparison II}
\bigg|r^{-1}x_i\big[ (x_n-q_n)\cos\varphi+(x_{n+1}-q_{n+1})\sin\varphi\big] -r^{-1}x_n(x_i-q_i)\cos\varphi +r_z(x_1-q_1)\sin\varphi \bigg|\leq Cr^{-599}.
\end{equation}
At time $t=0$, considering a point with $x_i=r$ and $x_n=0$ equation \eqref{rotation vector comparison II} yields
\begin{equation}
-q_n \cos\varphi+ \left( x_{n+1}- q_{n+1}+r_z(r-q_1)\right)\sin\varphi\leq Cr^{-599}.
\end{equation}
In the case $q_n \leq 0$, we consider the points with  $x_{n+1} =20r$. Then, using also $| q_{n+1}| \leq 10r$ and $\cos\varphi \geq \tfrac12$, we obtain
\begin{equation}
\tfrac12 |q_n|+\left( 10r+r_z(r-q_i)\right)\sin\varphi\leq Cr^{-599}.
\end{equation}
Moreover, since $X$ lies on a (fine) neck, we have $|r_{x_{n+1}}| \leq \varepsilon$. Hence, \begin{equation}\label{bound for S+q2 }
\tfrac12 |q_n|+5 r\sin\varphi\leq Cr^{-599}.
\end{equation}
Since $\sin\varphi\geq 0$, we infer that
\begin{align}\label{est_q2}
|q_n| \leq Cr^{-599},
\end{align}
and
\begin{align}\label{est_sinp}
 |\sin\varphi| \leq Cr^{-600}.
\end{align}
In the case $q_n \geq 0$, we obtain the same estimates by considering points with $ x_{n+1}=-20r$.

Finally, considering points with $x_i=0$ and $x_n=r$ we obtain
\begin{equation}\label{est_q1}
|q_i|\leq Cr^{-599}.
\end{equation}
Since $|w_X-w_Y|\leq C|\sin \varphi|$, these inequalities prove the lemma.
\end{proof}

Combining the above results now yields the main theorem of this subsection:

\begin{theorem}[{fine asymptotics, c.f. \cite[Thm. 6.6]{CHH}}]\label{thm_neck_asympt}
There exist a point $ q=( q_1,\cdots, q_n,0)\in \mathbb{R}^{n+1}$ and a constant $C<\infty$ (both depending on $\mathcal M$) such that for all $t\in\mathbb{R}$ the hypersurface $(M_t- q)\cap \{ x_{n+1}-\psi(t)\geq C\}$ can be expressed in cylindrical coordinates over the $x_{n+1}$-axis with the estimate
\begin{equation}\label{global angular derivative}
\left|\nabla^{S^{n-1}} r\right|(\omega ,x_{n+1},t) \leq  r(\omega ,x_{n+1},t)^{-100}.
\end{equation} 
\end{theorem}

\begin{proof}
This follows from Corollary \ref{cor_strong_symm} (strong symmetry) and Lemma \ref{Local estimate for strong symmetry} (alignment), arguing similarly as in \cite[proof of Thm. 6.6]{CHH}.
\end{proof}

\section{Moving plane method without assuming smoothness}\label{sec_moving_planes}

Let $\mathcal{M}$ be an ancient asymptotically cylindrical flow with dominant plus mode. We recall from \eqref{eq_support_t} and \eqref{eq_enclosed_t} that we denote by $M_t$ the support at time $t$, and by $K_t$ the domain enclosed by $M_t$.

By Theorem \ref{thm_asympt_par} (cap size control) and Theorem \ref{thm_neck_asympt} (fine asymptotics), after suitable normalization and space-time isometry, we know that for all $t\in \mathbb{R}$ we can express $M_t\cap \{ x_{n+1}-\psi(t)\geq C\}$ as a smooth graph in cylindrical coordinates over the $x_{n+1}$-axis such that
\begin{equation}\label{expansion_cylindrical restated}
r(\omega,x_{n+1},t)=\sqrt{2(n-1)(x_{n+1}-\psi(t))}+o\big(\sqrt{x_{n+1}-\psi(t)}\big),
\end{equation}
and
\begin{equation}\label{global angular derivative restated}
\left|\nabla^{S^{n-1}} r\right|(\omega ,x_{n+1},t) \leq  r(\omega ,x_{n+1},t)^{-100}.
\end{equation}
Here, the height of the tip function $\psi$ (as defined in \eqref{eq_height_tip}) satisfies
\begin{equation}
\psi(t)=|t|+o(|t|).
\end{equation}
Moreover, using Theorem \ref{thm_asympt_par} (cap size control) again, an taking also into account Proposition \ref{fast_tip} (cylindrical scale of tip points) and Corollary \ref{cor_barrier} (barrier for the rescaled flow), we see that, possibly after increasing $C=C(\mathcal{M})<\infty$ a bit, the set $M_t\cap \{ x_{n+1}-\psi(t)< C\}$ is contained in the ball
\begin{equation}\label{eq_bad_ball}
B(t):=B_{C}(0,\ldots, 0,\psi(t))\subset \mathbb{R}^{n+1}.
\end{equation}
In particular, any potential singularities at time $t$ are contained in the ball $B(t)$. We also recall that the results from Section \ref{sec_new_tools} (new tools for Brakke flows) are applicable, since $\mathcal{M}$ is a tame Brakke flow (see Definition \ref{def_tame_Brakke}) thanks to Corollary \ref{cor_tameness} (tameness).\\

To set up the (parabolic variant of the) moving plane method for tame Brakke flows as above, given a constant $\mu \geq 0$ we consider the sets
\begin{align}
M_t^{\mu -}&=M_t\cap \{ x_1 < \mu\},\\
M_t^{\mu +}&= M_t\cap \{ x_1 > \mu\}.
\end{align}
Moreover, we denote by $M_t^{\mu <}$ the set that is obtained from $M_t^{\mu +}$ by reflection about the plane $\{x_1=\mu\}$, namely
\begin{equation}
M_t^{\mu <}=\left\{(2\mu-x_1,x_2,\ldots,x_{n+1}):x\in M_t^{\mu +} \right\}.
\end{equation}

Similarly, for the domain $K_t$ enclosed by $M_t$ we consider the regions
\begin{align}
K_t^{\mu -}&=K_t\cap \{ x_1 < \mu\},\\
K_t^{\mu +}&= K_t\cap \{ x_1 > \mu\},
\end{align}
and 
\begin{equation}
K_t^{\mu <}=\left\{(2\mu-x_1,x_2,\ldots,x_{n+1}):x\in K_t^{\mu +} \right\}.
\end{equation}

\begin{definition}\label{def_reach_level}
We say \emph{the moving plane can reach $\mu$} if for all $\tilde{\mu}\geq \mu$ we have the inclusion $K_t^{\tilde{\mu} <}\subseteq K_t^{\tilde{\mu} -}$ for all $t\in \mathbb{R}$.
\end{definition}

The following proposition shows that the reflected region cannot touch at spatial infinity.

\begin{proposition}[{no contact at infinity, c.f. \cite[Prop. 6.10]{CHH}}]\label{Infinity Dirichlet}
For every $\mu>0$, there exists a constant $h_\mu <\infty $ such that
\begin{equation}
K_t^{\tilde{\mu}  <}\cap \{x_{n+1} \geq \psi(t)+h_\mu \}\subseteq \mathrm{Int}(K_t^{\tilde{\mu}-})
\end{equation}
for every $\tilde{\mu}\geq \mu$ and  $t\in\mathbb{R}$. 
\end{proposition}

\begin{proof}
Denote the position vector of $M_t\cap \{x_{n+1} \geq \psi(t)+C\}$ by
\begin{equation}
X(\omega,x_{n+1},t)=\left( r(\omega,x_{n+1},t)\omega,x_{n+1}\right).
\end{equation}
Given $(\bar{x},\bar{t})$ with $\bar{x}_{n+1}-\psi(\bar{t})\geq C$, using Theorem \ref{thm_neck_asympt} (fine asymptotics), we obtain
\begin{equation}\label{r_diff_small}
|r(\omega,\bar{x}_{n+1},\bar{t})-r(\bar{\omega},\bar{x}_{n+1},\bar{t})| \leq \frac{2\pi}{r(\bar{\omega},\bar{x}_{n+1},\bar{t})^{100}}
\end{equation}
for any $\omega\in S^{n-1}$, from which, given any $\mu>0$, we directly infer that
\begin{equation}
K_{\bar{t}}^{\tilde{\mu} <}\cap \{x_{n+1}=\bar{x}_{n+1}\}\cap \{x_1 \leq \tfrac{1}{2}\tilde{\mu}\} \subseteq \textrm{Int}(K_{{\bar{t}}}^{\tilde{\mu} -}),
\end{equation}
for every $\tilde{\mu} \geq \mu$, provided that $\bar{x}_{n+1}-\psi(\bar{t})$ is sufficiently large, depending only on $\mu$.

Now, let $\gamma:[0,\pi/2]\rightarrow S^{n-1}$ be a unit speed geodesic starting from $(1,0,\ldots,0)$. Then 
\begin{align}
\frac{d}{ds}\left(r(\gamma(s),\bar{x}_{n+1},\bar{t})\sin(s)\right)&= \sin(s)\langle \nabla^{S^{n-1}}r, \gamma' \rangle+r \cos(s)>0
\end{align}
whenever $r\cos(s)\geq \frac{\mu}{4}$, where the last inequality holds true by Theorem \ref{thm_neck_asympt} (fine asymptotics) provided $\bar{x}_{n+1}-\psi(\bar{t})$ is sufficiently large. Since $r(\gamma(s),\bar{x}_{n+1},\bar{t})\sin (s)$ is the length of the projection of $X(\gamma(s),\bar{x}_{n+1},\bar{t})$ on the $\{x_1=0,\;x_{n+1}=\bar{x}_{n+1}\}$ plane, it follows that $M_{\bar{t}}^{\frac{\mu}{2}+}\cap \{x_{n+1}=\bar{x}_{n+1}\}$ is  a graph over the plane $\{x_1=0,\;x_{n+1}=\bar{x}_{n+1}\}$. This clearly implies that 
\begin{equation}
K_{\bar{t}}^{\tilde{\mu} <}\cap \{x_{n+1}=\bar{x}_{n+1}\}\cap \{x_1 \geq \tfrac{1}{2}\tilde{\mu}\} \subseteq \textrm{Int}(K_{\bar{t}}^{\tilde{\mu} -})
\end{equation}
for every $\tilde{\mu}\geq \mu$ whenever $\bar{x}_{n+1}-\psi(\bar{t})$ is sufficiently large, and thus finishes the proof.
\end{proof}

\begin{corollary}[start plane and start smoothness]\label{cor_start_plane}
There exists some $\mu<\infty$, such that the moving plane can reach $\mu$, and such that all points in $\mathcal{M} \cap \{x_1 \geq \mu\}$ are smooth.
\end{corollary}

\begin{proof}
Recall from above that $M_t\setminus B(t)$, with $B(t)$ as in \eqref{eq_bad_ball}, is smooth and can be expressed as a graph in cylindrical coordinates over the $x_{n+1}$-axis satisfying \eqref{expansion_cylindrical restated} and \eqref{global angular derivative restated}. Applying Proposition \ref{Infinity Dirichlet} (no contact at infinity) with $\mu=2C$, we get some $h$ such that
\begin{equation}
K_t^{{\tilde{\mu}}  <}\cap \{x_{n+1} \geq \psi(t)+h\}\subseteq \mathrm{Int}(K_t^{{\tilde{\mu}}-})
\end{equation}
for every $\tilde{\mu}\geq 2C$ and  $t\in\mathbb{R}$. Together with
\begin{equation}
K_t^{{\tilde{\mu}}  <}\cap \{x_{n+1} \leq  \psi(t)+h\} \subseteq  B_{2C+2h}(0,0,\ldots,\psi(t)),
\end{equation}
this yields the assertion with $\mu=2C+2h$.
\end{proof}

\begin{proposition}[smoothness]\label{prop_smothness_reached}
If the moving plane can reach $\mu> 0$, then all points in $\mathcal{M}\cap\{x_1\geq \mu\}$ are smooth. Moreover, there exists $R_{\mu}>0$, such  the regularity scale of $\mathcal{M}$ at every $(p,t)\in \mathcal{M}\cap \{x_1=\mu\}$ satisfies
\begin{equation}\label{reg_bound_below}
R(p,t) \geq R_{\mu}.
\end{equation}
\end{proposition}

\begin{proof}
Consider
\begin{equation}
I_\mu:=\{ \mu' \geq \mu \;|\; \textrm{all points in } \mathcal{M}\cap \{x_1\geq \mu'\}\;\textrm{ are smooth}\}.
\end{equation}
Note that $I_\mu\subseteq [\mu,\infty)$ is an interval, which by Corollary \ref{cor_start_plane} (start plane and start smoothness) contains a neighborhood of $\infty$. Let
\begin{equation}
\mu_\ast :=\inf I_\mu.
\end{equation}
We first claim that for every $t$ we have
\begin{equation}\label{no_inter_M1M2}
M_t^{\mu_{\ast}-}\cap M_t^{\mu_{\ast}<} =\emptyset. 
\end{equation}
Indeed, consider the halfspace $\mathbb{H}:=\{ x_1 <\mu_\ast\}$, and let $\mathcal{M}^1_{\mathbb{H}}:=\{M_t^{\mu_{\ast}<}\}$ be the smooth mean curvature flow which is obtained from $\mathcal{M}$ by reflecting across $\{x_1=\mu_\ast\}$ and restricting to $\mathbb{H}$, and let $\mathcal{M}^2_{\mathbb{H}}$ be the Brakke flow with support $M_t^{\mu_{\ast}-}$, which is obtained from $\mathcal{M}$ by restricting to $\mathbb{H}$. If there was some intersection point $X_0=(x_0,t_0)$, then by the strong maximum principle for Brakke flows (Theorem \ref{strong_max_Brakke}) the flows $\mathcal{M}^1_{\mathbb{H}}$ and $\mathcal{M}^2_{\mathbb{H}}$ would coincide in $P(X_0,r)$ for some $r>0$. Together with  connectedness it would then follow that $\mathcal{M}^1_{\mathbb{H}}=\mathcal{M}^2_{\mathbb{H}}$. However, since $\mu_{\ast}>0$ this would contradict Proposition \ref{Infinity Dirichlet} (no contact at infinity).

Next, we would like to apply Theorem \ref{mirror-theorem} (Hopf lemma without assuming smoothness) to show that all points in $\mathcal{M}\cap \{x_1= \mu_{\ast}\}$ are regular. To this end,  we consider the flow $\mathcal{M}^1$ that is obtained from $\mathcal{M}$ by reflection across $\{x_1=\mu_{\ast}\}$, and set $\mathcal{M}^2:=\mathcal{M}$. Let $x_0\in M_{t_0} \cap \{x_1=\mu_{\ast}\}$. If $\partial \mathbb{H}$ is the tangent flow to either $\mathcal{M}^1$ or $\mathcal{M}^2$ at $(x_0,t_0)$, then $(x_0,t_0)$ is a smooth point for $\mathcal{M}$, and we are done. Hence, we can assume that assumption (ii) of Theorem \ref{mirror-theorem} is satisfied. Note that by Theorem \ref{part_reg_thm} (partial regularity) and equation \eqref{no_inter_M1M2} assumptions (i) and (iii) hold as well. Hence, Theorem \ref{mirror-theorem} yields that $(x_0,t_0)$ is a regular point for $\mathcal{M}=\mathcal{M}^2$.  This shows that all the points in $\mathcal{M}\cap \{x_1\geq \mu_{\ast}\}$ are regular and consequently $\mu_\ast\in I_\mu$.

Now, suppose towards a contradiction that $\mu_\ast>\mu$. Then, there is a sequence $(p_i,t_i)\in \mathcal{M}\cap\{x_1=\mu_{\ast}\}$ with regularity scale $R(p_i,t_i)$ going to zero. Taking into account the structure of $\mathcal{M}$, as reviewed at the beginning of this section, we see that $|p_i- \psi(t_i)e_{n+1}|$ is uniformly bounded. Let $\overline{\mathcal{M}}$ be a subsequential limit of the flows $\mathcal{M}-(\psi(t_i)e_{n+1},t_i)$. Then, as above, Theorem \ref{mirror-theorem} (Hopf lemma without assuming smoothness) yields that all the points in $\overline{\mathcal{M}}\cap \{x_1= \mu_{\ast}\}$ are regular, which contradicts $R(p_i,t_i)\to 0$. Hence, $\mu_\ast=\mu$.

Combining the above, we conclude that $\mu\in I_\mu$. The argument in the above paragraph also shows the moreover part of the proposition, which complete its proof.
\end{proof}

Now,  for $h_\mu$ as in Proposition \ref{Infinity Dirichlet} and $\delta>0$ we define
\begin{align}
&E_t^{\mu}=\{x_{n+1} \leq \psi(t)+h_{ \mu /2}\}, && E_t^{\mu,\delta}=\{x\in E_t^{\mu}: d(x,M_t\cap \{ x_1 = \mu\})\geq \delta\}.
\end{align}

The following lemma shows that if the moving plane can reach $\mu>0$, then $\delta$-away from $M_t\cap \{ x_1 = \mu\}$ the distance between $M_t^{\mu -}$ and $M_t^{\mu <}$ is bounded below by a definite amount.

\begin{lemma}[{distance gap, c.f. \cite[Lem. 6.12]{CHH}}]\label{distance gap}
Suppose the moving plane can reach $\mu>0$. Then, there exists a positive increasing function $\alpha:(0,\delta_0)\to \mathbb{R}_+$ such that 
\begin{equation}
d(M_t^{\mu -},K_t^{\mu <}\cap  E_t^{\mu,\delta}) \geq \alpha(\delta)>0
\end{equation}
for all $t\in \mathbb{R}$.
\end{lemma}

\begin{proof}
We will first show that 
\begin{equation}\label{showfirst}
K_t^{\mu <}\subseteq \textrm{Int}(K_t^{\mu -})
\end{equation}
for all $t\in \mathbb{R}$. To this end, note that by assumption (see Definition \ref{def_reach_level}) we have $K_t^{\mu <}\subseteq K_t^{\mu -}$ for all $t$. If \eqref{showfirst} fails, then there must be some $t_0\in\mathbb{R}$ such that $ M_{t_0}^{\mu <}\cap M_{t_0}^{\mu -}\neq \emptyset$. By the strong maximum principle for Brakke flows (Theorem \ref{strong_max_Brakke}), which is applicable thanks to Proposition \ref{prop_smothness_reached} (smoothness), this yields a contradiction with Proposition \ref{Infinity Dirichlet} (no contact at infinity), and thus proves \eqref{showfirst}.

Now, suppose towards a contradiction that for some $\delta>0$ we have
\begin{equation}
\inf_{t\in\mathbb{R}} d(M_t^{\mu -} ,K_t^{\mu <}\cap  E_t^{\mu,\delta})=0.
\end{equation}
Choose a sequence of space-time points $(x_i,t_i)\in \mathcal{M}$ such that $x_i \in M_{t_i}^{\mu <}\cap  E_{t_i}^{\mu,\delta}$ and $\lim_{i\to \infty} d(x_i,K_{t_i}^{\mu -})=0$. By Proposition \ref{Infinity Dirichlet} (no contact at infinity), Theorem \ref{thm_neck_asympt} (fine asymptotics), and the uniform cap position control \eqref{eq_bad_ball}, the distance between $x_i$ and the point $\psi(t_i)e_{n+1}$ is uniformly bounded. Hence, we can take subsequential limits $\overline{\mathcal{M}}$ and $\bar x$ of the flows $\mathcal{M}-(\psi(t_i)e_{n+1},t_i)$ and points $x_i-\psi(t_i)e_{n+1}$. Applying the strong maximum principle (Theorem \ref{strong_max_Brakke}) for $\overline{\mathcal{M}}$ at the spacetime point $(\bar x,0)$ we obtain a similar contradiction as above. This proves the lemma.
\end{proof}

If the moving plane reaches $\mu>0$, then by Proposition \ref{prop_smothness_reached} (smoothness) points in
\begin{equation}
M_t^{\mu}:=M_t\cap \{ x_1 = \mu\}
\end{equation}
have a well defined normal vector $\nu$.

\begin{lemma}[{angle gap, c.f. \cite[Lem. 6.13]{CHH}}]\label{Angle gap}
Suppose that the moving plane can reach $\mu>0$. Then, there exists a positive constant $\theta_\mu>0$ such that $|\langle \nu(x,t),e_1\rangle|\geq \theta_\mu$ holds on $M_t^{\mu}\cap E_{t}^{\mu}$ for all $t\in\mathbb{R}$.
\end{lemma}

\begin{proof}
First, Proposition \ref{prop_smothness_reached} (smoothness) and the Hopf lemma for smooth flows show that $|\langle \nu(x,t),e_1\rangle|\neq 0$.

Now, suppose towards a contradiction there is a sequence $(x_i,t_i)\in \mathcal{M}$ such that $x_i \in M_{t_i}^{\mu}\cap E_{t_i}^{\mu}$ and $\lim_{i\to \infty} |\langle \nu(x_i,t_i),e_1\rangle|=0$. Then, similarly as above, we can take subsequential limits $\overline{\mathcal{M}}$ and $\bar x$ of the flows $\mathcal{M}-(\psi(t_i)e_{n+1},t_i)$ and points $x_i-\psi(t_i)e_{n+1}$. Since for the flow $\overline{\mathcal{M}}$ the moving plane can reach $\mu$, applying Proposition  \ref{prop_smothness_reached} (smoothness)  we see that $(\bar x,0)$ is a smooth point for $\overline{\mathcal{M}}$. Observing also that as a consequence of Lemma \ref{distance gap} (distance gap) we in particular have
\[
\overline{K}_t^{\mu <}\subseteq \textrm{Int}(\overline{K}_t^{\mu -}),
\]  
%where we denote by $\overline{K}$ the domain bounded by $\overline{\mathcal{M}}$, by $\overline{K}_t$ its time slices, and we let
%\begin{align}
%\overline{K}_t^{\mu -}&=\overline{K}_t\cap \{ x_1 < \mu\},\\
%\overline{K}_t^{\mu +}&= \overline{K}_t\cap \{ x_1 > \mu\}.
%\end{align}
we can thus apply Lemma \ref{Hopf lemma for graphs} (smooth Hopf lemma) for $\overline{\mathcal{M}}$ at the spacetime point $(\bar x,0)$ to infer that
\[
|\langle \nu(\bar{x},0),e_1\rangle|\neq 0.
\]
However, since by the local regularity theorem \cite{White_regularity} the convergence near $(\bar{x},0)$ is smooth, this contradicts $\lim_{i\to \infty} |\langle \nu(x_i,t_i),e_1\rangle|=0$, and thus proves the lemma.
\end{proof}

\begin{theorem}[rotational symmetry]\label{thm_rot_symm}
$\mathcal{M}$ is rotationally symmetric, and smooth away from the $x_{n+1}$-axis.
\end{theorem}

\begin{proof}
By Proposition \ref{prop_smothness_reached} (smoothness) it is enough to show that the moving plane can reach $\mu=0$.
Consider the interval
\begin{equation}
I=\{ \mu \geq 0: \text{the moving plane can reach}\; {\mu}\}
\end{equation}
Note that  $I \neq \emptyset$ by Corollary \ref{cor_start_plane} (start plane). Let $\mu:=\inf I$, and observe that $\mu\in I$. Suppose towards a contradiction that $\mu>0$.

First, by Proposition \ref{Infinity Dirichlet} (no contact at infinity) we have
\begin{equation}
K_t^{\frac{\mu}{2} <}\cap (E_t^{\mu})^c \subseteq  K_t^{\frac{\mu}{2} -}
\end{equation}
for all $t\in \mathbb{R}$, where we recall that 
\[
E_t^{\mu}=\{x_{n+1} \leq \psi(t)+h_{ \mu /2}\}.
\]
Next, by Lemma \ref{Angle gap} (angle gap) and Proposition \ref{prop_smothness_reached} (smoothness) there exists a $\delta_1 \in (0,\min\{ \delta_0,\tfrac{\mu}{2}\})$ such that for $\delta\in (0,\delta_1)$ we have
\begin{equation}
K_t^{(\mu-\delta) <}\cap E_t^{\mu}\cap \{x_1 \geq \mu-2\delta_1\} \subseteq K_t^{(\mu-\delta)-}
\end{equation}
for all $t\in \mathbb{R}$.

Finally, combining the above with Lemma \ref{distance gap} (distance gap) we conclude that every $\delta\in (0,\min\{\delta_1,\alpha(\delta_1)\})$ we have
\begin{equation}
K_t^{(\mu-\delta) <}\subseteq K_t^{(\mu-\delta)-}.
\end{equation}
for all $t\in \mathbb{R}$. Hence, the moving plane can reach $\mu-\delta$; a contradiction. This proves the theorem.
\end{proof}

\bigskip

\section{Classification of ancient asymptotically cylindrical flows}

In this section, we conclude the proof of Theorem \ref{thm_classification_asympt_cyl} (classification of ancient asymptotically cylindrical flows).

\subsection{The noncompact case}\label{sec_classification_noncompact}

Let $\mathcal M$ be an ancient asymptotically cylindrical flow that is not a round shrinking cylinder. In this subsection we conclude the classification in the case where the plus mode is dominant. We start with the following lemma.

\begin{lemma}[regularity of symmetric shrinkers]\label{regular_self_shrinkers}
Let $\Sigma$ be an a priori potentially singular self-shrinker with $\mathrm{Ent}[\Sigma]\leq \mathrm{Ent}[S^{n-1}\times \mathbb{R}]$, which is rotationally symmetric with respect to the $x_{n+1}$-axis. Then $\Sigma$ is smooth. 
\end{lemma}

\begin{proof}
For $p\in \Sigma$, let $C$ be any tangent cone to $\Sigma$ at $p$. Any such $C$ must be a stationary cone with entropy at most $\mathrm{Ent}[S^{n-1}\times \mathbb{R}]<3/2$.  If $p$ does not lie on the $x_{n+1}$-axis, then $C$ splits off $(n-1)$ lines (corresponding to the $O(n)$ action). Thus, by Lemma \ref{smooth_cones} (stationary cones), $C$ is a hyperplane, and consequently $p$ is regular. If $p$ does lie on the $x_{n+1}$-axis, then it follows that $C$ is a rotationally symmetric stationary cone with entropy less than $3/2$. Hence, $C$ is again a hyperplane, and $p$ is regular. This proves the lemma.
\end{proof}

\begin{proposition}[regularity]\label{regular_connected}
If the plus mode is dominant, then $\mathcal{M}$ is smooth. More precisely, there exists a constant $c=c(\mathcal{M})>0$ such that the regularity scale satisfies $R(X)\geq c$ for every $X\in \mathcal{M}$. Furthermore, $M_t$ is connected for every $t$.
\end{proposition}

\begin{proof}
From Section \ref{sec_cap_size} we know that $\mathcal M$ has a cap of controlled size and opens up like a parabola. From Section \ref{sec_moving_planes} we know that $\mathcal M$ is rotationally symmetric, and regular away from the axis of symmetry, which we can take to be the $x_{n+1}$-axis.  Since $\mathcal{M}$ is rotational symmetric, it follows that any tangent flow $\hat{\mathcal{M}}_X$  to $\mathcal{M}$ at any point $X=(0,\ldots,0,x_{n+1},t)$ is rotationally symmetric. By Lemma \ref{regular_self_shrinkers} (regularity of symmetric shrinkers) and \cite[Thm. 2]{KM}, such a rotationally symmetric self-shrinking flow is either the round shrinking sphere $S^n$, the round shrinking cylinder $S^{n-1}\times\mathbb{R}$, the flat hyperplane $\mathbb{R}^n$, or a shrinking torus $S^{n-1}\times S^1$. The possibility of the cylinder is excluded by the equality case of Huisken's monotonicity formula. The possibility of the torus is excluded by the entropy bound thanks to \cite{CIMW}. Hence, by unit regularity, it follows that  $\mathcal{M}$ may only have spherical singularities. 

Suppose, for the sake of contradiction, that $\mathcal{M}$ has a spherical singularity at $X_0=(x_0,t_0)$. For $t_1<t_0$, denote by $\mathcal{N}^{t_1}$ the space time connected component of $\mathcal{M}\cap\{t_1\leq t\leq t_0\}$ that contains $X_0$. Since $\mathcal{M}$ is connected, and since $\mathcal{M}$ has a cap of controlled size, there exists some first $\bar{t}_1$ such that $\mathcal{N}^{\bar{t}_1}$ is non-compact. Moreover, $\bar{t}_1$ is a singular time of the flow, and $\mathcal{S}_{\bar{t}_1}(\mathcal{M})$ consists of finitely many spherical singularities, $\mathcal{S}_{\bar{t}_1}(\mathcal{M})=\{s_1,\ldots,s_k\}$. Thus, there exists some uniform $r>0$ such that $\mathcal{M}$ is $\eps$-spherical around each $(s_i,\bar{t}_1)$ at scale $r$. This clearly contradicts the definition of $\bar{t}_1$.  Thus, $\mathcal{M}$ had no spherical singularities to begin with. Hence, $M_t$ has no compact components and is connected for every $t$.

Finally, the uniform lower bound for the regularity scale follows from the above by a contradiction argument, similarly as in Section \ref{sec_moving_planes}. This finishes the proof of the proposition.
\end{proof}

\begin{theorem}\label{thm_class_noncompact}
If the plus mode is dominant, then $\mathcal M$ is the bowl soliton.
\end{theorem}

\begin{proof}
From Section \ref{sec_cap_size} we know that $\mathcal M$ has a cap of controlled size and opens up like a parabola.
Thanks to Theorem \ref{thm_rot_symm} (rotational symmetry) and Proposition \ref{regular_connected} (regularity) we know that $\mathcal{M}$ is rotationally symmetric with regularity scale bounded from below. We can now run the same argument as in \cite[proof of Thm. 7.1]{CHH} to show that $\mathcal{M}$ is a mean-convex, noncollapsed, translating soliton. Together with the convexity estimate \cite[Thm. 1.10]{HaslhoferKleiner_meanconvex}, and the asymptotic structure, we conclude that $\mathcal{M}$ is a strictly convex rotationally symmetric translating soliton. Hence, $\mathcal{M}$ is the bowl soliton as constructed in \cite{AltschulerWu,CSS}.
\end{proof}

\bigskip

\subsection{The compact case}\label{sec_classification_compact}

Let $\mathcal M$ be an ancient asymptotically cylindrical flow that is not a round shrinking cylinder. In this subsection we conclude the classification in the case where the neutral mode is dominant. 

\begin{theorem}\label{thm_class_neutral}
If the neutral mode is dominant, then $\mathcal M$ is an ancient oval.
\end{theorem}

\begin{proof}

From Section \ref{sec_fine_neutral} we know that all time slices of $\mathcal M=\{M_t \}_{t\in (-\infty,T_e(\mathcal{M})]}$ are compact. We may assume without loss of generality that $T_e(\mathcal{M})=0$ and that $(0,0)\in \mathcal{M}$. Since the blowdown for $t\to -\infty$ is a cylinder for $t\leq \mathcal{T}$ there is a central neck $Z_t$ of length $L_0\sqrt{-t}$. Denote by $D_t$ the connected component of $M_t$ that contains the central neck $Z_t$. Using Corollary \ref{thm_compact} we see that $D_t\setminus Z_t$ has two connected components, and that these components are contained in the upper and lower halfspace, respectively.

Consider the height of the tip functions
\begin{equation}
\psi_{+}(t):=\max_{p\in D_t} x_{n+1}(p),
\end{equation}
and
\begin{equation}
\psi_{-}(t):=\min_{p\in D_t} x_{n+1}(p).
\end{equation}
Since the blowdown for $t\to -\infty$ is a cylinder, we have
\begin{equation}\label{lim_size_2}
\lim_{t\to -\infty}\frac{ \psi_\pm(t)}{\sqrt{|t|}}=\pm\infty.
\end{equation}

We claim that there exists some $C<\infty$ such that for $t$ sufficiently negative, for every tip point $p_t^\pm\in D_{t}$, i.e. every $p_t^\pm\in D_{t}$  with $x_{n+1}(p_t^\pm)=\psi_{\pm}(t)$, we have  
\begin{equation}\label{Z_growth_2}
Z(p_t^\pm,t)\leq C\sqrt{|t|}.
\end{equation}

Indeed, if $C\geq 2^{N+1}$, where $N$ is from Theorem \ref{thm_finding_sim} (almost selfsimilarity), then $Z(p_t^\pm,t)> C\sqrt{|t|}$ implies that $\mathcal{M}$ is $\eps$-compact or $\eps$-separating at some scale $r\in [2\sqrt{|t|},2^{N+1}\sqrt{|t|}]$. However, in the $\eps$-compact case, item (i) of Proposition \ref{prop_useful_for45} implies that $D_t \subseteq B(0,2^{N+1}\sqrt{|t|}/\eps)$, contradicting \eqref{lim_size_2}, and in the $\eps$-seperating case, item (ii) of Proposition \ref{prop_useful_for45} implies that $T_e(\mathcal{M})>|t|$, contradicting our assumption that $T_e(\mathcal{M})=0$. This proves \eqref{Z_growth_2}.

Combining \eqref{lim_size_2} and \eqref{Z_growth_2} we see that for any choice of tip points $p_t^\pm$, we have
\begin{equation}\label{real_Z_growth}
\frac{Z(p_t^\pm,t)}{|\psi_{\pm}(t)|}\rightarrow 0.
\end{equation}

Let $p_{t_j}^\pm$ be a sequence of tip points at times $t_j\rightarrow -\infty$. Let $ \mathcal{M}^j_\pm$ be the sequence of flows that is obtained from $\mathcal{M}$ by shifting $(p_{t_j}^\pm,t_j)$ to the origin and parabolically rescaling by $Z(p_{t_j}^\pm,t_j)^{-1}$, and pass to a subsequential limit $\mathcal{M}^\infty_\pm$. Applying Theorem \ref{thm_finding_sim} (almost selfsimilarity) along the approximating sequence we see that the limit $\mathcal{M}^\infty_\pm$ must be an ancient asymptotically cylindrical flow, whose axis is in $x_{n+1}$-direction since we tacitly chose coordinates for $\mathcal{M}$ as in \eqref{eq_conv_axis}. Arguing as in the proof of Proposition \ref{fast_tip}, we see that $\mathcal{M}^{\infty}_\pm$ is not the cylinder. Thanks to \eqref{real_Z_growth} the limit $\mathcal{M}^\infty_\pm$ is noncompact. Hence, by Corollary \ref{thm_compact} and Theorem \ref{thm_class_noncompact} it is a translating bowl soliton. In particular, by the preservation of neck backwards in time (see Section \ref{sec_sim_back_in_time}), there is a unique tip point $p_t^\pm$ at time $t\leq \mathcal{T}$.\\

Fix $t_0\leq\mathcal{T}$. Denote by $D_{t_0}^\pm$ the connected component of $D_{t_0}\setminus Z_{t_0}$ that contains $p_{t_0}^\pm$. Let us focus on $D_{t_0}^+$ (the argument for $D_{t_0}^-$ is similar). Note that $D_{t_0}^+$  has the central neck $Z_{t_0}$ as a collar and also contains another neck $Z_{t_0}^{+}$ bounding a convex cap $C_{t_0}$ that is $\alpha$-noncollapsed, say for $\alpha=\tfrac{1}{100n}$. Let $N_{t_0}:=Z_{t_0}\cup D_{t_0}^+\setminus C_{t_0}$, and set
\begin{equation}
a=\inf_{p\in N_{t_0}} x_{n+1}(p),\qquad b=\sup_{p\in N_{t_0}} x_{n+1}(p).
\end{equation}
Similarly to the proof of Proposition \ref{regular_connected}, for each $t_1<t_0$ let $\mathcal{N}^{t_1}$ the space-time conneted component of $N_{t_0}$ in
\begin{equation}
\mathcal{M} \cap \{x_{n+1}\in [a,b]\}\cap \{t\in [t_1,t_0]\},
\end{equation}
and denote by $N^{t_1}_t$ the time $t$ slice of $\mathcal{N}^{t_1}$. By the preservation of necks backwards in time (see Section \ref{sec_sim_back_in_time}) the two necks $Z_{t_0}$ and $Z_{t_0}^+$ give rise to necks $Z_{t}$ and $Z_{t}^+$ in $N^{t_1}_t$ for all $t\leq t_0$. For $t_1\ll t_0$, the entire domain $N^{t_1}_{t_1}$ constitutes a single neck. Fix such a time $t_1$ and denote $N'_t:=N^{t_1}_t$ for all $t\in [t_1,t_0]$.\\
 
We next claim that if $t_{\ast} \leq t_0$ is such that $N'_t$ is smooth on $[t_1,t_{\ast})$, then 
$N'_t$ is mean-convex and noncollapsed for all $t\in [t_1,t_{\ast})$. To see this, first observe that using the planes $\{x_{n+1}=a\}$ and $\{x_{n+1}=b\}$ as barriers, we get that $N'_t$ has no other boundary components except the the two collars.
Therefore, the parabolic maximum principle with boundary implies that the mean curvature satisfies
\begin{equation}\label{rj_size}
H\geq c(t_1)>0.
\end{equation}
Moreover, the parabolic maximum principle for $|A|/H$, applied with two collar boundaries, yields
\begin{equation}\label{AHbound}
|A| \leq C H.
\end{equation}
Now, suppose towards a contradiction there are times $t_j\in [t_1,t_{\ast})$ and points $p_j\in N'_{t_j}$ at which the maximal interior or exterior tangent ball is of radius
\begin{equation}\label{miss_ball}
r_j\leq j^{-1}H(p_j)^{-1}.
\end{equation}
Scale the flow around $(p_j,t_j)$ by $r_{j}^{-1}$ and pass to a subsequential limit $\overline{\mathcal{M}}$. Since by \eqref{rj_size} we have $r_j\to 0$, and since the collar boundaries, as well as the neck at time $t_1$, are noncollapsed, we infer that 
$\overline{\mathcal{M}}$ is an ancient mean-convex flow, defined on the entire space, which further satisfies the entropy condition
\begin{equation}\label{ent_over_M}
\mathrm{Ent}[\overline{\mathcal{M}}] \leq \mathrm{Ent}[S^{n-1}\times \mathbb{R}].
\end{equation} 
We next observe that any tangent flow of $\overline{\mathcal{M}}$ at $(0,0)$ is contained in a halfspace. Indeed, in the case of interior noncollapsing, this is immediate from \eqref{miss_ball} and mean-convexity, while in the case of exterior noncollapsing it follows using in addition Theorem \ref{part_reg_thm} (partial regularity) and the classification mean-convex shrinkers with small singular set from \cite[Theorem 7.4]{Zhu}. Hence, by the local regularity theorem \cite{White_regularity}, the point $(0,0)$ is a regular point and thanks to \eqref{miss_ball} satisfies $H^{\overline{\mathcal{M}}}(0,0)=0$. By the strong maximum principle for the mean curvature together with the $|A|/H$-bound from \eqref{AHbound} this implies that in some small backwards parabolic ball centered at $(0,0)$ the flow $\overline{\mathcal{M}}$ agrees with a static plane. In fact, since by \eqref{ent_over_M} and Theorem \ref{part_reg_thm} (partial regularity) any space-time curve can be perturbed to avoid the singular set, it follows that the entire backwards part $\overline{\mathcal{M}}\cap \{t\leq 0 \}$ agrees with a static plane (note also that there cannot be any other connected components since the entropy is less than two). For $j$ large enough this contradicts our assumption that $r_j$ was maximal, and thus establishes noncollapsing.\\

Next, suppose, for the sake of contradiction that there exists some first  $t'\in [t_1,t_0]$ at which $N'_t$ is non-smooth. If $x'\in N'_{t'}$ is a singular point, then by \cite{HaslhoferKleiner_meanconvex} and the entropy assumption, the singularity at $(x',t')$ is spherical. This, however, contradicts $N'_t$ being an annulus for $t<t'$, and thus proves that $N'_t$ is smooth and mean-convex annulus for all $t\in [t_1,t_0]$. This also shows that $N_{t_0}=N'_{t_0}$.
Similarly, by the parabolic maximum principle for  $\tfrac{\lambda_1+\lambda_2}{H}$, applied with two collar boundaries, the annulus $N_t$ satisfies  $\tfrac{\lambda_1+\lambda_2}{H}\geq \tfrac{1}{4n}$ for all $t\in [t_1,t_0]$.
Now, adding the cap $C_{t_0}$, we get that $D_{t_0}^+$ is a smooth, mean-convex disk and satisfies $|A|/H\leq 100$ and $\frac{\lambda_1+\lambda_2}{H}\geq \frac{1}{1000n}$.The same argument applies to $D_{t_0}^-$. Thus, $D_{t_0}$ is  is a smooth, mean-convex sphere and satisfies $|A|/H\leq 100$ and $\frac{\lambda_1+\lambda_2}{H}\geq \frac{1}{1000n}$.

Since $t_0\leq \mathcal{T}$ was arbitrary, we conclude that $D_t$ is a smooth, mean-convex, uniformly two-convex sphere and satisfies $|A|/H\leq 100$ and $\frac{\lambda_1+\lambda_2}{H}\geq \frac{1}{1000n}$ for $t\leq\mathcal{T}$.  As above, the mean-convexity, the entropy bound and the $|A|/H$ bound imply that   $\{D_t\}_{t\leq\mathcal{T}}$ is also $\alpha$-noncollapsed for some $\alpha>0$. Moreover, as $\mathcal{M}$ is connected in space-time, it follows that $M_t=D_t$ for all $t\leq\mathcal{T}$.  Therefore, by Angenent-Daskalopoulos-Sesum \cite{ADS2}, we conclude that $\mathcal{M}$ is an ancient oval.
\end{proof}

\section{Applications}\label{sec_app}

\subsection{Mean-convex neighborhoods}\label{sec_proof_mean_conv_nbd}
The purpose of this subsection is to prove Theorem \ref{thm_mean_convex_nbd_intro} (mean-convex neighborhoods).\\

We will assume that we have an inwards neck singularity at $(x,T)$ (the argument for outward neck singularities is the same). We recall that given a closed embedded surface $M\subset\mathbb{R}^{n+1}$, we denote by $\{M_t\}_{t\geq 0}$ its outer flow. Observe that if $x\in M_t$ is a regular point, then there exists a $\delta>0$ such that $M_t\cap B(x,\delta)$ splits $B(x,\delta)$ into two connected components: one in $\mathrm{Int}(K_t)$ and the other in $\mathbb{R}^{n+1}\setminus K_t$.\\

Theorem \ref{thm_mean_convex_nbd_intro} (mean-convex neighborhoods for neck singularities) will follow from a similar argument as in the proof of the mean-convex neighborhood conjecture for two-dimensional mean curvature flow in \cite[Sec. 8]{CHH}, after the following auxiliary proposition is established.

\begin{proposition}[{auxiliary canonical neighborhood, c.f. \cite[Prop. 8.2]{CHH}}]\label{basic_reg_1}
Under the assumption of an inwards neck singularity in Theorem \ref{thm_mean_convex_nbd_intro}, there exists a constant $\delta=\delta(X_0)>0$ and a unit-regular, integral Brakke flow $\mathcal{M}=\{\mu_t\}_{t\geq t_0-\delta}$ whose support is $\{M_t\}_{t\geq t_0-\delta}$ such that
\begin{enumerate}
\item[(i)] The tangent flow to $\mathcal M$ at $X_0$ is a multiplicity one cylinder $\{S^{n-1}(\!\!\sqrt{2(n-1)|t|})\times\mathbb{R}\}_{t<0}$.  
\item[(ii)]  We have $H\neq 0$ at every regular point of $\mathcal M$ in $\bar{B}(x_0,2\delta)\times [t_0-\delta, t_0+\delta]$.
\item[(iii)]  The flow $\mathcal M$ has only multiplicity one neck and spherical singularities in $\bar{B}(x_0,2\delta)\times [t_0-\delta,t_0+\delta]$. 
\item[(iv)]  The parabolic Hausdorff dimension of the set of singular points in $\bar{B}(x_0,2\delta)\times [t_0-\delta,t_0+\delta]$ is at most one. In particular, $\mathcal M$ is smooth in $\bar{B}(x_0,2\delta)$ for a.e. $t\in [t_0-\delta,t_0+\delta]$ and is smooth outside of a set of Hausdorff dimension $1$ for every $t\in [t_0-\delta,t_0+\delta]$.
\item[(v)] If $X_i\rightarrow X_0$ are regular points, then any subsequential limit of $\mathcal{M}_{X_i,R(X_i)}$ is either a round shrinking sphere, a round shrinking cylinder, a translating bowl, or an ancient oval.
\item[(vi)] There exist $A<\infty$ and $c>0$ such that if $X=(x,t)$ is a point of $\mathcal M$ in $\bar{B}(x_0,2\delta)\times [t_0-\delta, t_0+\delta]$ with $R(X) \leq c$, then $\mathcal{M}$ is smooth and connected in $P(X,AR(X))$ and there is a point $X'=(x',t')\in \mathcal{M}\cap P(X,AR(x))$ with $R(X') \geq 2R(X)$ and $|x'-x_0|\leq \max\{|x-x_0|-cR(X'),\delta/2\}$.
\end{enumerate}
\end{proposition}
\begin{proof}
The argument is related to the one in \cite[proof of Prop. 8.2]{CHH}, but the proof of (ii) and (iii) requires additional ideas.\\

First, it follows from \cite[Thm. B3]{HershkovitsWhite} that there exists an outer Brakke flow starting from $M$, whose support is $\{M_t\}_{t\geq 0}$. In particular, this, together with monotonicity, implies that
 $\mathrm{Ent}[\mathcal{H}^n\llcorner M_t]$ is uniformly bounded. Thus, the neck singularity assumption implies that there exists $t_{\ast}<t_0$ such that  
\begin{equation}\label{low_dens}
\frac{1}{(4\pi(t_0-t_{\ast}))^{n/2}}\int \exp\Big(\frac{-|x-x_0|^2}{4(t_0-t_{\ast})}\Big)d\mathcal{H}^n\llcorner M_{t_{\ast}}<2. 
\end{equation}
Therefore, applying \cite[Thm. B3]{HershkovitsWhite} once more, this time with the initial time $t_{\ast}$, we get a unit-regular  integral Brakke flow $\mathcal{M}:=\{\mu_t\}_{t\geq t_{\ast}}$ whose support is $\{M_t\}_{t\geq t_\ast}$. Together with the neck singularity assumption  and \eqref{low_dens} it follows that $\mathcal M$ has the multiplicity one cylinder $\{S^{n-1}(\!\!\sqrt{2(n-1)|t|})\times\mathbb{R}\}_{t<0}$ as tangent flow at $(x_0,t_0)$. This proves (i). \\

To prove (ii), suppose towards a contradiction that there exist regular points $X_i\rightarrow X_0$ with $H(X_i)=0$. Since $X_i\rightarrow X_0$, and since $X_0$ is a cylindrical singularity, we see that $Z(X_i)<\infty$ for $i$ large enough, and that there exists a sequence $\{r_i\}_{i=i_0}^{\infty}$ with $r_i/Z(X_i)\rightarrow \infty$ such that $\mathcal{M}$ remains $\eps$-cylindrical around $X_i$ at all scales between $Z(X_i)$ and $r_i$. Let $\mathcal{M}^i$ be the sequence of flows which is obtained from $\mathcal{M}$ by shifting $X_i$ to the origin and parabolically rescaling by $Z(X_i)^{-1}$. Remembering also Theorem \ref{thm_finding_sim} (almost selfsimilarity) it follows that $\mathcal{M}^i$ subconverges to an ancient asymptotically cylindrical flow $\mathcal{M}^{\infty}$ with $Z(0,0)=1$.

By Theorem \ref{thm_classification_asympt_cyl} (classification of ancient asymptotically cylindrical flows) the limit $\mathcal{M}^{\infty}$ is either a round shrinking cylinder, a translating bowl, or an ancient oval. 
If $\mathcal{M}^{\infty}$ is the cylinder, then $0$ cannot be its time of extinction, since $Z(0,0)=1$. Therefore, if $\mathcal{M}^{\infty}$ is either the cylinder or the bowl, it follows that $(0,0)$ is a regular point of $\mathcal{M}^{\infty}$. Hence, by the local regularity theorem, we infer that $H(0,0)=0$; this is a contradiction to the fact that the bowl and the cylinder are mean-convex.
If, on the other hand, $\mathcal{M}^{\infty}$ is an ancient oval, then it follows that for $i$ large enough $M^i_{-1}$ is compact, smooth and convex. Since mean-convexity is preserved under mean curvature flow, this contradicts $H(X_i)=0$. This proves (ii).\\

To prove (iii), first  observe that  there exists a $C<\infty$ and a two sided neighborhood of $X_0$ such that
\begin{equation}\label{alternative}
\;\;\;\; Z(X) \leq CR(X) \;\;\;\;\;\;\;\;\;\;\;\textrm{or }\;\;\;\;\;\;\;\;\;\;\; Z(X)\leq CL(X),
\end{equation} 
for all points $X$ (not necessarily regular) in it. Here, $L(X)$ denotes the \textit{last} $\eps$-spherical scale, i.e. the supremum over all $r$ such that $\mathcal{M}_{X,r}$ is $\eps$-close in $C^{\lfloor1/\varepsilon \rfloor}$ in $B(0,1/\varepsilon)\times [-2,-1]$ to the evolution of a round sphere with radius $\sqrt{-2nt}$ and center at the origin (by convention $L(X)=-\infty$ if there is no such $r$).

 Indeed, suppose there is $X_i\rightarrow X_0$ for which neither alternative in \eqref{alternative} is satisfied with $C=i$. Letting $\mathcal{M}^i$ be the flows obtained by shifting $X_i$ to the origin and parabolically rescaling by $Z(X_i)^{-1}$, remembering Theorem \ref{thm_finding_sim} (almost selfsimilarity) we see that $\mathcal{M}^i$ subconverge to an ancient asymptotically cylindrical flow $\mathcal{M}^{\infty}$ with $Z(0,0)=1$. By Theorem \ref{thm_classification_asympt_cyl} (classification of ancient asymptotically cylindrical flows) the limit $\mathcal{M}^{\infty}$ is either a round shrinking cylinder, a translating bowl, or an ancient oval. As \eqref{alternative} holds for all such flows, this leads to a contradiction.

Now, suppose that $X$ is a singular point in that neighborhood. Then, since $R(X)=0$, from \eqref{alternative} we infer that either $Z(X)=0$ or $L(X)>0$. In the first case, it follows that $\mathcal{M}$ has a neck singularity at $X$. In the second case, it follows that $\mathcal{M}$ has a spherical singularity at $X$. This proves (iii).\\

Assertion (iv) follows from (iii) and standard stratification, see e.g. \cite{White_stratification}.\\

Suppose $X_i\rightarrow X_0$ are regular points and $\hat{\mathcal{M}}$ is a limit of $\mathcal{M}_{X_i,R(X_i)}$. If there is a $C<\infty$ such that $Z(X_i)\leq CR(X_i)$ for all $i$, then remembering again Theorem \ref{thm_finding_sim} (almost selfsimilarity) we see that $\hat{\mathcal{M}}$ is an ancient asymptotically cylindrical flow, and hence, by Theorem \ref{thm_classification_asympt_cyl} (classification of ancient asymptotically cylindrical flows), a round shrinking cylinder, a translating bowl or an ancient oval. If there is no such $C$, then $\hat{\mathcal{M}}$ is a (nontrivial) blowup limit of an ancient asymptotically cylindrical flow, which only adds the round shrinking sphere to the list of possibilities. This proves (v).\\

Finally, observe by inspection of the three asymptotically cylindrical flows and the sphere, that there exist  $A,C<\infty$ such that if $X=(x,t)$ is a point on such a flow $\mathcal{M}^{\infty}$, then  for every unit vector $v$, there exists a point in $X'\in P(X,AR(X))\cap \{y:(y-x)\cdot v\geq R(X)\}\cap \mathcal{M}^{\infty}$ with 
\begin{equation}
4R(X) \leq R(X')\leq CR(X).
\end{equation}
Using this, assertion (vi) follows from (v).
\end{proof}

\begin{proof}[Proof of Theorem \ref{thm_mean_convex_nbd_intro}]
Using Proposition \ref{basic_reg_1} (auxiliary canonical neighborhood), the proof of Theorem \ref{thm_mean_convex_nbd_intro}  (mean-convex neighborhoods for neck singularities) now follows from the same argument as in \cite[Sec. 8]{CHH}, by first showing that $H$ does not change sign on regular points near $X_0$ (as $H\neq 0$ in such a neighborhood, this is a statement about connectedness within this neighborhood), and then showing that $K_t$ moves in one direction even if singular, similarly as in \cite[Thm. 3.5]{HershkovitsWhite}.
\end{proof}

\subsection{Uniqueness of weak flows}

In this subsection, we prove that mean curvature flow through neck singularities and spherical singularities is unique.

\begin{proof}[{Proof of Theorem \ref{thm_nonfattening_intro} (nonfattening)}]
If $T\leq T_{\mathrm{disc}}$ then, by definition, the outer flow $\{M_t\}_{t\in [0,T]}$ agrees with the level set flow $\{F_t(M)\}_{t\in [0,T]}$ and the inner flow $\{M_t'\}_{t\in [0,T]}$. In particular, if $(x_0,T)$ is a neck singularity (respectively spherical singularity) of $\{M_t\}_{t\in [0,T]}$, then either $\mathcal{K}_{X,\lambda}$ or $\mathcal{K}_{X,\lambda}'$ converges for $\lambda\to 0$ locally smoothly with multiplicity one to a round shrinking solid cylinder $\{\bar{B}^n(\!\!\sqrt{2(n-1)|t|})\times \mathbb{R} \}_{t<0}$ (respectively a round shrinking solid ball $\{\bar{B}^{n+1}(\!\!\sqrt{2n|t|})\}_{t<0}$).

The result now follows from combining the main theorem of Hershkovits-White \cite{HershkovitsWhite}, which establishes that $T< T_{\mathrm{disc}}$ assuming the existence of mean-convex neighborhoods a priori, and Theorem \ref{thm_mean_convex_nbd_intro} (mean-convex neighborhoods), which proves the existence of such mean-convex neighborhoods.
\end{proof}

\bigskip

\bibliography{asympt_cyl_revised}

\bibliographystyle{alpha}

\vspace{10mm}

{\sc Kyeongsu Choi, School of Mathematics, Korea Institute for Advanced Study, 85 Hoegiro, Dongdaemun-gu, Seoul, 02455, South Korea}\\

{\sc Robert Haslhofer, Department of Mathematics, University of Toronto,  40 St George Street, Toronto, ON M5S 2E4, Canada}\\

{\sc Or Hershkovits, Institute of Mathematics, Hebrew University, Givat Ram, Jerusalem, 91904, Israel}\\

{\sc Brian White, Department of Mathematics, Stanford University, 450 Serra Mall, Stanford, CA 94305, USA}\\

\end{document}